\documentclass[11pt]{article}

\usepackage[a4paper,top=3cm,bottom=3cm,left=3cm,right=3cm,marginparwidth=1.75cm]{geometry}
\usepackage{amsmath,amssymb,amsthm}
\usepackage{booktabs,cleveref} 
\usepackage[ruled]{algorithm2e} 
\usepackage{natbib}
\usepackage{algorithmic}
\usepackage{graphicx}
\usepackage{subcaption}
\usepackage{caption}
\usepackage{enumerate}
\usepackage{float}
\usepackage{authblk}
\usepackage{lipsum}

\newtheorem{theorem}{Theorem}
\newtheorem{proposition}{Proposition}
\newtheorem{lemma}{Lemma}
\newtheorem{corollary}{Corollary}
\newtheorem{definition}{Definition}
\newtheorem{assumption}{Assumption}
\newtheorem{remark}{Remark}
\usepackage{xcolor}

\newcommand{\off}{{{\scriptscriptstyle\mathrm{off}}}}
\newcommand{\mla}{{{\scriptscriptstyle\mathrm{MLA}}}}

\newcommand{\all}{{{\scriptscriptstyle\mathrm{all}}}}
\newcommand{\Fkt}{\mathcal{F}_{k,n}}
\newcommand{\Zkt}{Z_{k,n}}
\newcommand{\mukt}{\hat{\mu}_{k,n}^\mla}
\newcommand{\mukntk}{\hat{\mu}_{k,n_{k,t}}^\mla}

\newcommand{\E}{\mathbb{E}}
\renewcommand{\P}{\mathbb{P}}
\newcommand{\I}{\mathbb{I}}
\newcommand{\N}{\mathcal{N}}
\newcommand{\var}{\mathrm{Var}}

\newcommand{\cov}{\ensuremath{\operatorname{Cov}}}

\newcommand{\correps}{\hat c_{k,n}}

\newcommand{\lb}{\left(}
\newcommand{\rb}{\right)}
\newcommand{\reg}{\operatorname{Reg_T}}
\renewcommand{\ln}{\log}

\newcommand\blfootnote[1]{%
  \begingroup
  \renewcommand\thefootnote{}\footnote{#1}%
  \addtocounter{footnote}{-1}%
  \endgroup
}
\title{Multi-Armed Bandits With Machine Learning-Generated Surrogate Rewards}

\author[1]{Wenlong Ji}
\author[2]{Yihan Pan}
\author[3]{Ruihao Zhu}
\author[1]{Lihua Lei}

\affil[1]{Stanford University}
\affil[2]{Northwestern University}
\affil[3]{Cornell University}
\begin{document}

\maketitle
\blfootnote{We thank Emma Brunskill, Emmanuel Candès, Edward Malthouse, Art Owen, Jann Spiess, and Stefan Wager for their helpful comments. L.L. and W.J. are grateful for the support of the National Science Foundation grant DMS-2338464.}
\begin{abstract}
Multi-armed bandit (MAB) is a widely adopted framework for sequential decision-making under uncertainty. Traditional bandit algorithms rely solely on online data, which tends to be scarce as it must be gathered during the online phase when the arms are actively pulled. However, in many practical settings, rich auxiliary data, such as covariates of past users, is available prior to deploying any arms. We introduce a new setting for MAB where pre-trained machine learning (ML) models are applied to convert side information and historical data into \emph{surrogate rewards}. A prominent challenge of this setting is that the surrogate rewards may exhibit substantial bias, as true reward data is typically unavailable in the offline phase, forcing ML predictions to heavily rely on extrapolation. To address the issue, we propose the Machine Learning-Assisted Upper Confidence Bound (MLA-UCB) algorithm, which can be applied to any reward prediction model and any form of auxiliary data. When the predicted and true rewards are jointly Gaussian, it provably improves the cumulative regret, even in cases where the mean surrogate reward completely misaligns with the true mean rewards, and achieves the asymptotic optimality among a broad class of policies. Notably, our method requires no prior knowledge of the covariance matrix between true and surrogate rewards.
We further extend the method to a batched reward MAB problem, where each arm pull yields a batch of observations and rewards may be non-Gaussian, and we derive computable confidence bounds and regret guarantees that improve upon classical UCB algorithms. Finally, extensive simulations with both Gaussian and ML-generated surrogates, together with real-world studies on language model selection and video recommendation, demonstrate consistent and often substantial regret reductions with moderate offline surrogate sample sizes and correlations.
\end{abstract}

\section{Introduction}\label{sec:intro}
The multi-armed bandit (MAB) framework is a widely adopted framework for sequential and interactive decision-making under incomplete information. In the MAB setting, a decision maker interacts with an unknown environment by sequentially selecting from a pre-specified set of actions (a.k.a. pulling arms). Each selected action yields a random reward from an action-dependent distribution that is unknown to the decision maker. A common objective is to choose actions to maximize cumulative rewards over many rounds or to minimize cumulative regret, defined as the loss incurred from not always selecting the ex-post optimal action. For example, a pharmaceutical company could apply an MAB framework to more efficiently identify the most effective drugs and dosages in clinical trials \cite{villar2015multi, aziz2021multi}, while minimizing risk to participants. Similarly, online news vendors deploy MAB algorithms to continuously refine content recommendations and enhance user engagement in real time \citep{li2010acontextual,coenen2019nyt,su2024exploration}. 

At the heart of the MAB problem lies the exploration-exploitation trade-off -- the decision maker must explore different actions sufficiently to discover the most rewarding one, while also feeling the pull to exploit the current best arm to maximize immediate gains. Many algorithms have been developed to address the trade-off, typified by the upper confidence bound (UCB) algorithm \citep{auer2002finite} and the Thompson sampling \citep{agrawal2012analysis,russo2018tutorial}, both of which provably achieve favorable regret bounds in the long run (e.g., \cite{lattimore2020bandit}). 

Despite the appealing theoretical guarantees, most existing algorithms operate solely in the online phase where the decision maker is actively pulling arms, without taking advantage of offline/historical data that is prevalent in practice. Although rewards are not directly observable in the offline phase before an arm is pulled, other variables may be informative about the reward distributions. To extract information, we can leverage machine learning algorithms, including pre-trained AI models like large language models (LLMs), to convert these offline variables into \emph{surrogate rewards}. If ML-generated surrogate rewards are correlated with unrealized true rewards through the input to the ML model, they have the potential to reduce regret even in the early stage by increasing the effective number of pulls. For example, in a clinical trial where the primary endpoint takes time to measure, such as five-year survival, the researcher can predict the long-term outcome using patient data from other trials and use the ML predictions as the surrogate rewards. Similarly, an online retailer can feed consumer profiles into LLMs to simulate purchasing behavior under the current marketing strategy or recommendation system, thereby obtaining surrogate rewards. 

As ML-generated surrogate rewards are much less expensive to acquire, it is tempting to use them in place of the costly true rewards to guide arm ranking. For example, the decision maker can pick -- or at least prioritize --- the best arms based on the mean surrogate rewards for each arm computed solely on offline data. However, this approach is ineffective unless the mean surrogate rewards preserve the ranking of the true mean rewards, an assumption that is unrealistic in most applications. In fact, most MAB applications aim to test new treatments and interventions that do not appear in historical data, and thus ML-generated surrogate rewards must rely on extrapolation to unseen environments. 
To provide context, we describe how to apply ML algorithms to generate surrogate rewards in two scenarios.

\begin{itemize}
\item (Surrogate rewards for parametrizable arms) When the arms correspond to new values of tuning parameters of a pricing or recommendation algorithm, and multiple values have been deployed in the past, a reward model can be fit jointly on user features and the parameter values using offline data. Such models enable extrapolation to generate arm-specific surrogate rewards for each unseen arm in the online phase. Nonetheless, the new values are often distant from the previous ones or even lie outside their convex hull. In such cases, the mean surrogate rewards may substantially misalign with the true mean rewards. 
\item (One-size-fits-all surrogate rewards) When the arms cannot be naturally parametrized and no similar arm has been deployed before, a reward model can be trained solely on user features under the status quo treatment. Such models produce the same surrogate rewards for all arms in the online phase. By definition, it is unable to inform the ranking of arms. 
\end{itemize}

In either of the above scenarios, the ML-generated surrogate rewards are inaccurate and untrustworthy. On the other hand, if the true rewards are not statistically independent of input to the ML models (e.g., user features) that generate surrogate rewards, they are typically correlated with the surrogate rewards. Motivated by these practical considerations, we focus on the setting where the surrogate rewards are correlated with the true rewards, but their means may be arbitrarily misaligned. Importantly, we require the surrogate rewards to be computable for all experimental units -- both offline and online -- allowing joint observation of the surrogate and true rewards for online units. This joint observability is essential for debiasing the surrogate rewards and leveraging their correlation with the true rewards. We argue that this is not a restrictive assumption, as it clearly holds in both aforementioned scenarios.

\subsection{Main contributions}

We begin by introducing the setting of MAB with ML-generated surrogate rewards, which is relevant to a wide range of applications. In this new setting, we propose the Machine Learning-Assisted Upper Confidence Bound (MLA-UCB) algorithm that combines the online true rewards with ML-generated surrogate rewards for both offline and online units. MLA-UCB integrates the prediction-powered inference (PPI) approach \citep{angelopoulos2023prediction,angelopoulos2023ppi++} to improve estimation precision of mean rewards for each arm by leveraging the surrogate rewards. Algorithmically, unlike most UCB algorithms, our MLA-UCB algorithm does not require the knowledge of the variance, or the sub-Gaussian parameter, of the true or surrogate rewards. Instead, it constructs upper confidence bounds by inverting a one-sample t-test with a novel variance estimate and a round-dependent significance level.

Under the assumption that the true and surrogate rewards are bivariate Gaussian,  with an arbitrary unknown mean and covariance matrix, we derive a theoretical upper bound on the cumulative regret. The bound is never worse than the regret lower bound for MAB problems without surrogate rewards and strictly improves upon it when the correlation between surrogate and true rewards is non-zero for any suboptimal arm. In particular, the surrogate rewards are allowed to have arbitrarily large bias because the regret reduction is achieved through variance reduction instead of direct data aggregation. Even when the correlations are all zero, owing to the use of t-tests and a novel exact distribution analysis (Proposition \ref{prop: pp mean distribution}), our regret bound strictly outperforms the best available UCB algorithm \cite{cowan2018normal} at the second order in MAB problems with unknown reward variance and no surrogate rewards. Moreover, we derive a lower bound for MAB problems with surrogate reward, which confirms the asymptotic optimality of our MLA-UCB algorithm with mild requirements on offline sample size. 

We further extend our framework beyond the Gaussian model in a batched reward setting, where a decision is deployed to a batch of individuals in each round. In this setting, even when the individual-level rewards are non-Gaussian, batching can induce approximate normality and enable a valid and computable confidence bound with mild assumptions on the distribution. We develop a batched variant of the MLA-UCB algorithm that uses a similar PPI-based mean estimation but constructs the confidence bound via self-normalized moderate deviation control \cite{petrov2012sums, jing2003self}. We prove that this variant of the MLA-UCB algorithm can still provably reduce the regret compared with the UCB algorithm with no surrogate rewards, demonstrating the general applicability of integrating surrogate rewards into bandit problems. 

Finally, we validate the theoretical insights and practical impact of MLA-UCB through a suite of numerical studies and real-world case studies. In controlled simulations under the Gaussian model, we demonstrate the sizable efficiency gain of our MLA-UCB algorithm over the classical UCB algorithm that does not leverage surrogate rewards, and observe a sharper reduction with larger offline samples and stronger correlation. We further test the robustness of our method using ML-generated, non-Gaussian surrogates produced from various nonlinear models, where the empirical gains remain pronounced. Beyond simulations, we demonstrate the workflow of using surrogate rewards in two real-world applications: language model selection, where open-source LLMs assist businesses in choosing proprietary models to gain better accuracy, and video recommendation, where ML predictions based on user features help to maximize user engagement on a video-sharing platform. In both cases, MLA-UCB consistently improves over UCB and benefits from more offline data. 

\subsection{Related works}

The use of historical data has been explored in various works in the bandit setting.  \cite{shivaswamy2012multi} and \cite{banerjee2022artificial} study a setting where the decision maker has access to historical rewards drawn from the same distribution as the true online rewards. However, this setting assumes that all arms are available prior to the start of the algorithm -- a condition that is rarely met in practice, as discussed above. \cite{cheung2024leveraging} considers a more general setting where rewards for each arm are available in the offline phase but their distributions may differ from the online rewards, though both share identical and known sub-Gaussianity parameters. Unlike our ML-generated surrogate rewards, they do not assume the observability of biased offline rewards in the online phase so they cannot leverage the correlation structure. Instead, they assume a known upper bound on the gap between the mean offline and online rewards to limit the adverse influence of the bias. Their UCB algorithm fails to improve upon the classical UCB algorithm without offline data if the bias bound is too conservative. Furthermore, it is unclear whether the regret remains vanishing if the bias bounds are invalid. In contrast, our MLA-UCB algorithm delivers improvement regardless of the magnitude of the bias and without the knowledge of a bound on the bias. This highlights the benefit of jointly observing surrogate and true rewards during the online phase. \cite{zhang2025contextual} further extends the settings of \cite{cheung2024leveraging} to the contextual bandit settings.

In a different vein, \cite{verma2021stochastic} studies a setting where no offline information is observed but a set of control variates is available alongside the true reward in the online phase. Specifically, the control variates in \cite{verma2021stochastic} can be viewed as surrogate rewards with a known mean. However, their regret bound has a non-standard form that cannot be directly compared to the existing bounds. In particular, unlike our MLA-UCB algorithm, it remains unclear whether their algorithm always outperforms the classical UCB algorithm. More importantly, it is impractical to know the exact mean of control variates, unless there is infinite offline data that has the same distribution as the online data to compute the mean precisely -- in fact, their algorithm is indeed similar to MLA-UCB with infinite surrogate rewards for each arm. However, it is often unrealistic to assume that historical data have no distribution shift over such a long horizon -- units further back in history typically differ from the online participants. In contrast to \cite{verma2021stochastic}, our MLA-UCB algorithm requires only a poly-logarithmic number of offline surrogate reward observations per arm that follow the same distribution as their online counterparts.

 More broadly, surrogate variables have proven effective in various areas, including surrogate outcome models \cite{pepe1992inference,robins1994estimation, chen2000unified}, causal inference \citep{scharfstein1999adjusting,robins1994estimation,glynn2010introduction, athey2019surrogate}, econometrics \citep{chen2005measurement, chen2008semiparametric}, semi-supervised inference \citep{zhang2019semi, azriel2022semi}, and control variates \citep{lavenberg1981perspective, nelson1990control}. In particular, our work builds off the recent development of PPI \cite{angelopoulos2023prediction, angelopoulos2023ppi++, zrnic2024active, gan2024prediction, gronsbell2024another}, which essentially use ML predictions as surrogates \citep{ji2025predictions}. All of these studies focus on asymptotic inference in a static, cross-sectional setting. Our MLA-UCB algorithm applies the idea of ML-generated surrogates to a dynamic setting with bandit feedback. Due to this complication, we need to make stronger distributional assumptions (i.e., joint Gaussianity) on the true and surrogate rewards and derive the finite sample distribution of the estimator for the purpose of regret analysis.

\section{Preliminaries}

\subsection{Standard MAB setting and UCB algorithms}
In an MAB problem, there are $K$ arms (each corresponds to an action), indexed by $k=1, \cdots, K$. We denote by $n_{k, t}$ the total number of pulls for arm $k$ right before any arm is pulled at round $t$. Pulling arm \( k \) at round \( t \) yields a random reward \( R_{k,n_{k, t}+1} \), which is assumed to be drawn independently from some distribution $\mathcal{P}_{R_k}$ with mean $\mu_k$. When no confusion can arise, we may use the notation $R_t$ to represent the reward observed in round $t$ without specifying the arm being pulled. For each round, a policy $\pi$ specifies a decision $A_t\in [K]$ based on the history $(A_1, R_1, A_2, R_2,\cdots, A_{t-1}, R_{t-1})$. The objective is to maximize the cumulative reward or minimize the cumulative regret over a total number of round \( T \), defined as:
    \[
    \operatorname{Reg}_T = \sum_{t=1}^T \left( \mu^\star - \mu_{A_t} \right),
    \]
where \( \mu^\star = \max_{k} \mu_k \) is the mean reward of the optimal arm. In particular, we assume that the distribution of each arm is fixed ahead, and there are no ties between arms, meaning that $\mu_i \neq \mu_j$ if $i\neq j$. We also define $k^\star = \arg\max_{k} \mu_k$ as the optimal arm and $\Delta_k=\mu^\star-\mu_k$ as the sub-optimality gap of each arm $k$.

In this paper, we primarily focus on the Gaussian bandit setting, where the distribution $\mathcal{P}_{R_k}$ is Gaussian. Typically, the variance $\sigma_k^2$ is assumed to be known and the classical UCB algorithm \cite{lai1985asymptotically, auer2002finite} can be summarized as follows:
\begin{enumerate}
    \item For $t = 1$ to $K$, pull each arm once.
    \item For $t = K+1$ to $T$, 
    \begin{itemize}
        \item Compute the sample mean $\hat{\mu}_{k, n_{k,t}} = \frac{1}{n_{k,t}}\sum_{s=1}^{n_{k,t}}R_{k,s}$ for each arm.
    \item Pull the arm $A_t = \arg\max_{k}\left\{ \hat{\mu}_{k, n_{k,t}} + \sigma_k\sqrt{\frac{2 \ln t}{n_{k,t}}}\right\}$.
    \end{itemize}
\end{enumerate}

The UCB algorithm balances \textit{exploration} and \textit{exploitation} by iteratively pulling the arm with either a high empirical mean reward or high uncertainty due to large variance resulting from limited samples. 

In practice, the variance of the reward is rarely known a priori. A more realistic approach is to estimate the variance $\sigma_k^2$ with online rewards and adjust the upper confidence bounds by accounting for the uncertainty of estimated variance. When $\mathcal{P}_{R_k}$ is a Gaussian distribution with unknown variance $\sigma_k^2$, the asymptotically optimal algorithm for this setting is proposed in \cite{cowan2018normal}:
\begin{enumerate}
    \item For $t = 1$ to $3K$, pull each arm three times.
    \item For $t = 3K+1$ to $T$, 
    \begin{itemize}
        \item Compute the sample mean $\hat{\mu}_{k, n_{k,t}} = \frac{1}{n_{k,t}}\sum_{s=1}^{n_{k,t}}R_{k,s}$ for each arm. 
   \item Compute the sample variance $\hat\sigma_{R,k,n_{k,t}}^2 = \frac{1}{n_{k,t}-1} \sum_{s = 1}^{n_{k,t}}(R_{k,s} - \hat \mu_{k,n_{k,t}})^2$ for each arm.
    \item Pull the arm $A_t = \arg\max_{k}\left\{ \hat{\mu}_{k, n_{k,t}} + \hat\sigma_{R,k,n_{k,t}}\sqrt{t^{\frac{2}{n_{k,t}-2}}-1}.\right\}$.
    \end{itemize}
\end{enumerate}

As we will show later in Figure \ref{fig: quantile}, the multiplier $\sqrt{ t^{\frac{2}{n_{k,t}-2}}-1}$ is a conservative upper bound on a certain quantile of a t-distribution. We do not present the algorithm in \cite{verma2021stochastic}, which uses the t-distribution directly, since their regret bound is suboptimal and not quite comparable with the existing ones. For the above algorithm, \cite{cowan2018normal} derive the following regret bound for any $T\geq 3K$:
\begin{equation}
\label{eqn: UCB regret}
    \E[\reg] \leq \sum_{k\neq k^\star}\frac{2\Delta_k\ln T}{\ln\lb 1+\frac{\Delta_k^2}{\sigma_{k}^2}\rb} + O((\ln T)^{3/4}\ln\ln T).
\end{equation}
In particular, the leading term of this regret bound matches with the lower bound established in \cite{burnetas1996optimal}, implying that the algorithm is asymptotically optimal for normal bandits with unknown variance. In this paper, we only consider the general case of multi-armed bandits with Gaussian rewards and unknown variance. Therefore, we set this optimal algorithm as the baseline to compare with, and refer to it as the \emph{classical UCB} algorithm thereafter.

\subsection{MAB with surrogate rewards}\label{sec: model}

We formally introduce the setting of MAB with surrogate rewards. In the online phase, if the arm $k$ is pulled for the $s$-th time, the decision maker can observe a surrogate reward $\hat{R}_{k, s}$ alongside a true reward $R_{k, s}$ as in the standard MAB setting. In the offline phase, the decision maker has access to a static pool of surrogate rewards $\{\hat{R}_{k,s}^\off\}_{s=1}^{N_k}$ for arm $k$. We assume that $\hat{R}_{k,1}^\off, \ldots, \hat{R}_{k, N_k}^\off, \hat{R}_{k, 1}, \hat{R}_{k, 2}, \ldots $ are independent and identically distributed (i.i.d.). As discussed in Section \ref{sec:intro}, the surrogate rewards can be generated by ML models. For example, if a feature vector $X_{k, s}^{\off}$ and $X_{k, s}$ is available for each unit during both the offline and online phases, a predictive model $f_k$ can be applied to produce surrogate rewards $\hat{R}_{k, s}^\off = f_k(X_{k, s}^\off)$ and $\hat{R}_{k, s} = f_k(X_{k, s})$. Here, $f_k$ may be trained using historical data or be a pre-trained AI model, and it need not be the same across arms, like the one-size-fits-all surrogate rewards described in Section \ref{sec:intro}. If the feature vectors are i.i.d. across both phases, the surrogate rewards will also be i.i.d. Notably, in most dynamic settings, this assumption becomes less plausible as $N_k$ increases.

For regret analysis, we assume that the true and surrogate rewards are bivariate Gaussian with unknown mean and covariance matrix:
\begin{equation}
\label{eqn: gaussian model}
\begin{pmatrix}
    R_{k,s} \\
    \hat R_{k,s} \\
\end{pmatrix} \sim \mathcal{N}\left(\begin{pmatrix}
    \mu_k \\
    \tilde \mu_k \\
\end{pmatrix}, \begin{pmatrix}
    \sigma_k^2& \rho_k\sigma_k\tilde \sigma_k  \\
     \rho_k\sigma_k\tilde \sigma_k & \tilde\sigma_k^2\\
\end{pmatrix}\right), \quad \forall s=1,2,\cdots.
\end{equation}
Here, $\sigma_k^2$ and $\tilde \sigma_k^2$ are the variances of the true reward and prediction, and $\rho_k$ is the correlation coefficient measuring the quality of the ML model. Let $\Sigma_k$ denote the 2 by 2 covariance matrix in \eqref{eqn: gaussian model}. Importantly, we allow the means $\mu_k$ and $\tilde \mu_k$ to be arbitrarily different. Moreover, unlike \cite{cheung2024leveraging}, we do not even require access to any partial knowledge of the bias $\tilde \mu_k - \mu_k$. Therefore, to effectively leverage the correlation between the true and surrogate rewards, it is crucial to jointly observe both of them in the online phase. This marks a departure from prior work on MAB with offline data \cite{shivaswamy2012multi, banerjee2022artificial, zhang2019warm, cheung2024leveraging}, where auxiliary information is available only during the offline phase. Throughout the paper, we will assume $|\rho_k|< 1$ to avoid degenerate settings.

Although we adopt the Gaussian model~\eqref{eqn: gaussian model} to enable sharp, finite-sample analysis, the resulting algorithms are not restricted to Gaussian data. In particular, many of our estimators and confidence constructions are based on sample averages, which are approximately normal in moderate-to-large samples by the central limit theorem. We empirically validate this robustness under a broad class of non-Gaussian distributions through extensive simulations and real-world case studies in Sections~\ref{sec: simulation} and~\ref{sec: application}. Finally, in Section~\ref{sec: batch reward} we extend the theory to a batched MAB setting and leverage Cram\'er-type self-normalized moderate deviation bounds \cite{petrov2012sums, jing2003self} to obtain Gaussian-calibrated tail control for general distributions.


\section{MLA-UCB: algorithm and regret analysis}
\subsection{Machine learning-assisted mean estimator}
\label{sec: pp mean} 
The core building block of the UCB algorithm is the estimation of the mean reward for each arm. With surrogate rewards, we can apply the debiasing technique to reduce the variance of the sample average. We begin by defining several quantities. For a fixed number $n$ of online observations, define
$$
\E_{k,n}[R] =  \frac{1}{n}\sum_{s=1}^{n}R_{k,s},\quad \E_{k,n}[\hat R] =  \frac{1}{n}\sum_{s=1}^{n}\hat R_{k,s},\quad \E^\off_{k}[\hat R] =  \frac{1}{N_{k}}\sum_{s=1}^{N_{k}}\hat R^{\off}_{k,s},
$$
and 
$$
\hat \sigma_{R,\hat R, k,n} = \frac{1}{n-1}\sum_{s=1}^{n}(R_{k,s}-\E_{k,n}[R])(\hat R_{k,s}-\E_{k,n}[\hat R]),\quad  \hat \sigma_{\hat R, k,n}^2 = \frac{1}{n-1}\sum_{s=1}^{n}(\hat R_{k,s}-\E_{k,n}[\hat R])^2
$$
With the notation, we define the machine learning-assisted mean estimator for the mean reward $\mu_k$ of arm $k$ as 
\begin{equation}
        \label{eqn: pp mean re}
            \mukt = \E_{k,n}[R] - \hat{\lambda}_{k,n}\left(\E_{k,n}[\hat R]  - \E_{k}^\off[\hat R]\right).
\end{equation}
where
\begin{equation}
        \label{eqn: power tuning empirical}
            \hat{\lambda}_{k,n} = \frac{N_k}{n+N_k}\frac{\hat \sigma_{R,\hat R, k,n}}{\hat \sigma_{\hat R, k,n}^2}.
\end{equation} 

The expression of \eqref{eqn: pp mean re} is motivated by the PPI++ method \cite{angelopoulos2023ppi++}, which has also been extensively studied in various fields, including causal inference \citep{robins1994estimation,scharfstein1999adjusting,glynn2010introduction}, semi-supervised inference \citep{zhang2019semi, azriel2022semi}, and control variates \citep{lavenberg1981perspective, nelson1990control}. For any non-random $\hat{\lambda}_{k,n}$, $\mukt$ is an unbiased estimate of $\mu_k$ as surrogate rewards are i.i.d. across both the offline and online phases. When $\hat{\lambda}_{k,n}=1$, the estimator can be written as $(\E_{k,n}[R] - \E_{k,n}[\hat R])  + \E_{k}^\off[\hat R]$, where the last term is the biased offline estimate of $\mu_k$ and the difference of the first two terms can be viewed as a bias estimate. The choice of $\hat{\lambda}_{k,n}$ in \eqref{eqn: power tuning empirical} is the plug-in estimate of the variance minimizer -- it ensures that the $\mukt$ is no worse than the sample mean estimator.

While asymptotic inference of the estimator \eqref{eqn: pp mean re} is well established, it is insufficient for regret analysis in the MAB setting. In particular, UCB algorithms typically employ a decaying sequence of significance levels, or a fixed horizon-dependent level, and the regret analysis involves simultaneous coverage of confidence intervals over rounds, neither of which are challenges that the standard theory is equipped to address. For our purpose, we need a refined analysis of the finite-sample distribution of $\mukt$.

The key challenge lies in the ratio form of $\hat{\lambda}_{k, n}$. To address this, we start by reformulating the estimator in a way that facilitates the analysis of the finite-sample distribution of $\mukt$.

\begin{proposition}
\label{prop: intercept}
    Denote by the pooled sample mean of the online and offline predictions
    $$
    \E_{k,n}^{\all}[\hat R] = \frac{1}{n+N_k}\lb\sum_{s=1}^{n}\hat R_{k,s}+\sum_{s=1}^{N_k}\hat R_{k,s}^\off\rb.
    $$
    Then $\mukt$ is given by the intercept of the following ordinary least squares (OLS) estimator, i.e., 
    \begin{equation}
    \label{eqn: OLS equivalence}
        \mukt = \arg\min_{\mu}\min_{\beta} \sum_{s=1}^{n}\lb R_{k,s} - \mu -\beta \lb\hat R_{k,s} - \E_{k,n}^{\all}[\hat R]\rb\rb^2.
    \end{equation}
\end{proposition}
Building off the equivalence with an OLS estimator, we can derive the following characterization of the distribution of $\mukt$.

\begin{proposition}
    \label{prop: pp mean distribution}
    Define $\Fkt = \sigma \lb\{\hat{R}_{k,s}^\off\}_{s=1}^{N_k},\{\hat{R}_{k,s}\}_{s=1}^{n}\rb$ as the $\sigma$-field generated by all surrogate rewards for arm $k$ up to round $t$. Then $\mukt$ can be decomposed as 
    \begin{equation}
        \mukt = \mu_k + S_{1,k,n} + S_{2,k,n}, 
    \end{equation}
    where $S_{1,k,n}$ and $S_{2,k,n}$ are independent random variables, with $S_{1,k,n}\in \Fkt$, 
    \begin{equation}
    \label{eqn: pp mean distribution}
    \begin{aligned}
        S_{1,k,n} \sim \mathcal{N}\lb0, \frac{1}{n+N_{k}}\rho_k^2 \sigma_k^2\rb, \quad
        S_{2,k,n}\big|\Fkt \sim \N \lb0, \frac{1}{n} Z_{k,n} (1-\rho_k^2) \sigma_k^2 \rb, 
    \end{aligned}
    \end{equation}
    and $Z_{k,n}\in \Fkt$ is defined as 
    \begin{equation*}
    \label{eqn: Zkt}
        Z_{k,n} = 1 + \frac{n(\E_{k,n} [\hat R] - \E_{k,n}^\all [\hat R])^2}{\sum_{s=1}^{n} (\hat R_{k,s} - \E_{k,n}[\hat R])^2}.
    \end{equation*}
\end{proposition}
Proposition \ref{prop: pp mean distribution} demonstrates that the estimation error of $\mukt$ can be decomposed into two components: the first component $S_{1,k,n}$ represents the bias of using the empirical mean of surrogate rewards instead of the true mean in the regressor, and the second component represents the uncertainty of intercept estimation in the linear regression model. 

Despite the clean distributional characterization, it does not directly imply an upper confidence bound on $\mu_k$ since the variance $\sigma_k^2$ and the correlation $\rho_k$ are unknown. While both can be estimated consistently via plug-in methods, as previously noted, standard consistency alone is not enough. Following the logic of Student's t-tests, if we can 
construct an unbiased estimator $\hat{\sigma}_1^2$ of $\rho_k^2 \sigma_k^2$ that is independent of $S_{1,k,n}$ and $k_1\hat{\sigma}_1^2$ is $\chi^2$-distributed with $k_1$ degrees of freedom, and an unbiased estimator $\hat{\sigma}_2^2$ of $(1 - \rho_k^2)\sigma_k^2$ that is independent of $S_{2,k,n}$ and  $k_2\hat{\sigma}_2^2$ is $\chi^2$-distributed with $k_2$ degrees of freedom, then
\[\P\lb -\frac{S_{1,k,n}}{\hat{\sigma}_1} \le q_{k_1}(\delta) \sqrt{\frac{1}{n + N_k}} \rb\ge 1 - \delta, \,\,\text{and} \,\, \P\lb -\frac{S_{2,k,n}}{\hat{\sigma}_2} \le q_{k_2}(\delta) \sqrt{\frac{Z_{k, n}}{n}} \rb\ge 1 - \delta.\]
This can yield an upper confidence bound on $\mu_k$ by \eqref{eqn: pp mean distribution} and a union bound. We prove that the empirical variance of true rewards can be used as a conservative version of $\hat{\sigma}_1$, since $\var(R_k) = \sigma_k^2 \ge \rho_k^2 \sigma_k^2$, and the residual mean square error of the regression \eqref{eqn: OLS equivalence} can be used as $\hat{\sigma}_2^2$, since  $\var(R_k|\hat R_k) = (1-\rho_k^2)\sigma_k^2$. We are unable to find an unbiased estimator of $\rho_k^2 \sigma_k^2$ that is independent of $S_{1,k,n}$, though we show it has a negligible effect when $N_k >\!\!> n$ for suboptimal arms.

\begin{proposition}
    \label{prop: CI mean}
    For $n\geq 3$, let $\hat \beta_{k,n}$ be the optimal $\beta$ in \eqref{eqn: OLS equivalence} and 
    \begin{equation*}
        \label{eqn: def var estimate}
        \begin{aligned}
            &\hat{\sigma}^2_{R,k,n} = \frac{1}{n - 1} \sum_{s=1}^{n}\lb R_{k,s}  - \E_{k,n}[ R]\rb^2,\\
            &\hat{\sigma}^2_{\epsilon,k,n} = \frac{1}{n - 2} \sum_{s=1}^{n}\lb R_{k,s} - \mukt -\hat \beta_{k,n} \lb\hat R_{k,s} - \E_{k,n}^{\all}[\hat R]\rb\rb^2. 
        \end{aligned}
    \end{equation*}
    Then, for any $\delta\in (0,\frac{1}{2})$ we have 
    \begin{equation}
        \label{eqn: pp mean concentration}
        \P\lb \mu_k \leq \mukt + q_{n - 2}(\delta) \lb\sqrt{\frac{ \hat\sigma_{R,k,n}^2}{n+N_k}}+\sqrt{\frac{\Zkt\hat\sigma_{\epsilon,k,n}^2}{n}}\rb\rb \geq 1-2\delta,
    \end{equation}
    where $q_{d}(\delta)$ is the $1-\delta$ quantile of the Student's t-distribution of $d$ degrees of freedom.
\end{proposition}

\subsection{MLA-UCB algorithm}
\label{sec: alg design}
\begin{algorithm}[ht]
   \caption{Machine Learning-Assisted Upper Confidence Bound (MLA-UCB)}
   \label{alg:PPUCB}
\begin{algorithmic}
    \STATE \textbf{Initialization}: Pull each arm four times.
    \FOR{$t = 4K+1$ to $T$}
   \STATE Compute the machine learning-assisted mean estimator $\mukntk$ for each arm.
   \STATE Compute the variance estimate $Z_{k,n_{k,t}}, \hat\sigma_{R,k,n_{k,t}}^2, \hat\sigma_{\epsilon,k,n_{k,t}}^2$ for each arm.
    \STATE Pull the arm $$A_t = \arg\max_{k}\left\{ \mukntk +q_{n_{k,t} - 2}\lb\frac{1}{2t\sqrt{\ln t}}\rb \lb\sqrt{\frac{ \hat\sigma_{R,k,n_{k,t}}^2}{n_{k,t}+N_k}}+\sqrt{\frac{Z_{k,n_{k,t}}\hat\sigma_{\epsilon,k,n_{k,t}}^2}{n_{k,t}}}\rb\right\}.$$
   \ENDFOR
\end{algorithmic}
\end{algorithm}

As with the standard UCB algorithm, the MLA-UCB algorithm pulls the arm with the largest upper confidence bound defined in \eqref{eqn: pp mean concentration}. To kick off the process with an initial variance estimate, we need to pull each arm four times. After the initial pulls, the significance level is set at $\frac{1}{2t\sqrt{\ln t}}$ at round $t$. The method is summarized in Algorithm \ref{alg:PPUCB}. Note that the algorithm of \cite{verma2021stochastic} can be viewed as a special case of MLA-UCB when $N_k = \infty$, except that it uses a smaller significance level $1/t^2$, which may lead to suboptimal performance.

Compared to the upper confidence bound constructed in \cite{cowan2018normal} for standard normal bandits, we use the exact quantile of t-distribution $q_{n_{k,t} - 2}\lb\frac{1}{2t\sqrt{\ln t}}\rb$ to scale the standard deviation, while they use a scaling parameter of $\sqrt{n_{k,t} (t^{\frac{2}{n_{k,t} - 2}}-1)}$. As we shall see in the following proposition, their scaling parameter can be viewed as an upper bound for the exact quantile. \footnote{Our method requires one extra degree of freedom compared to the method in \cite{cowan2018normal}, since we need to estimate the variance ratio $\hat\lambda_{k,n_{k,t}}$ (defined in \eqref{eqn: power tuning empirical}) in addition to the sample mean. This difference is negligible when $n_{k,t}$ is large. } In Figure \ref{fig: quantile}, we plot the difference between the two choices of scaling parameter. As we shall see, when the number of pulls is relatively small, the quantile bound can be substantially larger than the true quantile. The gap remains sizable even with a large number of pulls.

\begin{proposition}
    \label{prop: quantile bound}
    For any real number $s>1$ and integer $d\geq 2$ we have
    \begin{equation*}
        q_{d}\lb\frac{1}{2s\sqrt{\ln s}}\rb \leq \sqrt{d(s^{\frac{2}{d-1}}-1)}
    \end{equation*}
\end{proposition}
\begin{figure}
    \centering
    \includegraphics[width=0.7\linewidth]{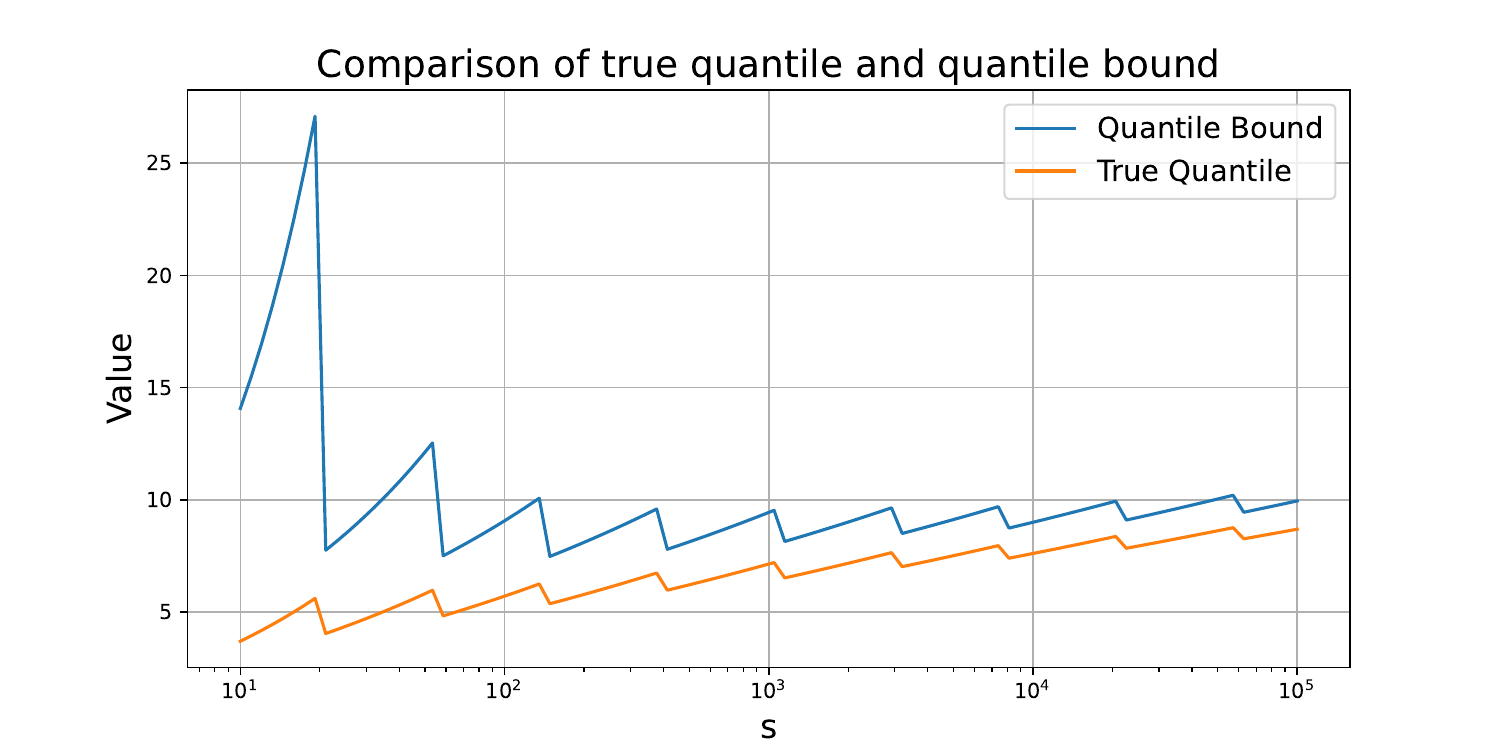}
    \caption{Comparison of the quantile bound $\sqrt{d(s^{\frac{2}{d-1}}-1)}$ and the true quantile $q_{d}\lb\frac{1}{2s\sqrt{\ln s}}\rb$ in Proposition \ref{prop: quantile bound}. We set $d = \lfloor\ln s \rfloor$ and vary the values of $s$ to reflect the order of regret bound \eqref{eqn: UCB regret}. }
    \label{fig: quantile}
\end{figure}

Moreover, by the law of large numbers, the empirical variance should converge to the true variance, i.e., as $n_{k,t}\rightarrow \infty$ we have $\hat\sigma_{R,k,n_{k,t}}^2 \rightarrow \sigma_k^2, \hat\sigma_{\epsilon,k,n_{k,t}}^2 \rightarrow (1-\rho_k^2)\sigma_k^2$, and similarly $Z_{k,n_{k,t}}\rightarrow 1$. Therefore, as long as $N_k \gg n_{k,t}$, the confidence bound in \eqref{eqn: pp mean concentration} will be approximately $\mukt + q_{n_{k,t} - 2}\lb\frac{1}{2t\sqrt{\ln t}}\rb\sqrt{\frac{(1-\rho_k^2)\sigma_{k}^2}{n_{k,t}}}$ when $n_{k,t}$ is large. This demonstrates that the surrogate rewards effectively reduce the variance by $\rho_k^2$.

\subsection{Regret analysis}
\label{sec: Theory known rho}
In this section, we analyze the regret of the MLA-UCB algorithm. Intuitively, reducing the variance of the mean reward estimate results in a sharper regret bound. The formal result is presented in the following theorem. 

\begin{theorem}
\label{thm: regret bound}
    Under the Gaussian reward model \eqref{eqn: gaussian model}, for any $T\geq 4K$, if offline sample size satisfies 
\begin{equation}
    \label{eqn: sample size condition}
    N_k\geq \frac{1}{\delta_k}\lb\frac{2\ln T}{\ln\lb 1+\frac{\Delta_k^2}{24\sigma_{k}^2} \rb}+3\rb, \forall k \neq k^\star,
\end{equation}
for some $\delta_k<1,$ then the expected regret of Algorithm \ref{alg:PPUCB} can be bounded by:
\begin{equation}
\label{eqn: regret bound inf eps}
    \begin{split}
        \E[\reg] \leq \sum_{k\neq k^\star}\frac{2\Delta_k\ln T}{\ln\lb 1+\frac{\Delta_k^2}{\sigma_{k}^2}\frac{1}{(\sqrt{1-\rho_k^2}+\sqrt{\delta_k})^2}\rb}+O\lb(\ln T)^{5/6}\rb.
    \end{split}
\end{equation}
\end{theorem}
Compared to the regret bound of the classical UCB algorithm in \eqref{eqn: UCB regret}, the regret bound in \eqref{eqn: regret bound inf eps} has an extra factor of $(\sqrt{1-\rho_k^2} + \sqrt{\delta_k})^2$ multiplied by the reward variance $\sigma_k^2$ on the first term. In the following corollary, we further show that the factor can be reduced to $(1-\rho_k^2)$ as long as $N_k$ is poly-logarithmic in $T$. 
\begin{corollary}
\label{cor: regret bound inf N}
    Under the Gaussian data generation model \eqref{eqn: gaussian model}, if $N_k = \Omega ((\ln T)^{4/3}), \forall k\neq k^\star$, then the expected regret of Algorithm \ref{alg:PPUCB} can be bounded by:
\begin{equation}
\label{eqn: regret bound inf N}
    \begin{split}
        \E[\reg] \leq \sum_{k\neq k^\star}\frac{2\Delta_k\ln T}{\ln\lb 1+\frac{\Delta_k^2}{\sigma_{k}^2(1-\rho_k^2)}\rb}+O\lb(\ln T)^{5/6}\rb.
    \end{split}
\end{equation}
\end{corollary}
We highlight a few implications of Corollary \ref{cor: regret bound inf N}:
\begin{enumerate}
    \item If $\rho_k \rightarrow 1, \forall k=1,\cdots, K$, the surrogate reward becomes nearly an affine transformation of the true reward, yielding effectively  $N_k$ additional reward observations for each arm. As $N_k$ grows, the leading term in the regret vanishes, even though the MLA-UCB algorithm is agnostic to the high quality of the surrogates.
    \item If $\rho_k = 0, \forall k=1,\cdots, K$, the predictions are independent of the true rewards under the Gaussian model \eqref{eqn: gaussian model} and hence provide no information. In this scenario, the regret bound \eqref{eqn: regret bound inf N} of MLA-UCB matches the leading term of the regret bound \eqref{eqn: UCB regret}, which is known to be asymptotically optimal \citep{lai1985asymptotically}. 
    \item In the general case, if $\rho_k$ are neither all 1s nor all 0s, the predictions are informative but are not perfect to infer the true reward. In this case, the regret of our MLA-UCB algorithm interpolates the two extreme cases, without prior knowledge of the bias or the correlation. 
\end{enumerate}

Next, we will establish a regret lower bound to show that our MLA-UCB algorithm achieves the optimal performance. Following the standard settings in the MAB literature \citep{lai1985asymptotically, burnetas1996optimal}, we consider the class of uniformly good policies as follows. Here, we will use $\E_{\theta}^\pi$ and $\P_{\theta}^\pi$ to denote the expectation and probability under a specific parameter and policy.
\begin{definition}[Uniformly good policies]
\label{def: uniformly good policy}
A policy $\pi$ is called \emph{uniformly good} if for every instance $\theta:=(\{\mu_k,\tilde\mu_k,\Sigma_k\}_{k=1}^K)$ of \eqref{eqn: gaussian model} and every $a>0$
\[
\mathbb E^\pi_\theta[\mathrm{Reg}_T] = o(T^a),\quad T\to\infty.
\]
\end{definition}
The reason why we restrict our attention to uniformly good policies is to eliminate trivial policies that perform well only on a subset of instances.  For example, the policy that always pulls the first arm has 0 regret on instances with $k^\star = 1$, but it has $O(T)$ regret on instances with $k^\star \neq 1$. Apparently, the condition can be easily satisfied by the classical UCB algorithm and our MLA-UCB algorithm. Under this definition, we can derive the following regret lower bound for all uniformly good policies.
\begin{theorem}
\label{thm: lower bound}
    Under the Gaussian data generation model \eqref{eqn: gaussian model}, let $\pi$ be any uniformly good policy as defined in Definition \ref{def: uniformly good policy}, then 
    \begin{equation}
    \label{eqn: lower bound}
\liminf_{T\to\infty}\frac{\mathbb E_\theta^\pi[\mathrm{Reg}_T]}{\ln T}
\ge
\sum_{k\ne k^\star}
\frac{2\Delta_k}{\ln\lb1+\dfrac{\Delta_k^2}{\sigma_k^2(1-\rho_k^2)}\rb}.
    \end{equation}
\end{theorem}
Here, we can find that the lower bound in \eqref{eqn: lower bound} matches the leading term of the upper bound \eqref{eqn: regret bound inf N}, which implies that our MLA-UCB algorithm achieves optimal regret as long as $N_k=\Omega((\ln T)^{4/3})$.   
\section{Extension: Batched reward with non-Gaussian distribution}
\label{sec: batch reward}
So far, we have derived the asymptotic optimal algorithm under a Gaussian model. In this section, we further extend our MLA-UCB algorithm beyond the Gaussian model in a batched reward MAB setting, where a decision is assigned to a batch of individuals in each round. This corresponds to the practical setting where decision-making is constrained by both implementation logistics and legal or regulatory requirements, making it infeasible to pull arms at the level of individual units. \footnote{The batched-reward model in this section should not be confused with the classical batched bandit setting \cite{perchet2016batched} that limits the number of policy updates. Here, the policy is still updated after every decision epoch using all past data; the only change is that each chosen arm is deployed to a batch of $m$ i.i.d. units, so each pull yields a batch of observations rather than a single reward pair.}  For instance, online retailers may be restricted from updating prices too frequently due to customer experience concerns, technical limitations, or fair pricing regulations. In this scenario, our estimator remains the same form as in Section \ref{sec: pp mean}, but the analysis can no longer rely on exact Gaussian/$t$-tail calculations. Instead, we control the relevant confidence events using self-normalized moderate deviation bounds for batched averages \citep{petrov2012sums, jing2003self}, which yields a computable confidence bound and a regret guarantee whose leading term preserves the same variance-reduction effect, an improvement proportional to $(1-\rho_k^2)$ in the leading term. 

\subsection{Basic settings of batched reward}
We first introduce the formal setting of MAB with batched reward. Fix a batch size $m\ge 1$. The online interaction proceeds over $T$ rounds. At each round $t\in\{1,\dots,T\}$, the policy selects an arm $A_t\in[K]$ based on the history up to $t-1$. We denote by $n_{k, t}$ the total number of batched pulls for arm $k$ right before any arm is pulled at round $t$. If $A_t=k$, then arm $k$ is deployed to $m$ independent units in that round and we observe a new batch of $m$ i.i.d. pairs of true and surrogate rewards
\[
\{(R_{k,mn_{k,t}+j},\hat R_{k,mn_{k,t}+j})\}_{j=1}^m \stackrel{\text{i.i.d.}}{\sim} \mathcal{P}_k.
\]
Then the total number of
online individual observations of arm $k$ available at round $t$ is $m n_{k,t}$, i.e.,
$\{(R_{k,s},\hat R_{k,s})\}_{s=1}^{m n_{k,t}}$ are observed. In addition to the online samples, we assume the decision maker has access to a static pool of $mN_k$ surrogate rewards $\{\hat{R}_{k,s}^\off\}_{s=1}^{mN_k}$ for arm $k$, and the sample size $mN_k$ is assumed to be a multiplier of $m$ for simplicity. We assume that $\hat{R}_{k,1}^\off, \ldots, \hat{R}_{k, mN_k}^\off, \hat{R}_{k, 1}, \hat{R}_{k, 2}, \ldots $ are independent and identically distributed (i.i.d.).
Accordingly, the cumulative regret over $T$ decision rounds is
\[
\mathrm{Reg}_T=\sum_{t=1}^T m(\mu^\star-\mu_{A_t}),
\]
and we will often report $\mathrm{Reg}_T/m$ as the batch average regret.

We will make the following assumption on the data-generating model.
\begin{assumption}
\label{asm: batch reward}
    We assume that the observation $(R_{k,s},\hat R_{k,s})$ is i.i.d. drawn from the distribution $\mathcal{P}_k$ that satisfies the following conditions:
    \begin{enumerate}
        \item The mean $\mu_k:=\E[R_{k,s}],\tilde \mu_k:=\E[\hat R_{k,s}]$, and the variance $ \sigma_k^2:=\var(R_{k,s}), \tilde \sigma_k^2:=\var(\hat R_{k,s})$, exist and are finite. 
        \item Consider the linear projection of $R_{k,s}$ onto $\hat R_{k,s}$,
        define the correlation $\rho_k=\frac{\cov(R_k,\hat R_k)}{\sigma_k\tilde\sigma_k},$ the slope $\beta_k=\rho_k\frac{\sigma_k}{\tilde\sigma_k}$, and the residual after linear projection
        \begin{equation}
            \label{eqn: def residual}
            \epsilon_{k,s} = (R_{k,s}-\mu_k)-\beta_k(\hat R_{k,s}-\tilde\mu_k).
        \end{equation}
        Then we assume the skewness ratios
\begin{equation}\label{eq:ell_R}
\ell_{R_k}:=\frac{\E[|R_{k,s}-\mu_k|^3]}{\sigma_k^3}, \ell_{\hat R_k}:=\frac{\E[|\hat R_{k,s}-\tilde\mu_k|^3]}{\tilde\sigma_k^3}, \ell_{\epsilon_k}:=\frac{\E[|\epsilon_{k,s}|^3]}{(1-\rho_k^2)^{3/2}\sigma_k^3}
\end{equation}
are finite and uniformly bounded by some constant $C>0$, and $R_{k,s}, \hat R_{k,s}, \epsilon_{k,s}$ are sub-Gaussian variables such that for every arm $k\in[K]$,
\begin{equation*}
\left\|\frac{R_{k,s}-\mu_k}{\sigma_k}\right\|_{\psi_2}\le L,
\qquad
\left\|\frac{\hat R_{k,s}-\tilde\mu_k}{\tilde\sigma_k}\right\|_{\psi_2}\le L,
\qquad
\left\|\frac{\epsilon_{k,s}}{\sqrt{1-\rho_k^2}\,\sigma_k}\right\|_{\psi_2}\le L
\end{equation*}
for some constant $L>0$, where $\|\cdot\|_{\psi_2}$ is the sub-Gaussian norm. 
    \end{enumerate}
\end{assumption}
Assumption \ref{asm: batch reward} requires sub-Gaussianity on $R_{k,s}, \hat R_{k,s}, \epsilon_{k,s}$ to control the tail behavior. Importantly, we do not know the variances, skewness ratios, or sub-Gaussian norm -- this distinguishes our method from standard UCB algorithms that exploit the sub-Gaussian norm. Notably, the linear projection implies that the residuals $\epsilon_{k,s}$ are i.i.d. with $\E[\epsilon_{k,s}]=0, \var(\epsilon_{k,s}) = (1-\rho_k^2)\sigma_k^2$.

Furthermore, we make the following assumption on the model parameters.
\begin{assumption}
\label{asm: m and T}
    Define $\sigma_{\max} = \max_{k}\sigma_k,$ $ \sigma_{\min} = \min_{k}\sigma_k$, we assume that:
    \begin{enumerate}[(a)]
        \item $m\geq \max\left\{8(\ln T)^4, 8(\ln T)^3 \max\left\{{\ell_{R_k}^2},{\ell_{\hat R_k}^2},{\ell_{\epsilon_k}^2}\right\}, 4\max\left\{{\ell_{R_k}^{-1}},{\ell_{\hat R_k}^{-1}},{\ell_{\epsilon_k}^{-1}}\right\}\right\}$
        \item $\frac{\sigma_k^2}{m\Delta_k^2}=\Omega(1), \forall k\neq k^\star$,  $\frac{\sigma_{\max}}{\sigma_{\min}} = O(1)$.
    \end{enumerate}
\end{assumption}

Here, Assumption~2(a) guarantees that the batch size $m$ is sufficiently large for the self-normalized moderate deviation bound to deliver the required tail control for the machine-learning-assisted mean estimator. Assumption~2(b) rules out a degenerate ``easy'' regime and ensures the batched reward doesn't make the MAB problem trivial: since $\sigma_k^2/m$ is the variance of the batch-average reward and $\Delta_k^2$ quantifies the signal strength, a vanishing ratio $\sigma_k^2/(m\Delta_k^2)$ would imply that the optimal arm can be identified with overwhelming probability from the initialization batches alone. 
Moreover, Assumption~2(b) aligns with the local asymptotic regime in classical asymptotic statistics \citep{van2000asymptotic, hirano2009asymptotics}, where one studies alternatives separated by $\Delta \asymp h/\sqrt{n}$ for a fixed constant $h$ and sample size $n$. This regime is non-trivial in the sense that separations of larger order enable trivial identification, whereas separations of smaller order make the identification impossible. The second condition in Assumption ~2(b) ensures that all the arms are in the same regime for technical simplicity.

\subsection{MLA-UCB algorithm for MAB with batched reward}
Similar to the Gaussian setting, the machine learning assisted mean estimator can be used to improve the mean estimation and identify the best arms. For round $t$ and a fixed number $n$ of online reward batches, define the empirical means of true reward and surrogate reward by
\[
\E_{k,n}[R] := \frac{1}{m n}\sum_{s=1}^{m n} R_{k,s},
\quad
\E_{k,n}[\hat R] := \frac{1}{m n}\sum_{s=1}^{m n} \hat R_{k,s}.
\]

\[
\E_k^{\off}[\hat R] := \frac{1}{mN_k}\sum_{s=1}^{mN_k}\hat R^{\off}_{k,s},
\quad
\E^{\mathrm{all}}_{k,n}[\hat R]
:=\frac{1}{m(n+N_k)}
\left(\sum_{s=1}^{mn}\hat R_{k,s}+\sum_{s=1}^{mN_k}\hat R^{\off}_{k,s}\right).
\]
Let the online sample covariance and variance be
\begin{align*}
    \hat \sigma_{R,\hat R, k,n}
:=& \frac{1}{mn}\sum_{s=1}^{m n}
(R_{k,s}-\E_{k,n}[R])(\hat R_{k,s}-\E_{k,n}[\hat R]),\\
\hat \sigma_{\hat R, k,n}^2
:=& \frac{1}{mn}\sum_{s=1}^{m n}
(\hat R_{k,s}-\E_{k,n}[\hat R])^2.
\end{align*}
Similar to \eqref{eqn: pp mean re}, we  can define the machine learning assisted mean estimator as
\[
\hat\mu^{\mla}_{k,n}
:=\E_{k,n}[R]-\hat\lambda_{k,n}\big(\E_{k,n}[\hat R]-\E_{k}^{\off}[\hat R]\big),
\quad
\hat\lambda_{k,n}
:=\frac{N_k}{n+N_k}\cdot
\frac{\hat \sigma_{R,\hat R, k,n}}{\hat \sigma_{\hat R, k,n}^2}.
\]
Equivalently, letting $
\hat\beta_{k,n}:=\frac{\hat \sigma_{R,\hat R, k,n}}{\hat \sigma_{\hat R, k,n}^2}
$
then Proposition \ref{prop: intercept} implies that $\hat\mu^{\mathrm{MLA}}_{k,n}$ and $\hat\beta_{k,n}$ are the intercept and slope in the OLS regression of $R_{k,s}$ on
$\hat R_{k,s}-\E^{\mathrm{all}}_{k,n}[\hat R]$ using $s=1,\dots,m n$.
Define the empirical variance of rewards by
\[
\hat\sigma^2_{R,k,n}
:=\frac{1}{mn}\sum_{s=1}^{m n}\big(R_{k,s}-\E_{k,n}[R]\big)^2,
\]
and define the residual variance estimator by \footnote{The scaling parameter becomes $mn$ for technical simplicity, since we don't have a clean $\chi^2$ argument without Gaussian assumption.}
\[
\hat\sigma^2_{\epsilon,k,n}
:=\frac{1}{mn}\sum_{s=1}^{m n}
\Big(R_{k,s}-\hat\mu^{\mla}_{k,n}-\hat\beta_{k,n}\big(\hat R_{k,s}-\E^{\all}_{k,n}[\hat R]\big)\Big)^2.
\]
Again, we will use these empirical variance estimates to construct the confidence bound. 

The algorithm design in the batch setting is shown in Algorithm \ref{alg:PPUCB batch}. Compared to Algorithm \ref{alg:PPUCB}, Algorithm \ref{alg:PPUCB batch} uses the standard $\sqrt{2\ln t}$ instead of the quantile of Student's t-distribution, because batched data allows us to approximate the tail probability using a Gaussian distribution directly. Moreover, we introduce an additional term $(1+m^{-1/3})$ for the ease of regret analysis, which should be negligible with a relatively large batch size $m$. 

\begin{algorithm}[ht]
   \caption{MLA-UCB for MAB with batched rewards}
   \label{alg:PPUCB batch}
\begin{algorithmic}
    \STATE \textbf{Initialization}: Pull each arm once.
    \FOR{$t = K+1$ to $T$}
   \STATE Compute the machine learning-assisted mean estimator $\mukntk$ for each arm.
   \STATE Compute the variance estimate $\hat\sigma_{R,k,n_{k,t}}^2, \hat\sigma_{\epsilon,k,n_{k,t}}^2$ for each arm.
    \STATE Pull the arm $$A_t = \arg\max_{k}\left\{ \mukt + \lb1+{m^{-\frac{1}{3}}}\rb\lb\sqrt{\frac{ 2\hat\sigma_{R,k,n_{k,t}}^2\ln t}{m(n_{k,t}+N_k)}}+\sqrt{\frac{2\hat\sigma_{\epsilon,k,n_{k,t}}^2\ln t}{mn_{k,t}}}\rb\right\}.$$
   \ENDFOR
\end{algorithmic}
\end{algorithm}
\subsection{Regret analysis}
The core difficulty of handling general distribution is to construct upper confidence bound for $\mu_k$ using empirical mean and variance estimators. In the Gaussian model, we can utilize the Student's t-distribution to construct confidence bound, which is difficult to generalize beyond Gaussian model. Classical concentration inequality, such as empirical Bernstein inequality, would require boundedness of $R_{k,s},\hat R_{k,s},  \epsilon_{k,s}$ in our framework. Although the range of $R_{k,s},\hat R_{k,s}$ can often be determined in prior, the range of $\epsilon_{k,s}$ is usually unknown as it cannot be directly observed in practice. Fortunately, when the data comes in batches, the distribution of batch-average reward is approximately Gaussian, and the following self-normalized Cramér-type large deviations bound \citep{petrov2012sums, jing2003self} allows us to control the tail behavior of machine learning-assisted mean estimator for general distributions. 
\begin{lemma}[Theorem 2.3 of \cite{jing2003self}]
\label{lem: self normalize}
Let $X_{1}, X_{2}, \dots$ be independent random variables with $E X_{i} = 0$ and $0 < E X_{i}^{2} < \infty$. Set $$S_{n} = \sum_{i=1}^{n} X_{i},\quad  B_{n}^{2} = \sum_{i=1}^{n} E X_{i}^{2}, \quad V_{n}^{2} = \sum_{i=1}^{n} X_{i}^{2}, \quad L_{n} = \sum_{i=1}^{n} E|X_{i}|^{3}, \quad d_{n} = \frac{B_n}{L_{n}^{1/3}}.$$
Then, for $0 \le x \le d_{n}$,
\begin{equation}
\frac{P(S_{n}/V_{n} \ge x)}{1-\Phi(x)} = 1 + O(1) \left(\frac{1+x}{d_{n}}\right)^{3},\quad \frac{P(S_{n}/V_{n} \le -x)}{1-\Phi(x)} = 1 + O(1) \left(\frac{1+x}{d_{n}}\right)^{3}
\end{equation}
where $O(1)$ is bounded by an absolute constant. 
\end{lemma}

Using the result above, we shall derive the following regret upper bound. 
\begin{theorem}
\label{thm: regret bound non Gaussian}
    Under Assumption \ref{asm: batch reward} and \ref{asm: m and T}, if $T\geq \max\{K+1, 5\}$, and the offline sample size $N_k = \Omega (\frac{\sigma_k^2}{m\Delta_k^2}(\ln T)^{5/3}), \forall k\neq k^\star$, then the expected regret of Algorithm \ref{alg:PPUCB batch} can be bounded by:
    \begin{equation}
        \frac{1}{m}\E[\reg]\leq \sum_{k\neq k^\star}\lb \frac{2(1-\rho_k^2)\sigma_k^2\ln T}{m\Delta_k^2} +\frac{C\sigma_k^2(\ln T)^{2/3}}{m\Delta_k^2}\rb\Delta_k,
    \end{equation}
    where $C>0$ is a constant independent of $\mu_k, \sigma_k^2, \tilde\mu_k, \tilde\sigma_k^2, T$.
\end{theorem}
Notably, in our setting, consider an alternative algorithm that runs the bandit without collecting batches, and assume that the reward $R_k$ is from Gaussian distribution with known $\sigma_k^2$, then the standard regret bound of classical UCB algorithm is $\frac{1}{m}\E[\reg]=\sum_{k\neq k^\star}\frac{2\sigma_k^2\ln (mT)}{m\Delta_k}+o(\frac{\sigma_k^2\ln (mT)}{m\Delta_k})$. Therefore, our algorithm reduce the variance by $(1-\rho_k^2)$ in the leading term. One may notice the batched algorithm is in the order of $\ln T$, which is smaller than the non-batched algorithm's order $\ln (mT)$. This is because our assumption requires that the sub-optimality gap $\Delta_k$ is small enough such that the regret of collecting initial batch $m\Delta_k$ is negligible compared with the leading term $\frac{2(1-\rho_k^2)\sigma_k^2\ln T}{\Delta_k}$, and batching may have slight gain in this scenario in terms of the upper bound. The regular bandit setting is still favorable in more general scenario when $\Delta_k$ is not small.
\section{Numerical simulations}
\label{sec: simulation}
In this section, we conduct numerical simulations to evaluate the performance of our MLA-UCB algorithm and compare it with the standard UCB algorithm without surrogate rewards. In Section \ref{sec: simulation Gaussian}, we directly simulate surrogate rewards from the Gaussian model \eqref{eqn: gaussian model}. In Section \ref{sec: ML simulation}, we simulate individual feature vectors and apply actual ML algorithms to produce surrogate rewards. While these surrogate rewards are typically non-Gaussian, meaning that our theory does not strictly apply, we nonetheless observe a sizable reduction in regret. 

\subsection{Non-ML generated Gaussian surrogate rewards}
\label{sec: simulation Gaussian}
We simulate a multi-arm bandit model with $K=5$ arms and a total number of rounds $T=1000$. The true and surrogate rewards from a bivariate Gaussian distribution \eqref{eqn: gaussian model} with the mean of true rewards $[\Delta, 0, 0, 0, 0]$ for some $\Delta > 0$ and the mean of surrogate rewards $[0,0.25,0.5,0.75,1]$. For simplicity, we set the correlation $\rho_k=\rho$ and offline sample size $N_k=N$ to be equal across arms. We vary the values of $\rho, N, \Delta$ in different ranges and compare the cumulative regret of MLA-UCB to the UCB algorithm \cite{cowan2018normal}. The experimental results are shown in Figure \ref{fig: MLA-UCB}.
\begin{figure}[ht]
    \centering
    \begin{subfigure}[t]{0.42\linewidth}
        \includegraphics[width=\linewidth]{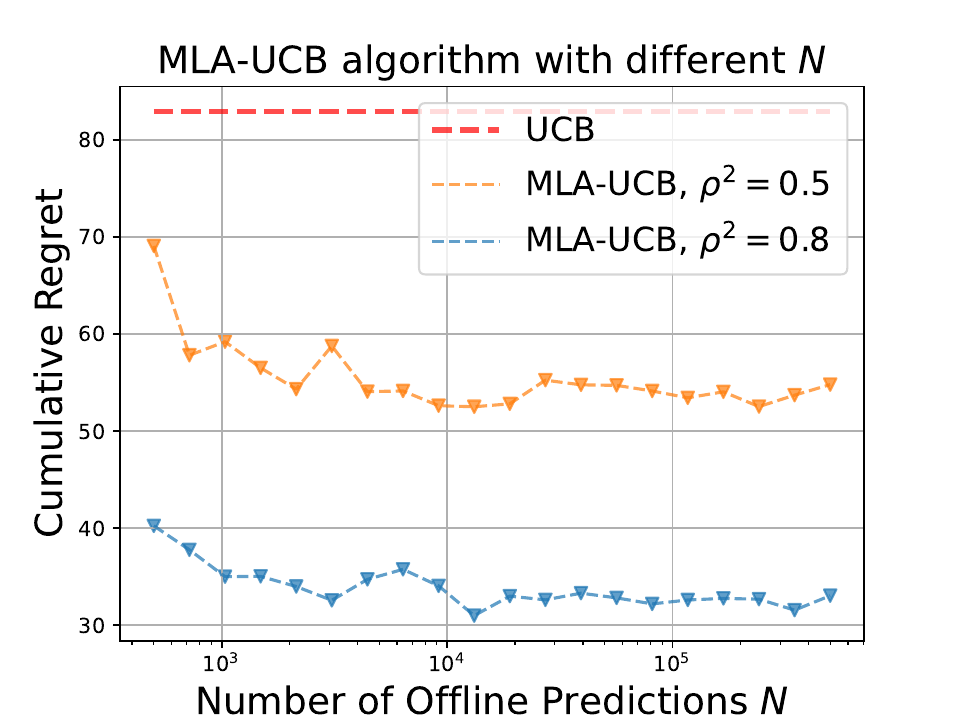}
        \caption{Cumulative regret under different $N$.}
        \label{fig: MLA-UCB N}
    \end{subfigure}
    \begin{subfigure}[t]{0.42\linewidth}
        \includegraphics[width=\linewidth]{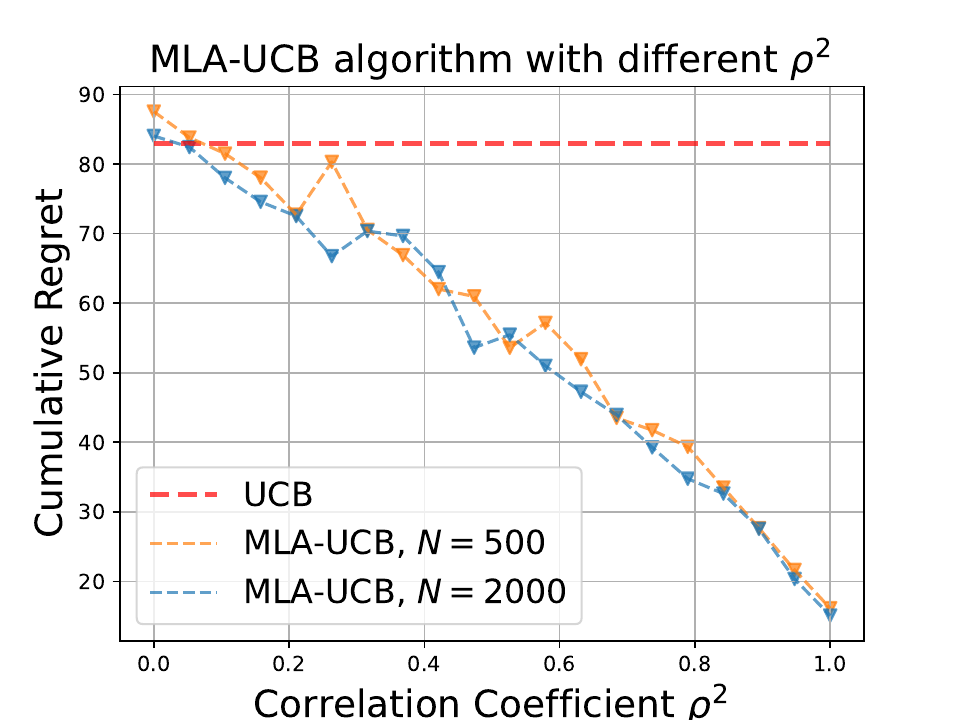}
        \caption{Cumulative regret under different $\rho^2$.}
        \label{fig: MLA-UCB rho}
    \end{subfigure}
    \begin{subfigure}[t]{0.42\linewidth}
        \includegraphics[width=\linewidth]{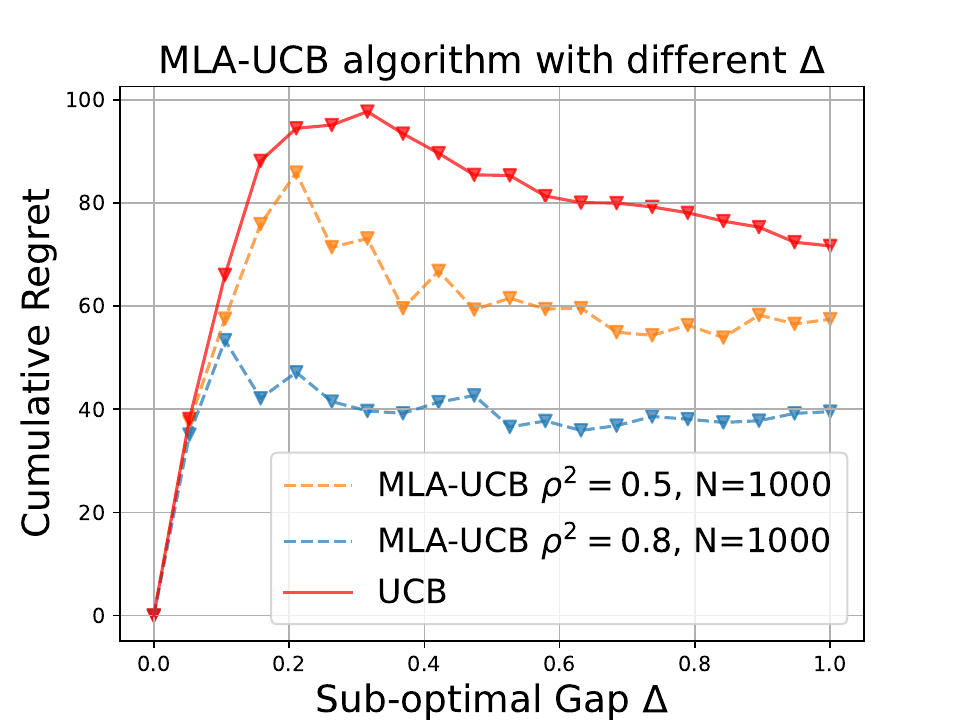}
        \caption{Cumulative regret under different $\Delta$.}
        \label{fig: MLA-UCB delta}
    \end{subfigure}
    \begin{subfigure}[t]{0.42\linewidth}
        \includegraphics[width=\linewidth]{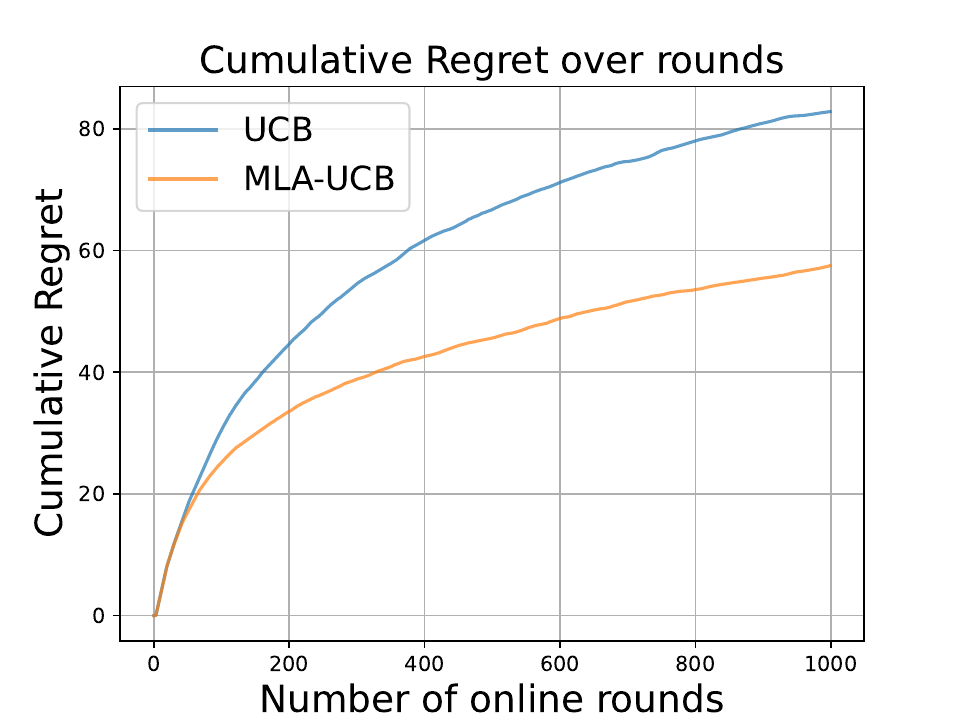}
        \caption{Cumulative regret over rounds.}
        \label{fig: MLA-UCB T}
    \end{subfigure}
    \caption{\textbf{Cumulative regret of the MLA-UCB algorithm (Algorithm \ref{alg:PPUCB}) and the classical UCB algorithm under a Gaussian model.} For Figure \ref{fig: MLA-UCB N} to \ref{fig: MLA-UCB delta}, we report the final cumulative regret at $T=1000$ steps. For Figure \ref{fig: MLA-UCB T}, we choose $\rho^2 = 0.5, \Delta = 0.5, N = 1000$. Each experiment is repeated 100 times to report the average cumulative regrets.}
    \label{fig: MLA-UCB}
\end{figure}

From Figure \ref{fig: MLA-UCB}, we observe that the MLA-UCB algorithm can significantly improve upon the classical UCB algorithm under various settings.  In particular, Figure \ref{fig: MLA-UCB N} demonstrates that as the number of offline predictions $N$ grows, the cumulative regret decreases and finally converges to a limit determined by the correlation $\rho_k$.  For this experiment setting, collecting around $N=1000$ offline predictions for each arm is enough to reach the optimal regret. Figure \ref{fig: MLA-UCB rho} illustrates that as $\rho$ grows, the cumulative regret decays approximately linearly in $\rho^2$. In particular, we observe that as long as $N\geq 500$, the MLA-UCB algorithm can improve upon the baseline UCB algorithm under any $\rho_k^2$ above $0.1$. Figure \ref{fig: MLA-UCB delta} illustrates the behavior of regret w.r.t. the gap $\Delta$, it shows that our algorithm can provide significant improvement when $\Delta$ is relatively large. 
In Figure \ref{fig: MLA-UCB T}, we demonstrate that the shape of the cumulative regret curve of the MLA-UCB algorithm is similar to the UCB algorithm, and it is shrunk towards 0 due to the variance reduction. Overall, it confirms that the MLA-UCB algorithm improves cumulative regret in various settings.

\subsection{ML-generated non-Gaussian surrogate rewards}
\label{sec: ML simulation}
Next, we generate surrogate rewards from actual ML models. We consider the following reward generation process:
\begin{equation*}
    R_k = \sin (w_{k,1}x_1^2+w_{k,2}x_2^2)+\epsilon,
\end{equation*}
where $\mathbf{x} = (x_1, x_2)^\top$ is the feature vector, $\mathbf w_k = (w_{k,1}, w_{k,2})$ is the arm-specific weight parameter, $\epsilon$ is the random noise. The features are assumed to be only visible when generating the predictions and we do not use them directly to make decisions. We generate $X\sim\mathcal{N}(0, I_2)$ and $\epsilon\sim\mathcal{N}(0, \tau^2)$, due to the square and sine operation, the dependency of $R_k$ on the $\mathbf{x}$ are highly nonlinear and hence its distribution is different from a Gaussian distribution. 
\begin{figure}[ht]
    \centering
    \includegraphics[width=0.35\linewidth]{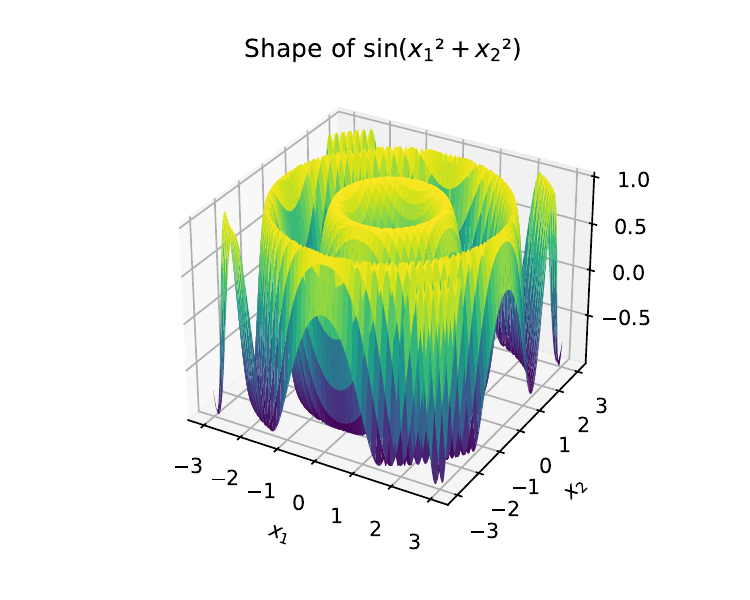}
    \includegraphics[width=0.35\linewidth]{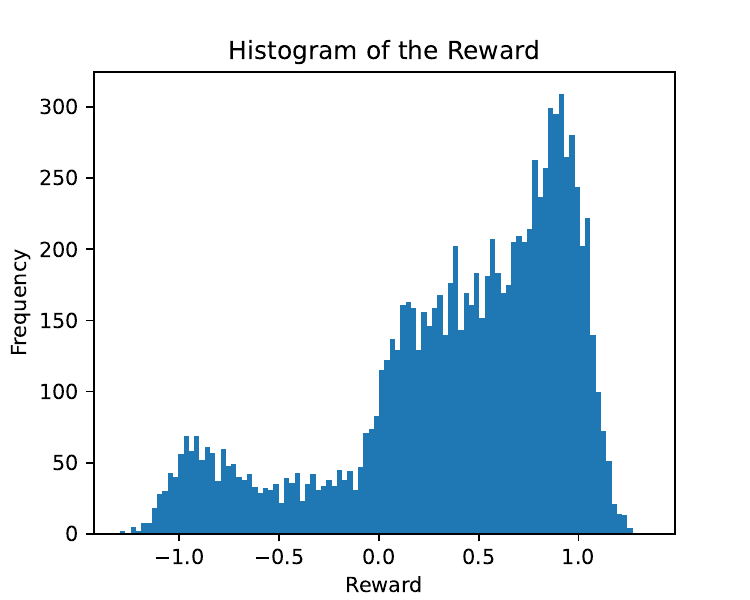}
    \caption{\textbf{The conditional expectation and distribution of the true rewards.} The true reward is generated using $w_{k,1} = w_{k,2}= 1$ and $\tau = 0.1$. The reward function is highly nonlinear and the reward distribution is  different from the Gaussian distribution.}
    \label{fig:shape}
\end{figure}

We consider using four different ML algorithms in MLA-UCB to predict the reward for each arm based on the features: (1) linear regression, (2) support vector regression (SVR), (3) two-layer neural networks, and (4) decision trees. For each arm, we train an individual model to fit its reward. We set $K=5$ in our simulation, and the weights $w_{k,1}, w_{k,2}$ are generated from a standard normal distribution. The correlation of predictions with the true reward of each model is summarized in Table \ref{tab:correlation}, and the experimental results of the MLA-UCB algorithm using these predictions are shown in Figure \ref{fig: ML}. 

\begin{table}[ht]
    \centering
    \begin{tabular}{c|cccc}
        \hline
         Average $\rho_k^2$&  Linear Regression & Support Vector Regression &  Neural Nets &Decision Tree  \\
         \hline
         \hline
         $\tau=0$&0.002 &0.662&0.675&0.849 \\
         \hline
         $\tau=0.2$&0.002 &0.571&0.572&0.655 \\
         \hline
         $\tau=0.4$&0.002 &0.390&0.402&0.314 \\
         \hline
         $\tau=0.6$&0.002 &0.248&0.276&0.157 \\
         \hline
    \end{tabular}
    \caption{\textbf{Average $\rho_k^2$ between predictions and true rewards among all arms.} }
    \label{tab:correlation}
\end{table}
\begin{figure}[ht]
    \centering
    \includegraphics[width=0.4\linewidth]{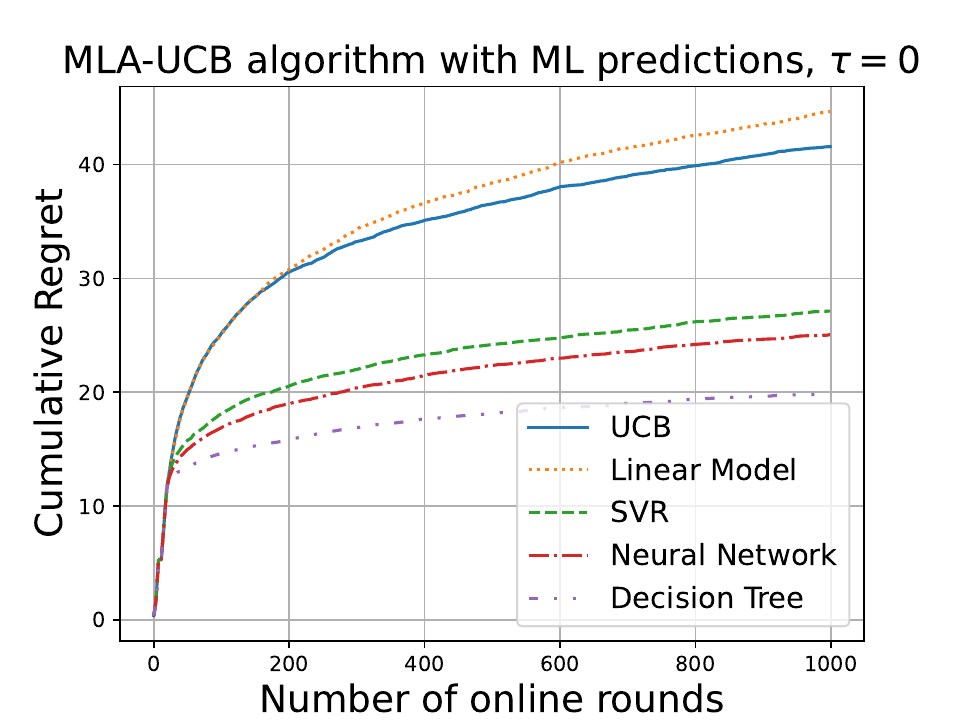}
    \includegraphics[width=0.4\linewidth]{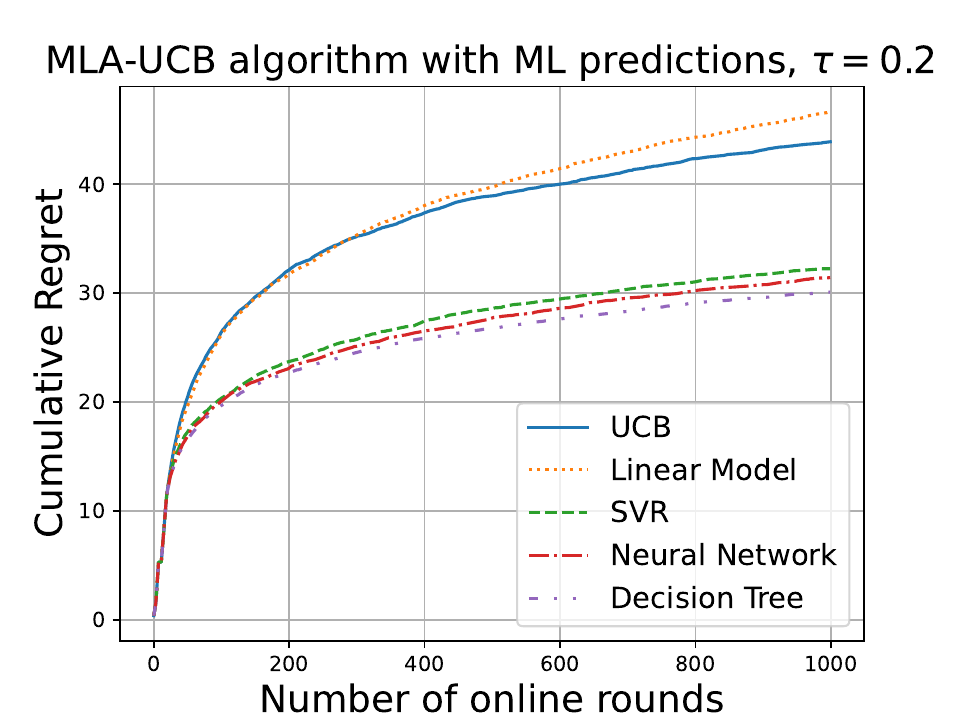}
    \includegraphics[width=0.4\linewidth]{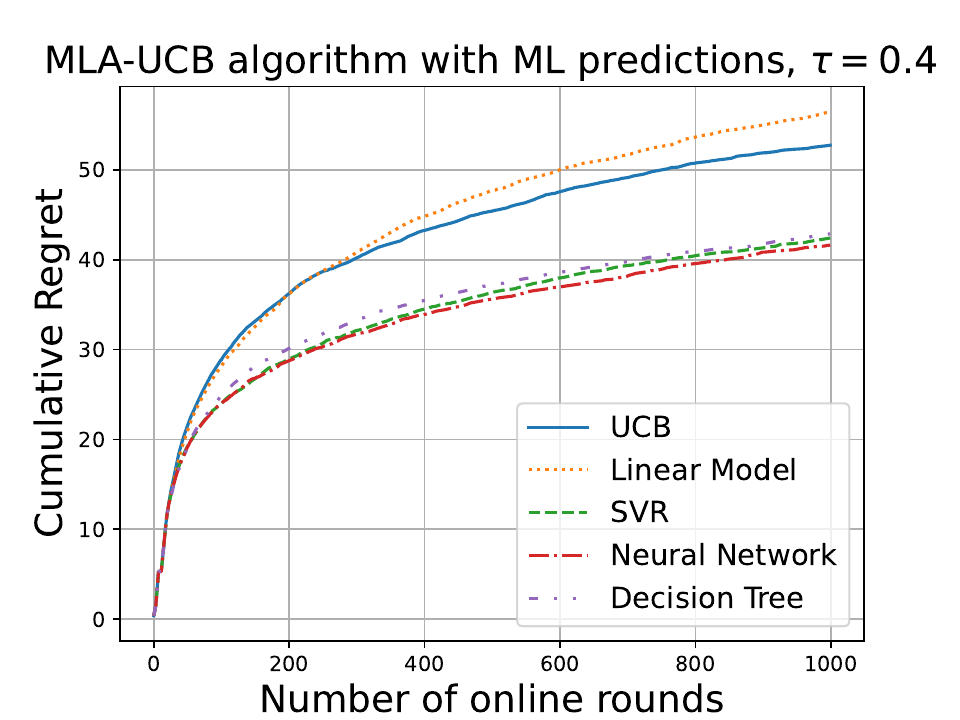}
    \includegraphics[width=0.4\linewidth]{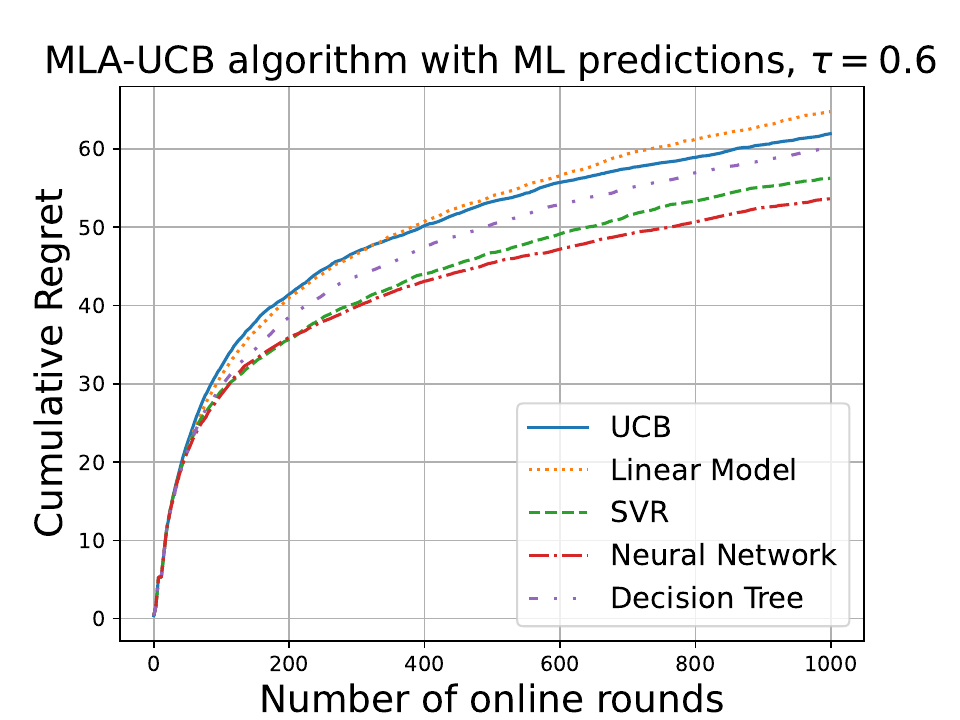}
    \caption{\textbf{Cumulative regret of the MLA-UCB algorithm and the classical UCB algorithm under different settings.} The true correlations of predictions of each setting are shown in Table \ref{tab:correlation}. We repeat each experiment 100 times under the same data-generating model and machine learning model to report the average regret.} 
    \label{fig: ML}
\end{figure}
From Figure \ref{fig: ML}, we observe that the MLA-UCB algorithm can effectively reduce the cumulative regret for predictions from various ML models, even if the Gaussianity assumption \eqref{eqn: gaussian model} does not hold. The regret reduction is more significant with better prediction models in terms of their $\rho_k^2$, while the exact relationship is more complicated than the experiments in Section \ref{sec: simulation Gaussian} since the correlation differs between arms. In particular, we found that although linear models fail to make meaningful predictions since the underlying function is highly nonlinear, the algorithm can still achieve comparable performance to the classical UCB algorithm. Overall, these experiments confirm that our algorithm can improve upon the classical UCB algorithm in various settings. 

\subsection{Batched reward}
In this section, we assess the performance of Algorithm \ref{alg:PPUCB batch} under batched rewards settings. We will use the same reward generation process as in Section \ref{sec: ML simulation}, where the true and surrogate rewards are from a non-Gaussian distribution. We will focus on three ML models: linear models, neural networks, and decision trees. The support vector regression is not considered here due to its inefficiency for large sample sizes. In our simulation, we compare the regret of the MLA-UCB algorithm with a batched version of the UCB algorithm, which pulls the arm 
$$A_t = \arg\max_{k}\left\{ \E_{k,n_{k,t}}[R] + \lb1+{m^{-\frac{1}{3}}}\rb\sqrt{\frac{ 2\hat\sigma_{R,k,n_{k,t}}^2\ln t}{mn_{k,t}}}\right\}$$
in each round. 

The experiment results are shown in Figure \ref{fig: ML batch}. We will choose batch size $m=100$, noise level $\tau=0.2$, average suboptimal gap $\Delta = 0.15$, and a total number of rounds $T=1000$ unless particularly specified. Again, we observe that the MLA-UCB algorithm can effectively reduce the cumulative regret for surrogates from various ML models in the batched reward settings. The regret reduction is more significant with a moderate batch size $m$, a smaller noise level $\tau$ (which implies a higher $\rho_k^2$), and a moderate level of suboptimal gap $\Delta$. In particular, the regret reduction is most significant when $m\Delta^2\approx 1$, which aligns with the local asymptotic regime specified in Assumption \ref{asm: m and T}. Similar to the non-batched settings, our algorithm can still achieve comparable performance to the classical UCB algorithm with surrogates from linear models, which fail to make a meaningful surrogate since the underlying function is highly nonlinear. Overall, these experiments confirm our theory that the MLA-UCB algorithm can reliably reduce the regret in batched reward settings.
\begin{figure}[ht]
    \centering
    \begin{subfigure}[t]{0.42\linewidth}
        \includegraphics[width=\linewidth]{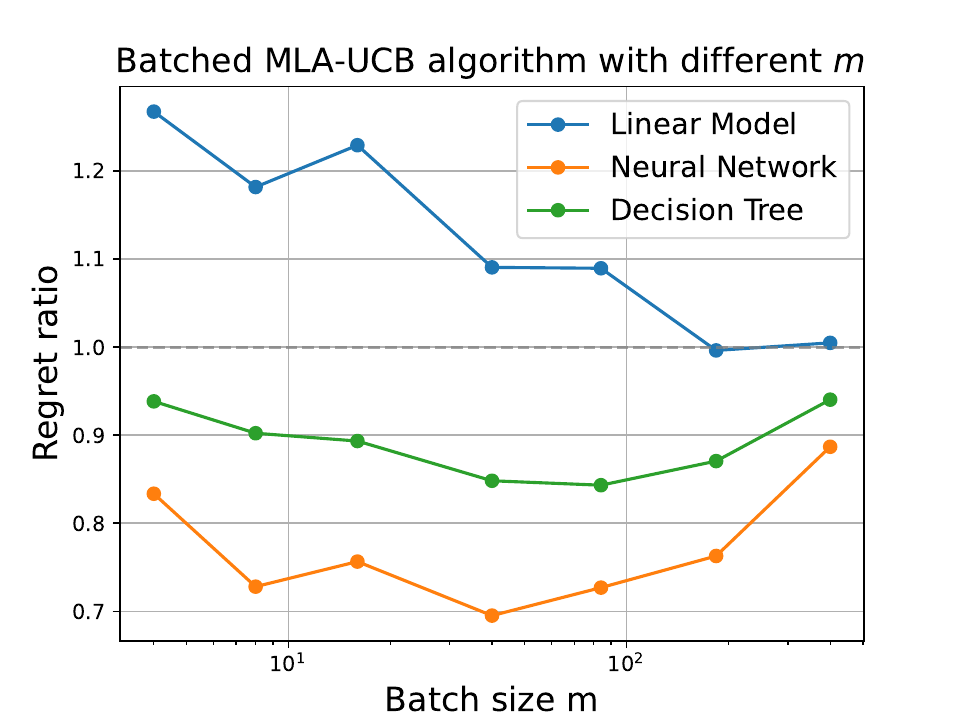}
        \caption{Regret ratio under different $m$}
        \label{fig: batch m}
    \end{subfigure}
    \begin{subfigure}[t]{0.42\linewidth}
        \includegraphics[width=\linewidth]{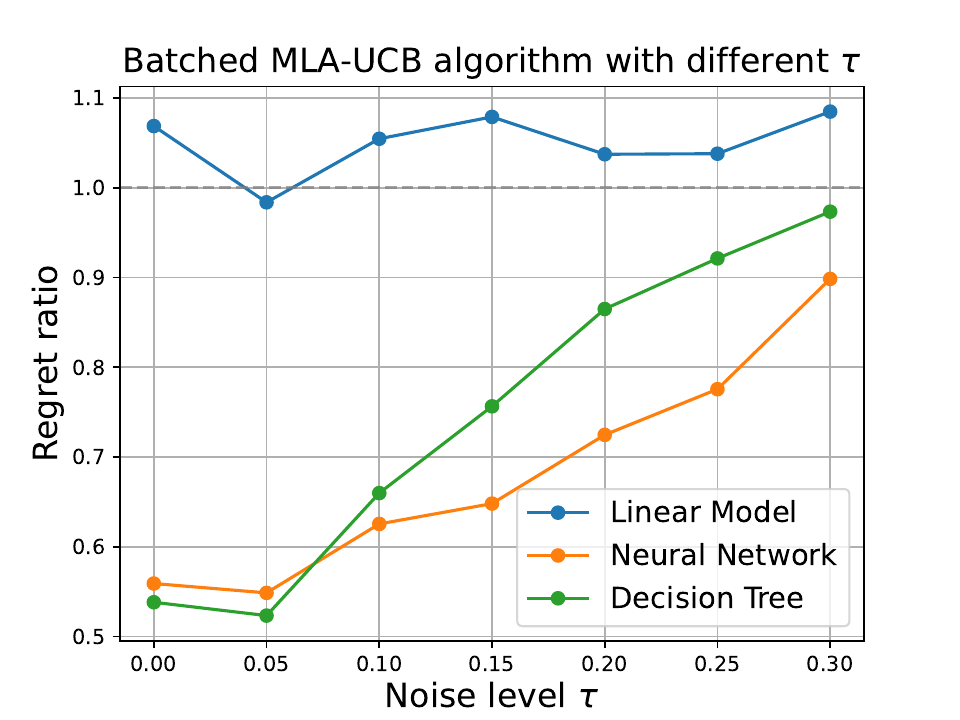}
        \caption{Regret ratio under different $\tau$}
        \label{fig: batch sigma}
    \end{subfigure}
    \begin{subfigure}[t]{0.42\linewidth}
        \includegraphics[width=\linewidth]{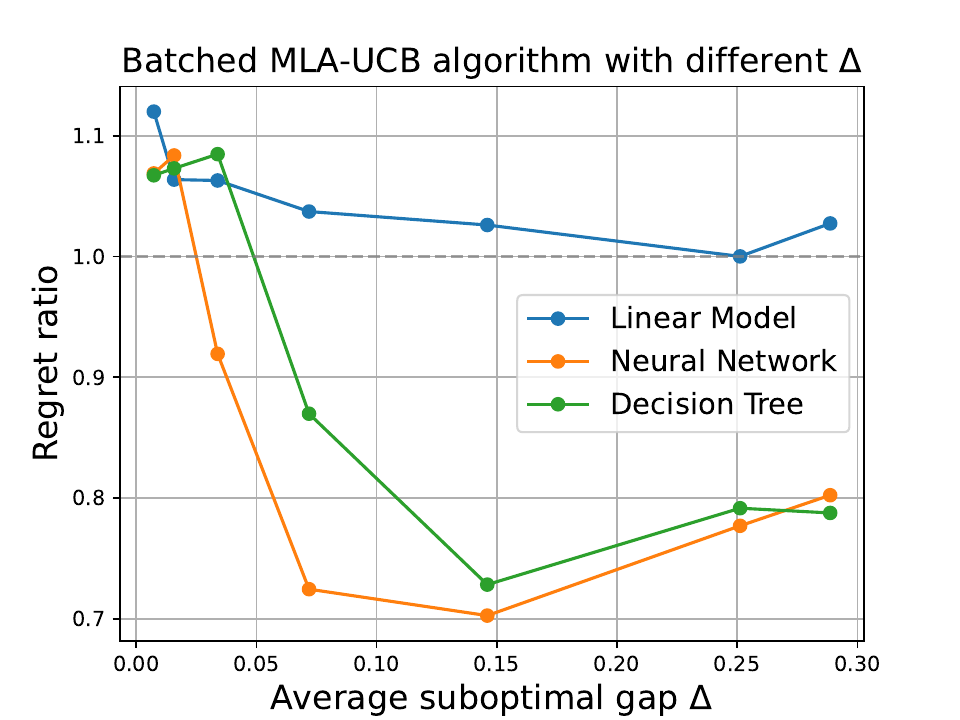}
        \caption{Regret ratio under different $\Delta$}
        \label{fig: batch delta}
    \end{subfigure}
    \begin{subfigure}[t]{0.42\linewidth}
        \includegraphics[width=\linewidth]{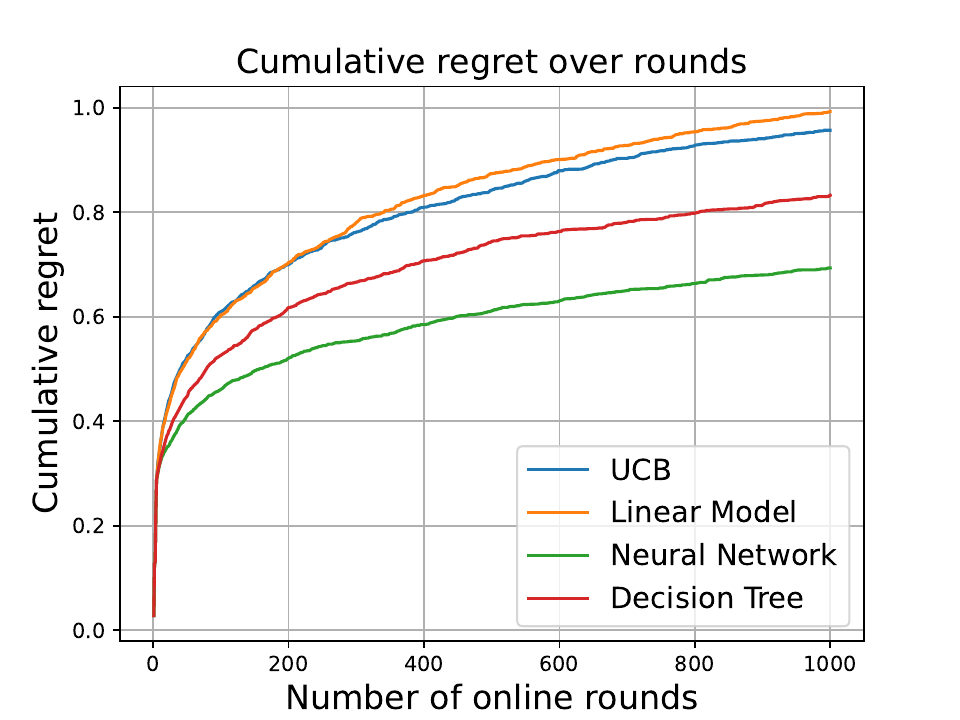}
        \caption{Cumulative regret over rounds}
        \label{fig: batch T}
    \end{subfigure}
    \caption{\textbf{Cumulative regret of the MLA-UCB algorithm under batched reward settings.} In Figure \ref{fig: batch m} to Figure \ref{fig: batch delta}, we report the ratio between the regret of the MLA-UCB algorithm and the UCB algorithm. In Figure \ref{fig: batch delta}, the x-axis represents the  $\Delta_k$ over all suboptimal arms (i.e., $k\neq k^\star$). We repeat each experiment 100 times under the same data-generating model and machine learning model to report the average regret.}
    \label{fig: ML batch}
\end{figure}
\section{Application}
\label{sec: application}
In this section, we demonstrate the advantage of our algorithms on two real-world applications. 
\subsection{Language model selection}
With the development of proprietary large language models (LLMs) (e.g., GPT, Claude, Gemini), practitioners often need to choose the most accurate model for a specific domain. However, evaluating these models on large datasets can be expensive and slow due to API costs and rate limits. Therefore, practitioners can apply the MAB framework to sequentially identify the most powerful model while maximizing the accuracy on all queries. 

Meanwhile, there are also open-source LLMs (e.g., Llama, DeepSeek, Qwen) that can be run locally without API cost. These models are often less accurate compared to proprietary models, but are more convenient and affordable. For a specific question, the correctness of these models can reflect the hardness of the question, which correlates with the performance of proprietary models and hence can be used as surrogates in our MLA-UCB framework. Moreover, practitioners can evaluate the performance of these open-sourced models on a large set of questions without budget concerns, which provides a source of offline surrogates in our settings.

Here, we consider solving a question answering task using the MMLU-pro dataset \citep{wang2024mmlu}, which comprises over 12,000 rigorously curated questions from academic exams and textbooks, spanning 14 diverse domains, including Math, Engineering, History, etc. We choose three proprietary models, GPT 4o, Claude 3.5 sonnet, and Gemini 2.0 flash as the three arms to choose from. And we use Llama 3.1 70B to generate the surrogate. The reward is the binary correctness of the model response, and we average the correctness over a batch of questions. We simulate a sequential decision-making setting where, in each round, a batch of questions is randomly picked from the question pool. The algorithm then selects a proprietary model $k$, and the system observes the average correctness of the proprietary model $R_{k,s}$ (true rewards) and the average correctness of the open-sourced model $\hat R_{k,s}$. In addition, a set of questions is held out, and we evaluate the open-sourced model on them to generate the offline surrogate $\hat R_{k,s}^\off$. The experiment result is shown in Figure \ref{fig: lm}. In particular, the average $\rho_k^2$ of the Llama model is 0.1457, and we observe a proportional reduction of regret over the classical UCB algorithm.
\begin{figure}
    \centering
    \includegraphics[width=0.49\linewidth]{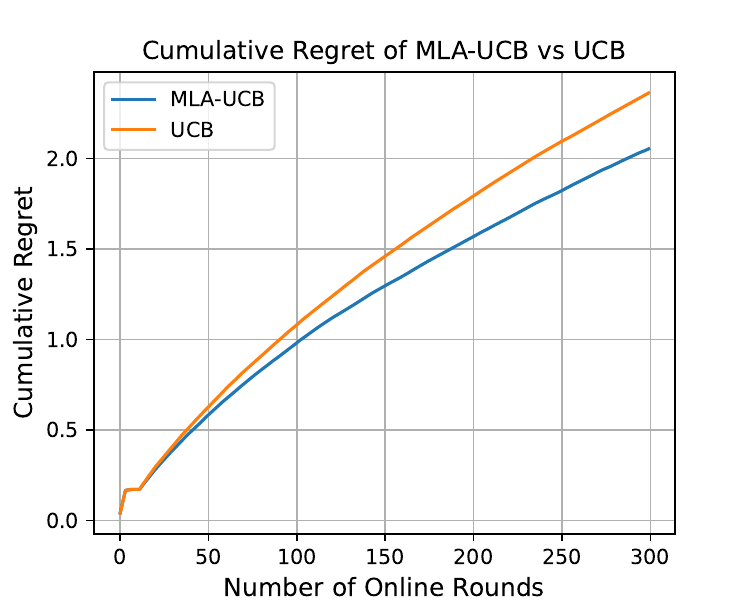}
    \includegraphics[width=0.49\linewidth]{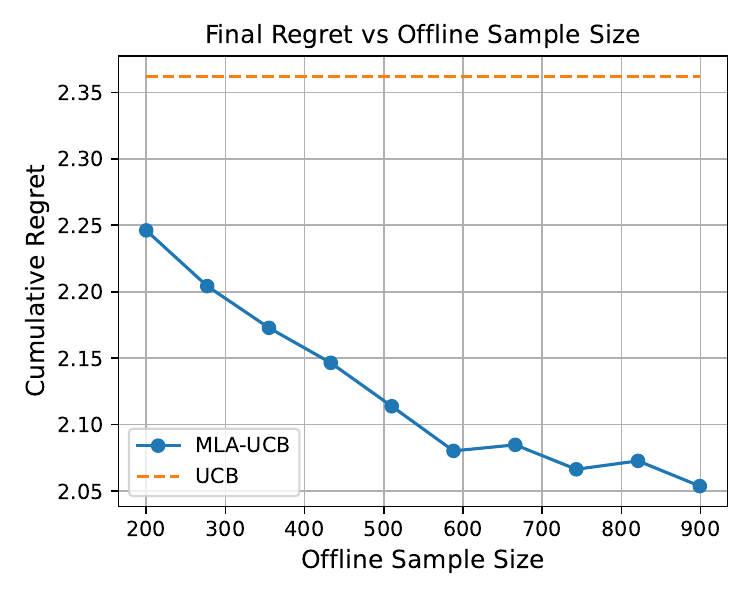}
    \caption{Experimental result for language model selection. In the left plot, we show the cumulative regret over the number of online rounds. In the right plot, we show the cumulative regret against the offline sample size, as long as the offline sample size is relatively large, our algorithm can always reduce the cumulative regret against the classical UCB algorithm. We repeat each experiment 1000 times and report the average regret.}
    \label{fig: lm}
\end{figure}
\subsection{Video recommendation}
We further evaluate our MLA-UCB algorithm on a real-world video recommendation scenario using the KuaiRec dataset \cite{gao2022kuairec}. KuaiRec is a large-scale dataset collected from the video-sharing mobile app Kuaishou. It is suitable for bandit algorithm evaluation because it contains a "fully observed" user-item interaction matrix, where interactions between nearly all users (1,411) and items (3,327) are recorded. 

We formulate this problem as a Multi-Armed Bandit task, where the arms correspond to different videos, and the reward is the watch ratio, which is defined as the ratio between the user's watch time and the video's duration. Our goal is to maximize user engagement through identifying and recommending the most popular video. Meanwhile, both user and item features are available to train machine learning models and construct surrogates. The dataset provides 30 user features, including activity on the platform and encrypted personal information, and 56 item daily features, including the properties of the video, such as duration and video type, and the daily statistics on the platform. 

For our experiment, we will pick 20 videos as candidate arms and use the data on the remaining videos to train a gradient boosting regression tree model and generate the surrogate. This represents the practical setting where we can use the historical data to build prediction models and assist in the decision-making for unseen new items. Due to the difference between the training and testing domains, the surrogates generated by such models are biased and hence unable to preserve the ranking of true rewards, which calls for the need for our MLA-UCB algorithm to eliminate the bias and utilize the information safely. We simulate a sequential decision-making setting where, in each round $t$, a user is randomly picked from the user pool. The algorithm then selects a video $k$, and the system observes the true watch ratio $R_{k,s}$ and the surrogate reward $\hat R_{k,s}$ generated by the model using user and video features. In addition, a set of users is held out and we use them to generate offline surrogate $\hat R_{k,s}^\off$. The experiment result is shown in Figure \ref{fig: video}. In particular, the average $\rho_k^2$ of our tree model is 0.0726, and we observe a proportional reduction of regret over the classical UCB algorithm.

\begin{figure}
    \centering
    \includegraphics[width=0.49\linewidth]{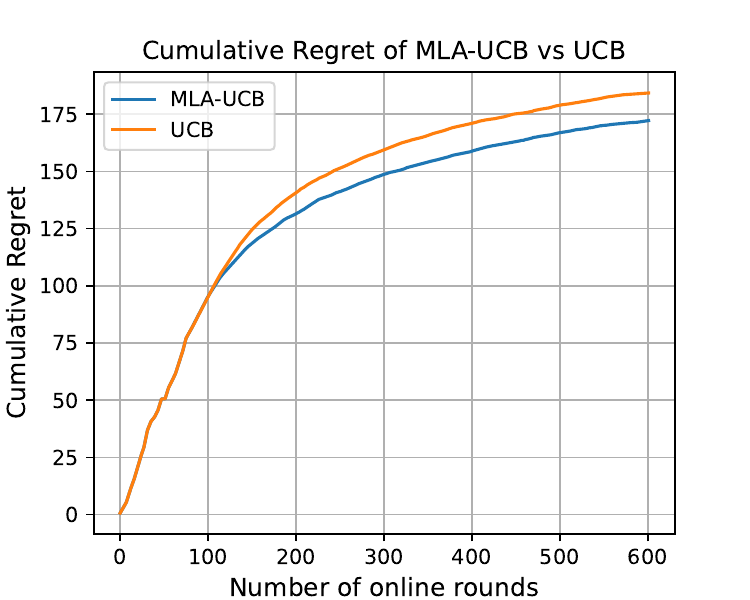}
    \includegraphics[width=0.49\linewidth]{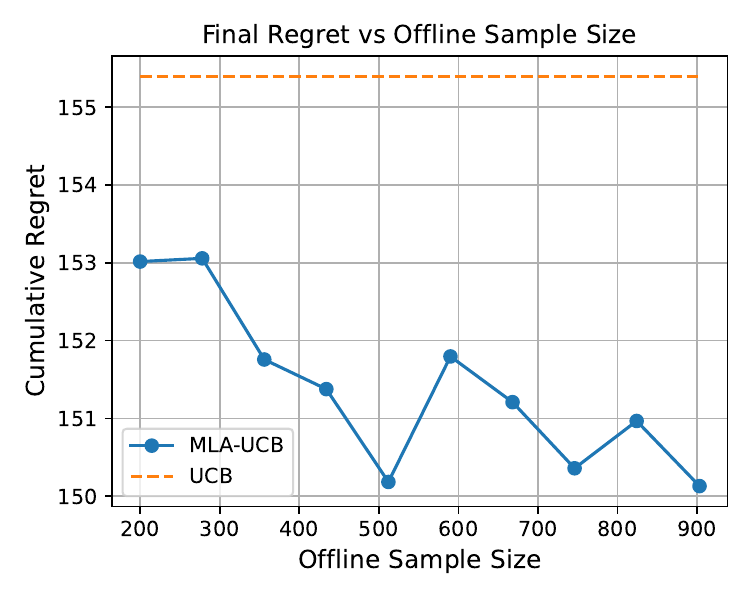}
    \caption{Experimental result for video recommendation. In the left plot, we show the cumulative regret over the number of online rounds. In the right plot, we show the cumulative regret against the offline sample size, as long as the offline sample size is relatively large, our algorithm can always reduce the cumulative regret against the classical UCB algorithm. We repeat each experiment 1000 times and report the average regret.}
    \label{fig: video}
\end{figure}

\section{Conclusion}
In this paper, we introduced a multi-armed bandit setting in which the learner has access to machine learning–generated surrogate rewards observed both offline and online. The surrogates may be arbitrarily biased and need not preserve the arm ranking, but their correlation with the true reward can be exploited for variance reduction. Building on machine-learning-assisted mean estimation, we proposed the MLA-UCB algorithm, which uses variance-adaptive confidence bounds and requires no prior knowledge of reward variances or reward–surrogate covariance. Under a joint Gaussian model with unknown mean and covariance, we established regret upper bounds and matching lower bounds, showing that MLA-UCB is asymptotically optimal and improves the leading regret constant in proportion to $1-\rho_k^2$ for suboptimal arms.

We further extended the approach to batched rewards with potentially non-Gaussian distributions. Using self-normalized moderate deviation bounds, we obtained computable confidence bounds and regret guarantees that preserve the same correlation-driven variance reduction in the leading term. Simulations and case studies on language-model selection and video recommendation illustrate that MLA-UCB can deliver consistent regret reductions with moderate offline surrogate sample sizes, including regimes where surrogate means are uninformative for arm ranking. Although the primary focus of this paper is on the MAB setting, ML-generated surrogate rewards can be useful in more complicated settings, such as contextual bandits. We leave this as an important future direction to explore further.

\bibliographystyle{plain}
\bibliography{main}

\appendix

\section{Proof}
\label{appendix: proof}
This section provides proof of the theoretical results presented in this paper. 
\subsection{Analysis of machine learning-assisted mean estimator}
First, we prove the properties of the machine learning-assisted mean estimator presented in Proposition \ref{prop: intercept} - \ref{prop: CI mean}.
\begin{proof}[Proof of Proposition \ref{prop: intercept}]
    Recall that the definition of the machine learning-assisted mean estimator is 
    $$
    \mukt = \E_{k,n}[R] - \hat \lambda_{k,n}\left(\E_{k,n}[\hat R]  - \E_{k}^\off[\hat R]\right), 
    $$
    where $\hat \lambda_{k,n} = \frac{N_k}{n+N_k}\frac{\hat \sigma_{R,\hat R, k,n}}{\hat \sigma_{\hat R, k,n}^2}$. By the standard regression result, we know that the solution of the ordinary least squares problem \eqref{eqn: OLS equivalence} is: 
    \begin{equation}
    \begin{aligned}
    \label{eqn: alpha beta}
        \hat \beta_{k,n} =& \frac{\widehat \cov(R_k, \hat R_k - \E_{k,n}^\all[\hat R])}{\widehat \var(\hat R_k - \E_{k,n}^\all[\hat R])} = \frac{\hat \sigma_{R,\hat R, k,n}}{\hat \sigma_{\hat R, k,n}^2}\\
        \hat \alpha_{k,n} =& \E_{k,n}[R] - \hat \beta_{k,n}(\E_{k,n}[\hat R] - \E_{k,n}^\all[\hat R])\\
        =&\E_{k,n}[R] - \frac{\hat \sigma_{R,\hat R, k,n}}{\hat \sigma_{\hat R, k,n}^2}\lb \E_{k,n}[\hat R] - \frac{1}{n+N_k}\lb n\E_{k,n}[\hat R] + N_k\E_{k}^\off[\hat R]\rb\rb\\
        =&\E_{k,n}[R] - \frac{N_k}{n+N_k}\frac{\hat \sigma_{R,\hat R, k,n}}{\hat \sigma_{\hat R, k,n}^2}\lb \E_{k,n}[\hat R] - \E_{k}^\off[\hat R]\rb\\
        =&\mukt,
    \end{aligned}
    \end{equation}
    which finishes the proof.
\end{proof}
\begin{proof}[Proof of Proposition \ref{prop: pp mean distribution}]
    From \eqref{eqn: alpha beta}, we know that 
    \begin{equation}
    \begin{aligned}
        \mukt = & \E_{k,n}[R] - \hat \beta_{k,n}(\E_{k,n}[\hat R] - \E_{k,n}^\all[\hat R]).
    \end{aligned}
    \end{equation}
    From the data generation distribution \ref{eqn: gaussian model}, using the conditional distribution of multivariate Gaussian variables, we can decompose the reward as 
    \begin{equation}
    \label{eqn: Gaussian decomposition}
    R_{k,s} = \mu_k + \rho_k \frac{\sigma_{k}}{\tilde\sigma_{k}} (\E_{k,n}^\all[\hat R]- \tilde\mu_k) + \rho_k \frac{\sigma_{k}}{\tilde\sigma_{k}}(\hat R_{k,s} - \E_{k,n}^\all[\hat R]) + \epsilon_{k,s},
\end{equation}
where $\epsilon_{k,s}\sim \mathcal{N}(0,(1-\rho_k^2)\sigma_k^2)$ is independent of $\hat R_{k,s} - \E_{k,n}^\all[\hat R]$.
Then it gives us
\begin{equation}
\begin{aligned}
\label{eqn: decomposition mu}
    &\mukt - \mu_k \\
    =& \rho_k \frac{\sigma_{k}}{\tilde\sigma_{k}} (\E_{k,n}[\hat R]- \tilde\mu_k)+ \frac{1}{n}\sum_{s=1}^{n}\epsilon_{k,s} - \lb\frac{\sum_{s=1}^{n} \epsilon_{k,s}(\hat R_{k,s} - \E_{k,n} [\hat R] )}{\sum_{s=1}^{n} (\hat R_{k,s} - \E_{k,n} [\hat R])^2}+  \rho_k \frac{\sigma_{k}}{\tilde\sigma_{k}} \rb\left(\E_{k,n}[\hat R]  - \E_{k,n}^\all[\hat R]\right) \\
    =&\rho_k \frac{\sigma_{k}}{\tilde\sigma_{k}} (\E_{k,n}^\all[\hat R] - \tilde\mu_k) + \sum_{s=1}^{n}\lb\frac{1}{n}-\frac{(\hat R_{k,s} - \E_{k,n} [\hat R] )}{\sum_{s=1}^{n} (\hat R_{k,s} - \E_{k,n} [\hat R])^2}\left(\E_{k,n}[\hat R]  - \E_{k,n}^\all[\hat R]\right)\rb\epsilon_{k,s}.
\end{aligned}
\end{equation}
By a standard orthogonality argument, $\E_{k,n}^\all[\hat R]$ is independent of $\hat R_{k,s} - \E_{k,n} [\hat R]$ and $\E_{k,n}[\hat R]  - \E_{k,n}^\all[\hat R]$. Therefore, define 
\begin{equation}
\begin{aligned}
    S_{1,k,n}=& \rho_k \frac{\sigma_{k}}{\tilde\sigma_{k}} (\E_{k,n}^\all[\hat R] - \tilde\mu_k)\\
    S_{2,k,n} =& \sum_{s=1}^{n}\lb\frac{1}{n}-\frac{(\hat R_{k,s} - \E_{k,n} [\hat R] )}{\sum_{s=1}^{n} (\hat R_{k,s} - \E_{k,n} [\hat R])^2}\left(\E_{k,n}[\hat R]  - \E_{k,n}^\all[\hat R]\right)\rb\epsilon_{k,s},
\end{aligned}
\end{equation}
then we have $S_{1,k,n}$ and $S_{2,k,n}$ are independent. Then it is straightforward to verify that $S_{1,k,n}\in \Fkt$ and $S_{1,k,n}\sim \mathcal{N}\lb0, \frac{1}{n+N_{k}}\rho_k^2 \sigma_k^2\rb$. For $S_{2,k,n}$, conditional on $\Fkt$, it is a linear combination of $\epsilon_{k,s}$, hence it is a Gaussian random variable with mean 0, and the variance is 
\begin{equation}
\label{eqn: S2 var}
    \begin{aligned}
        \var(S_{2,k,n}|\Fkt) = &(1-\rho_k^2)\sigma_k^2\sum_{s=1}^{n}\lb\frac{1}{n}-\frac{(\hat R_{k,s} - \E_{k,n} [\hat R] )}{\sum_{s=1}^{n} (\hat R_{k,s} - \E_{k,n} [\hat R])^2}\left(\E_{k,n}[\hat R]  - \E_{k,n}^\all[\hat R]\right)\rb^2 \\ 
        = & (1-\rho_k^2)\sigma_k^2\sum_{s=1}^{n}\lb\lb\frac{1}{n}\rb^2+\lb\frac{(\hat R_{k,s} - \E_{k,n} [\hat R] )}{\sum_{s=1}^{n} (\hat R_{k,s} - \E_{k,n} [\hat R])^2}\left(\E_{k,n}[\hat R]  - \E_{k,n}^\all[\hat R]\right)\rb^2\rb\\
        =& (1-\rho_k^2)\sigma_k^2\lb\frac{1}{n} + \frac{(\E_{k,n} [\hat R] - \E_{k,n}^\all [\hat R])^2}{\sum_{s=1}^{n} (\hat R_{k,s} - \E_{k,n}[\hat R])^2}\rb = \frac{1}{n} Z_{k,n} (1-\rho_k^2) \sigma_k^2, 
    \end{aligned}
\end{equation}
which finishes the proof. 
\end{proof}

As a corollary, we can obtain the conditional distribution of $\mukt$ on $Z_{k,n}$, which will be useful in the proof of Theorem \ref{thm: regret bound} later. 
\begin{corollary}
\label{cor: Skt}
Define $\mathcal{S}_{k,n}$ as the sigma field generated by $\{\hat R_{k,s} - \E_{k,n} [\hat R]\}_{s=1}^{n}$ and $\E_{k,n}[\hat R]  - \E_{k,n}^\all[\hat R]$, then $Z_{k,n}\in\mathcal{S}_{k,n} \subset \Fkt$, and 
    $$\mukt - \mu_k | \mathcal{S}_{k,n} \sim \mathcal{N}\lb0, \frac{1}{n+N_{k}}\rho_k^2 \sigma_k^2+\frac{1}{n} Z_{k,n} (1-\rho_k^2) \sigma_k^2\rb$$
\end{corollary}
\begin{proof}
    By definition, $Z_{k,n}\in\mathcal{S}_{k,n} \subset \Fkt$ holds. And by orthogonality $\E_{k,n}^\all[\hat R]$ is independent of $\hat R_{k,s} - \E_{k,n} [\hat R]$ and $\E_{k,n}[\hat R]  - \E_{k,n}^\all[\hat R]$. Hence, $S_{1,k,n}$ is independent of $\mathcal{S}_{k,n}$. Moreover, conditional on $\mathcal{S}_{k,n}$, $S_{2,k,n}$ is still a linear combination of $\epsilon_{k,s}$, it is independent of $S_{1,k,n}$ and the variance is the same as in \eqref{eqn: S2 var}. Hence, we know that $S_{1,k,n}$ and $S_{2,k,n}$ are two independent Gaussian random variables conditional on $\mathcal{S}_{k,n}$, and the variance is $\frac{1}{n+N_{k}}\rho_k^2 \sigma_k^2+\frac{1}{n} Z_{k,n} (1-\rho_k^2)\sigma_k^2$, which finishes the proof. 
\end{proof}

Next we prove the concentration bound in Proposition \ref{prop: CI mean}.
\begin{proof}[Proof of Proposition \ref{prop: CI mean}]
    To prove the concentration bound \eqref{eqn: pp mean concentration}, we handle the two quantities $S_{1,k,n}$ and $S_{2,k,n}$ separately. First, for $S_{1,k,n}$, we have proven that $S_{1,k,n}\sim \N(0, \frac{1}{n+N_{k}}\rho_k^2 \sigma_k^2),$ and by definition, 
    \begin{equation}
    \label{eqn: chi sq Rkt}
         \hat\sigma^2_{R,k,n} = \frac{1}{n - 1} \sum_{s=1}^{n}\lb R_{k,s}  - \E_{k,n}[ R]\rb^2 \sim \frac{\sigma_k^2}{n-1}\chi_{n-1}^2.
    \end{equation}
    Again, by orthogonality, $S_{1,k,n}$ is independent of $\hat\sigma^2_{R,k,n}$, hence
    $$
    \frac{S_{1,k,n}}{\sqrt{\hat\sigma^2_{R,k,n}}}\sim \sqrt\frac{1}{n+N_{k}}\rho_k t_{n-1}.
    $$
    Therefore, we can use the quantile of Student's t-distribution and obtain that
    \begin{equation}
        \label{eqn: bound S1}
        \P\lb S_{1,k,n} \leq - q_{n-1}(\delta)\sqrt\frac{\hat\sigma_{R,k,n}^2}{n+N_{k}}|\rho_k|\rb\leq \delta.
    \end{equation}
    Next, for $S_{2,k,n}$, from \eqref{eqn: S2 var} we know that
    $$
    S_{2,k,n}|\Fkt \sim \N\lb 0, \frac{1}{n} Z_{k,n} (1-\rho_k^2) \sigma_k^2\rb.
    $$
    Moreover, using Proposition \ref{prop: intercept}, we know that $\hat\sigma_{\epsilon,k,n}^2$ is the average sum of residuals for the ordinary least squares problem \eqref{eqn: OLS equivalence}. Therefore, using the linear representation \eqref{eqn: Gaussian decomposition}, the classical results in OLS suggest that 
    \begin{equation}
        \label{eqn: chi sq ekt}
        \hat\sigma_{\epsilon,k,n}^2|\Fkt \sim \frac{(1-\rho_k^2)\sigma_k^2}{n-2}\chi_{n-2}^2,
    \end{equation}
    and $\hat\sigma_{\epsilon,k,n}^2$ is independent of $\hat \alpha_{k,n},\hat \beta_{k,n}$ conditional on $\Fkt$. Notice that $S_{2,k,n} = \mukt -\mu_k - S_{1,k,n}$, $S_{1,k,n} \in \Fkt$, and $\mukt = \hat \alpha_{k,n}$ by Proposition \ref{prop: intercept}, thus $\hat\sigma_{\epsilon,k,n}^2$ and $S_{2,k,n}$ are independent conditional on $\Fkt$. Therefore, we know that 
    $$
    \frac{S_{2,k,n}}{\sqrt{\hat\sigma_{\epsilon,k,n}^2}}\Big|\Fkt \sim \sqrt\frac{Z_{k,n}}{n}t_{n-2},
    $$
    Again, we can use the quantile of Student's t-distribution and obtain that
    \begin{equation}
        \label{eqn: bound S2}
        \P\lb S_{2,k,n} \leq - q_{n-2}(\delta)\sqrt\frac{Z_{k,n}\hat\sigma_{\epsilon,k,n}^2}{n}\Bigg|\Fkt\rb\leq \delta.
    \end{equation}
    In particular, it is well known that for any fixed $c>0$, the tail probability $\P(t_{d}>c)$ is decreasing with respect to $d$ (a proof can be found in Corollary 4.3 in \cite{ghosh1973some}). Hence, we have $q_{n-2}(\delta) \geq q_{n-1}(\delta)$ for any $\delta \in (0,\frac{1}{2})$. Combining \eqref{eqn: bound S1} and \eqref{eqn: bound S2} together, we can obtain that 
    \begin{equation*}
    \begin{aligned}
        &\P\lb \mu_k > \mukt + q_{n - 2}(\delta) \lb\sqrt{\frac{ \hat\sigma_{R,k,n}^2}{n+N_k}}+\sqrt{\frac{\Zkt\hat\sigma_{\epsilon,k,n}^2}{n}}\rb\rb \\
        =&\P\lb  S_{1,k,n}+S_{2,k,n}< -q_{n - 2}(\delta) \lb\sqrt{\frac{ \hat\sigma_{R,k,n}^2}{n+N_k}}+\sqrt{\frac{\Zkt\hat\sigma_{\epsilon,k,n}^2}{n}}\rb  \rb\\
        \leq& \P\lb  S_{1,k,n}< -q_{n - 2}(\delta) \sqrt{\frac{ \hat\sigma_{R,k,n}^2}{n+N_k}}  \rb + \P\lb  S_{2,k,n}< -q_{n - 2}(\delta) \sqrt{\frac{\Zkt\hat\sigma_{\epsilon,k,n}^2}{n}} \rb\\
        \leq &\P\lb  S_{1,k,n}< -q_{n - 1}(\delta) \sqrt{\frac{ \hat\sigma_{R,k,n}^2}{n+N_k}}|\rho_k|  \rb + \E\left[\P\lb  S_{2,k,n}< -q_{n - 2}(\delta) \sqrt{\frac{\Zkt\hat\sigma_{\epsilon,k,n}^2}{n}} \Bigg|\Fkt\rb\right]
        \leq 2\delta,
    \end{aligned}
    \end{equation*}
    which finishes the proof. 
\end{proof}

Next, we prove the upper bound for the quantile of Student's t distribution in Proposition \ref{prop: quantile bound}. It is crucial to analyze the performance of the Algorithm \ref{alg:PPUCB} later. 
\begin{proof}[Proof of Proposition \ref{prop: quantile bound}]
    It suffices to prove that for a random variable $T_d\sim t_d$, the following inequality holds:
    \begin{equation}
        \P\lb T_{d} \geq \sqrt{d(s^{\frac{2}{d-1}}-1)}\rb \leq \frac{1}{2s\sqrt{\ln s}}.
    \end{equation}
    By the definition of the t-distribution, let $X\sim \N(0,1)$ and $Y\sim \chi_d^2$ be two independent random variables, then $\frac{X}{\sqrt{Y/d}} \sim t_d$, and then by the Mills' inequality $\P(X>t)\leq \frac{1}{\sqrt{2\pi}}\frac{1}{t}\exp(-\frac{t^2}{2})$, 
    \begin{equation}
    \label{eqn: t tail}
        \P\lb T_{d} \geq \sqrt{d(s^{\frac{2}{d-1}}-1)}\rb = \P\lb X \geq  \sqrt{Y(s^{\frac{2}{d-1}}-1)}\rb \leq \frac{1}{\sqrt{2\pi(s^{\frac{2}{d-1}}-1)}}\E\left[{Y^{-\frac{1}{2}}e^{-\frac{1}{2}Y(s^{\frac{2}{d-1}}-1)}}\right],
    \end{equation}
    for the expectation on the right hand side, notice that for any $k>0$
    \begin{equation}
    \begin{aligned}
        \E[Y^{-1/2}e^{-\frac{1}{2}kY}] &= \int_{0}^\infty \frac{1}{2^{d/2}\Gamma(d/2)}x^{(d-1)/2-1}e^{-(1+k)x/2}dx\\
        &=\frac{1}{(1+k)^{\frac{d-1}{2}}}\int_{0}^\infty \frac{1}{2^{d/2}\Gamma(d/2)}x^{(d-1)/2-1}e^{-x/2}dx \\
        &= \frac{1}{\sqrt{2}(1+k)^{\frac{d-1}{2}}}\frac{\Gamma((d-1)/2)}{\Gamma(d/2)},
    \end{aligned}
    \end{equation}
    take $k = s^{\frac{2}{d-1}}-1$, we have $\E\left[{Y^{-\frac{1}{2}}e^{-\frac{1}{2}Y(s^{\frac{2}{d-1}}-1)}}\right] = \frac{1}{\sqrt{2}s}\frac{\Gamma((d-1)/2)}{\Gamma(d/2)}$. Moreover, it is easy to prove by induction that 
    $$
    \frac{\Gamma((d-1)/2)}{\Gamma(d/2)} \leq \sqrt{\frac{2\pi}{d}}, \quad \forall d \geq 2, d\in \mathbb{N}^+,  
    $$
    therefore we can derive from \eqref{eqn: t tail} that 
    $$
    \P\lb T_{d} \geq \sqrt{d(s^{\frac{2}{d-1}}-1)}\rb \leq \frac{1}{\sqrt{2d}s\sqrt{s^{\frac{2}{d-1}}-1}} \leq \frac{1}{\sqrt{2d}s\sqrt{\frac{2\ln s}{d-1}}} \leq \frac{1}{2s\sqrt{\ln s}}
    $$
\end{proof}
In addition, we will use the following technical lemma on the tail bound of $\chi^2$ distribution from Lemma \ref{lem: chi square}.
\begin{lemma}[Proposition 8 from \cite{cowan2018normal}]
\label{lem: chi square}
    If a random variable $X\sim \chi^2_{d}$, then for any $\delta>0$, we have 
    $$
    \P(X > d(1+\delta)) \leq \lb e^{-\delta}(1+\delta)\rb^{d/2}
    $$
\end{lemma}
Using Lemma \ref{lem: chi square}, we can prove the tail bound for $\hat\sigma_{R,k,n}^2$ and $\hat\sigma_{\epsilon,k,n}^2$.
\begin{lemma}
    \label{lem: tail sigma}
    For any $\delta>0, n\geq 3$, the following inequalities hold:
    \begin{equation*}
        \begin{aligned}
            &\P(\hat \sigma_{R, k,n}^2 > (1+\delta)\sigma_k^2) \leq  \lb e^{-\delta}(1+\delta)\rb^{\frac{n-1}{2}}, \\
            &\P(\hat \sigma_{\epsilon, k,n}^2 > (1+\delta)(1-\rho_k^2)\sigma_k^2) \leq \lb e^{-\delta}(1+\delta)\rb^{\frac{n-2}{2}}.
        \end{aligned}
    \end{equation*}
\end{lemma}
\begin{proof}
    Since we have proven that $\hat \sigma^2_{R,k,n} \sim \frac{\sigma_k^2}{n-1}\chi_{n-1}^2$ and $\hat\sigma_{\epsilon,k,n}^2|\Fkt \sim \frac{(1-\rho_k^2)\sigma_k^2}{n-2}\chi_{n-2}^2$ in \eqref{eqn: chi sq Rkt} and \eqref{eqn: chi sq ekt}, directly apply Lemma \ref{lem: chi square} on $\hat\sigma_{R,k,n}^2$ and $\hat\sigma_{\epsilon,k,n}^2$ will finish the proof. 
\end{proof}
Similarly, we can prove the tail bound for $Z_{k,n}$. 
\begin{lemma}
    \label{lem: tail Z}
    For any $\delta>0$, if $n\geq 6$, the following inequality holds:
    $$
    \P(Z_{k,n}\geq 1+\delta) \leq  \frac{3}{\delta^2}\frac{1}{(n-3)(n-5)}
    $$
\end{lemma}
\begin{proof}
    Recall that by definition \eqref{eqn: Zkt}, 
    $$
    Z_{k,n} = 1+  \frac{n(\E_{k,n} [\hat R] - \E_{k,n}^\all [\hat R])^2}{\sum_{s=1}^{n} (\hat R_{k,s} - \E_{k,n}[\hat R])^2}.
    $$
    Notice that 
    $$
    \E_{k,n} [\hat R] - \E_{k,n}^\all [\hat R] = \lb\frac{1}{n} - \frac{1}{n+N_{k}}\rb \sum_{s=1}^{n}\hat R_{k,s}- \frac{1}{n+N_{k}}\sum_{s=1}^{N_{k}}\hat R_{k,s}^\off,
    $$
    thus $\E_{k,n} [\hat R] - \E_{k,n}^\all [\hat R]$ is a Gaussian random variable with mean $0$ and variance
    $$
    \var(\E_{k,n} [\hat R] - \E_{k,n}^\all [\hat R]) = \lb\frac{N_k^2}{n(n+N_{k})^2} + \frac{N_k}{(n+N_{k})^2}\rb\tilde\sigma_k^2=\frac{N_k\tilde\sigma_k^2}{n(n+N_{k})}.
    $$
    On the other hand, we know that 
    $$
    \sum_{s=1}^{n} (\hat R_{k,s} - \E_{k,n}[\hat R])^2 \sim \tilde\sigma_k^2\chi_{n-1}^2.
    $$
    Moreover, 
    \begin{equation*}
        \begin{aligned}
        &\cov(\hat R_{k,s} - \E_{k,n}[\hat R], \E_{k,n} [\hat R] - \E_{k,n}^\all [\hat R]) \\
        =&\cov(\hat R_{k,s}, \E_{k,n} [\hat R] - \E_{k,n}^\all [\hat R]) - \cov(\E_{k,n}[\hat R], \E_{k,n} [\hat R] - \E_{k,n}^\all [\hat R])\\
        =&\cov(\hat R_{k,s}, \E_{k,n} [\hat R] - \E_{k,n}^\all [\hat R]) - \frac{1}{n}\sum_{j=1}^{n} \cov(\hat R_{k,j}, \E_{k,n} [\hat R] - \E_{k,n}^\all [\hat R]) = 0.\\
        \end{aligned} 
    \end{equation*}
    Therefore, $(\E_{k,n} [\hat R] - \E_{k,n}^\all [\hat R])^2$ and $\sum_{s=1}^{n} (\hat R_{k,s} - \E_{k,n}[\hat R])^2$ are independent random variables, and thus 
    $$
    Z_{k,n} - 1 = \frac{n(\E_{k,n} [\hat R] - \E_{k,n}^\all [\hat R])^2}{\sum_{s=1}^{n} (\hat R_{k,s} - \E_{k,n}[\hat R])^2} \sim \frac{N_k}{(n+N_{k})(n-1)}t_{n-1}^2. 
    $$
    Using Markov's inequality, we can conclude that 
    \begin{equation*}
        \begin{aligned}
            &\P(Z_{k,n}\geq 1+\delta) = \P\lb\frac{N_k}{(n+N_{k})(n-1)}t_{n-1}^2 \geq \delta \rb\\
            \leq& \frac{1}{\delta^2}\lb\frac{N_k}{(n+N_{k})(n-1)}\rb^2\E[t_{n-1}^4] = \frac{3}{\delta^2}\lb\frac{N_k}{n+N_{k}}\rb^2\frac{1}{(n-3)(n-5)}\\
            \leq & \frac{3}{\delta^2}\frac{1}{(n-3)(n-5)}
        \end{aligned}
    \end{equation*}
    which finishes the proof. 
\end{proof}
\subsection{Proof of regret upper bound}
To analyze the regret, we need a stronger result than Proposition \ref{prop: CI mean} to provide all-time validity of the confidence bound, since the online sample size $n_{k,t}$ is a random variable. We will achieve this through a shifted concentration lemma. For each round $t\ge 4K+1$, define
\[
\delta_t:=\frac1{2t\sqrt{\ln t}},
\quad
q_{k,n}(t):=q_{n-2}(\delta_t),
\]
and the fixed-count index
\[
\hat U_{k,n}(t):=\hat\mu^{\mla}_{k,n}+\hat B_{k,n}(t),
\]
where
\[
\hat B_{k,n}(t):=q_{k,n}(t)\left(\sqrt{\frac{\hat\sigma_{R,k,n}^2}{n+N_k}}+\sqrt{\frac{Z_{k,n}\hat\sigma_{\epsilon,k,n}^2}{n}}\right).
\]
When $n=n_{k,t}$, this is exactly the index used by Algorithm \ref{alg:PPUCB}.
\begin{lemma}
\label{lem:gaussian-shifted-kn}
Fix an arm $k\in[K]$, a deterministic online pull count $n\ge 4$, a round index $t\ge n$, a shift $h>0$.
Then
\begin{align}
\P\bigl(\hat U_{k,n}(t)+h<\mu_k\bigr)
&\le \delta_t\exp\!\left(-\frac{(n+N_k)h^2}{8\sigma_k^2}\right)
+\delta_t\exp\!\left(-\frac{n h^2}{16\sigma_k^2}\right)
+\frac{25\,\delta_t}{n^2}.
\label{eq:gaussian-shifted-kn-bound}
\end{align}
\end{lemma}

\begin{proof}
From Proposition \ref{prop: pp mean distribution}, we can write
\[
\hat\mu^{\mla}_{k,n}-\mu_k=S_{1,k,n}+S_{2,k,n},
\]
where $S_{1,k,n}$ and $S_{2,k,n}$ are independent with distribution
\[
S_{1,k,n}\sim \N\!\left(0,\frac{\rho_k^2\sigma_k^2}{n+N_k}\right),
\quad
S_{2,k,n}\mid \mathcal F_{k,n}\sim \N\!\left(0,\frac{Z_{k,n}(1-\rho_k^2)\sigma_k^2}{n}\right).
\]
Set
\[
A_{k,n}:=\sqrt{\frac{\hat\sigma_{R,k,n}^2}{n+N_k}},
\quad
B_{k,n}:=\sqrt{\frac{Z_{k,n}\hat\sigma_{\epsilon,k,n}^2}{n}}.
\]
Then
\begin{align*}
\{\hat U_{k,n}(t)<\mu_k-h\}
&=\{S_{1,k,n}+S_{2,k,n}<-q_{k,n}(t)(A_{k,n}+B_{k,n})-h\}
\\
&\subseteq \Bigl\{S_{1,k,n}<-q_{k,n}(t)A_{k,n}-\frac h2\Bigr\}
\cup \Bigl\{S_{2,k,n}<-q_{k,n}(t)B_{k,n}-\frac h2\Bigr\}.
\end{align*}
We bound the two terms separately.

\textbf{Step 1: the $S_{1,k,n}$ term.}
If $\rho_k=0$, then $S_{1,k,n}\equiv 0$ and there is nothing to prove.
Assume $\rho_k\neq 0$ and define
\[
\tau_{1,k,n}:=\frac{|\rho_k|\sigma_k}{\sqrt{n+N_k}}.
\]
Since $\hat\sigma_{R,k,n}^2$ is independent of $S_{1,k,n}$, then conditionally on $A_{k,n}$,
\[
\P\Bigl(S_{1,k,n}<-q_{k,n}(t)A_{k,n}-\frac h2\,\Big|\,A_{k,n}\Bigr)
=\Phi\!\bigl(-(x+u)\bigr),
\]
where
\[
x:=\frac{q_{k,n}(t)A_{k,n}}{\tau_{1,k,n}},
\quad
u:=\frac{h}{2\tau_{1,k,n}}.
\]
For any $x,u\ge 0$,
\begin{equation}
\Phi(-(x+u))
=\int_{x+u}^{\infty}\frac{e^{-z^2/2}}{\sqrt{2\pi}}\,dz
\le e^{-u^2/2}\int_x^{\infty}\frac{e^{-uz}e^{-z^2/2}}{\sqrt{2\pi}}\,dz
\le e^{-u^2/2}\Phi(-x).
\label{eq:gaussian-tail-shift-basic}
\end{equation}
Therefore,
\[
\P\Bigl(S_{1,k,n}<-q_{k,n}(t)A_{k,n}-\frac h2\Bigr)
\le e^{-u^2/2}\P\bigl(S_{1,k,n}<-q_{k,n}(t)A_{k,n}\bigr).
\]
Now exactly as in the proof of Proposition \ref{prop: CI mean},
\[
\P\bigl(S_{1,k,n}< -q_{n-1}(\delta_t)|\rho_k|A_{k,n}\bigr)\le \delta_t.
\]
Since $q_{k,n}(t)=q_{n-2}(\delta_t)\ge q_{n-1}(\delta_t)$ and $|\rho_k|\le 1$,
\[
q_{k,n}(t)A_{k,n}\ge q_{n-1}(\delta_t)|\rho_k|A_{k,n},
\]
so
\[
\P\bigl(S_{1,k,n}< -q_{k,n}(t)A_{k,n}\bigr)\le \delta_t.
\]
Finally,
\[
e^{-u^2/2}=\exp\!\left(-\frac{(n+N_k)h^2}{8\rho_k^2\sigma_k^2}\right)
\le \exp\!\left(-\frac{(n+N_k)h^2}{8\sigma_k^2}\right),
\]
which yields
\begin{equation}
\P\Bigl(S_{1,k,n}<-q_{k,n}(t)A_{k,n}-\frac h2\Bigr)
\le \delta_t\exp\!\left(-\frac{(n+N_k)h^2}{8\sigma_k^2}\right).
\label{eq:gaussian-S1-shifted-final}
\end{equation}

\textbf{Step 2: the $S_{2,k,n}$ term.}
Condition on $\mathcal F_{k,n}$.
Define
\[
\tau_{2,k,n}:=\sqrt{\frac{Z_{k,n}(1-\rho_k^2)\sigma_k^2}{n}},
\quad
V_{k,n}:=\sqrt{\frac{\hat\sigma_{\epsilon,k,n}^2}{(1-\rho_k^2)\sigma_k^2}}.
\]
Conditionally on $(\mathcal F_{k,n},V_{k,n})$, the random variable $S_{2,k,n}$ is Gaussian with variance $\tau_{2,k,n}^2$, and $B_{k,n}=\tau_{2,k,n}V_{k,n}$. Hence
\[
\P\Bigl(S_{2,k,n}<-q_{k,n}(t)B_{k,n}-\frac h2\,\Big|\,\mathcal F_{k,n},V_{k,n}\Bigr)
=\Phi\!\bigl(-(x+u)\bigr),
\]
where
\[
x:=q_{k,n}(t)V_{k,n},
\quad
u:=\frac{h}{2\tau_{2,k,n}}=\frac{h\sqrt n}{2\sigma_k\sqrt{Z_{k,n}(1-\rho_k^2)}}.
\]
Applying \eqref{eq:gaussian-tail-shift-basic} once again gives
\[
\Phi\!\bigl(-(x+u)\bigr)\le e^{-u^2/2}\Phi(-x).
\]
Taking conditional expectation with respect to $V_{k,n}$ and using the same $t$-distribution argument as in Proposition \ref{prop: CI mean} yields
\begin{align}
\P\Bigl(S_{2,k,n}<-q_{k,n}(t)B_{k,n}-\frac h2\,\Big|\,\mathcal F_{k,n}\Bigr)
&\le \delta_t\exp\!\left(-\frac{n h^2}{8Z_{k,n}(1-\rho_k^2)\sigma_k^2}\right)
\nonumber\\
&\le \delta_t\exp\!\left(-\frac{n h^2}{8Z_{k,n}\sigma_k^2}\right).
\label{eq:gaussian-S2-cond}
\end{align}
Now split according to $\{Z_{k,n}\le 2\}$ and $\{Z_{k,n}>2\}$.
On $\{Z_{k,n}\le 2\}$, \eqref{eq:gaussian-S2-cond} gives
\[
\P\Bigl(S_{2,k,n}<-q_{k,n}(t)B_{k,n}-\frac h2,\ Z_{k,n}\le 2\Bigr)
\le \delta_t\exp\!\left(-\frac{n h^2}{16\sigma_k^2}\right).
\]
On $\{Z_{k,n}>2\}$, the conditional probability is at most $\delta_t$, hence
\[
\P\Bigl(S_{2,k,n}<-q_{k,n}(t)B_{k,n}-\frac h2,\ Z_{k,n}>2\Bigr)
\le \delta_t\P(Z_{k,n}>2).
\]
By Lemma \ref{lem: tail Z}, for $n\ge 6$,
\[
\P(Z_{k,n}>2)\le \frac{3}{(n-3)(n-5)}\le \frac{25}{n^2}.
\]
For $n=4,5$, the bound $\P(Z_{k,n}>2)\le 1\le 25/n^2$ is trivial. Therefore, for all $n\ge 4$,
\begin{equation}
\P\Bigl(S_{2,k,n}<-q_{k,n}(t)B_{k,n}-\frac h2\Bigr)
\le \delta_t\exp\!\left(-\frac{n h^2}{16\sigma_k^2}\right)+\frac{25\,\delta_t}{n^2}.
\label{eq:gaussian-S2-shifted-final}
\end{equation}
Combining \eqref{eq:gaussian-S1-shifted-final} and \eqref{eq:gaussian-S2-shifted-final} proves the lemma.
\end{proof}

Now we are ready to prove the regret bound of the MLA-UCB algorithm in Theorem \ref{thm: regret bound}. We will start with the following generalized version of regret bound. 
\begin{theorem}
\label{thm: MLA-UCB}
Under the Gaussian data generation model \eqref{eqn: gaussian model}, for any $\epsilon\in(0,1)$ and $T\geq 4K$, if the sample size of offline predictions satisfies 
\begin{equation}
    N_k\geq \frac{1}{\delta_k}\lb\frac{2\ln T}{\ln\lb 1+\frac{\Delta_k^2}{4\sigma_{k}^2} \frac{(1-\epsilon)^2}{1+\epsilon}\rb}+3\rb, \forall k \neq k^\star
\end{equation}
for some $\delta_k<1$, then the expected regret of Algorithm \ref{alg:PPUCB} can be bounded by:
\begin{equation}
\label{eqn: regret bound}
    \begin{split}
        \E[\reg] \leq \sum_{k\neq k^\star}\lb\frac{2\ln T}{\ln\lb 1+\frac{\Delta_k^2}{\sigma_{k}^2}\frac{(1-\epsilon)^2}{(1+\epsilon)}\frac{1}{(\sqrt{1-\rho_k^2}+\sqrt{\delta_k})^2}\rb}+C\lb
1+\frac{\sigma_{k^\star}^2}{\epsilon^2\Delta_k^2}\rb\sqrt{\ln T}+\frac{C}{\epsilon^2}\lb1+\frac{\sigma_{k}^2}{\Delta_k^2}\rb\rb\Delta_k,
    \end{split}
\end{equation}
where $C>0$ is a universal constant.
\end{theorem}
\begin{proof}[Proof of Theorem \ref{thm: MLA-UCB}]
    For any $\epsilon\in(0,1)$, define $\tilde \epsilon_k = {\Delta_k\epsilon}$. Define the following events 
    \begin{equation}
    \begin{aligned}
        \mathcal{A}_{k,t} =& \{A_{t} = k\} \subset \{\mukntk +\hat B_{k,n_{k,t}}(t)\geq  \hat \mu_{k^\star,n_{k^\star,t}}^\mla +\hat B_{k^\star,n_{k^\star,t}}(t) \}, \\
        \mathcal{B}_{k,t} =& \{\mu_{k} +\hat B_{k,n_{k,t}}(t) + \tilde \epsilon_k\geq\mu^\star\} =  \{\Delta_k (1-\epsilon) \leq \hat B_{k,n_{k,t}}(t)\}, \\
        \mathcal{D}_{k,t} =& \{\mukntk \leq \mu_k+\tilde \epsilon_k / 2 \},\\ 
        \mathcal{E}_{k,n}^1 = &\{ \hat\sigma_{R, k,n}^2 \leq (1+\epsilon)\sigma_k^2 \}, \\ 
        \mathcal{E}_{k,n}^2 = &\{ \hat\sigma_{\epsilon, k, n}^2 \leq (1+\epsilon/3)(1-\rho_k^2)\sigma_k^2 \}, \\
        \mathcal{E}_{k,n}^3 = &\{Z_{k,n} \leq (1+\epsilon/3)\}, \\
        \mathcal{E}_{k,t} = & \cap_{i=1}^3\mathcal{E}_{k,n_{k,t}}^i. 
    \end{aligned}
    \end{equation}
    Then we can express the expected regret as 
    \begin{equation}
        \E[\operatorname{Reg_T}] = \E\sum_{k\neq k^\star}n_{k,T}\Delta_k = \sum_{k\neq k^\star}\Delta_k \E\sum_{t=1}^T \I_{\mathcal{A}_{k,t}}. 
    \end{equation}
    Furthermore, we can decompose the expectation as follows
    \begin{equation}
    \label{eqn: regret decomposition unknown var}
    \begin{aligned}
        \E\sum_{t=1}^T\I_{\mathcal{A}_{k,t}} =& \E\sum_{t=4K}^T\I_{\mathcal{A}_{k,t}\cap \mathcal{B}_{k,t}\cap \mathcal{D}_{k,t}\cap \mathcal{E}_{k,t} } +\E\sum_{t=4K}^T\I_{\mathcal{A}_{k,t}\cap \mathcal{B}_{k,t}^C\cap \mathcal{D}_{k,t}\cap \mathcal{E}_{k,t} } \\
        &+ \E\sum_{t=4K}^T\I_{\mathcal{A}_{k,t}\cap \mathcal{D}_{k,t}^C\cap \mathcal{E}_{k,t} }+ \E\sum_{t=4K}^T\I_{\mathcal{A}_{k,t}\cap \mathcal{E}_{k,t}^C } + 4,
    \end{aligned}
    \end{equation}
    and we are going to bound these terms separately in the following analysis. For the first term of \eqref{eqn: regret decomposition unknown var},
    \begin{equation}
    \begin{aligned}
    \label{eqn: unknown var term 1}
        &\E\sum_{t=4K}^T\I_{\mathcal{A}_{k,t}\cap \mathcal{B}_{k,t}\cap \mathcal{D}_{k,t}\cap \mathcal{E}_{k,t} } \overset{(1)}{\le} \E\sum_{t=4K}^T\I_{ \mathcal{A}_{k,t} }\I_{ \mathcal{E}_{k,t} } \I\{\Delta_k (1-\epsilon) \leq \hat B_{k,n_{k,t}}(t)\} \\
        \overset{(2)}{=} & \E\sum_{t=4K}^T\I_{ \mathcal{A}_{k,t} }\I_{ \mathcal{E}_{k,t} } \I\left\{\Delta_k (1-\epsilon) \leq  q_{n_{k,t}-2}\lb\frac{1}{2t\sqrt{\ln t}}\rb \lb\sqrt\frac{\hat\sigma_{\epsilon, k, n_{k,t}}^2Z_{k,n_{k,t}}}{n_{k,t}}+\sqrt\frac{\hat\sigma_{R, k,n_{k,t}}^2}{n_{k,t}+N_k}\rb \right\} \\ 
        \overset{(3)}{\le} & \E\sum_{t=4K}^T\I_{ \mathcal{A}_{k,t} }\I_{ \mathcal{E}_{k,t} } \I\left\{\Delta_k (1-\epsilon) \leq  q_{n_{k,t}-2}\lb\frac{1}{2t\sqrt{\ln t}}\rb \sqrt{1+\epsilon}\lb\sqrt\frac{(1-\rho_k^2)\sigma_{k}^2}{n_{k,t}}+\sqrt\frac{\sigma_{k}^2}{n_{k,t}+N_k}\rb \right\} \\
        \overset{(4)}{\le}& \E\sum_{t=4K}^T\I_{ \mathcal{A}_{k,t} }\I_{ \mathcal{E}_{k,t} } \I\left\{\Delta_k (1-\epsilon) \leq  \sqrt{(1+\epsilon)(t^{\frac{2}{n_{k,t}-3}}-1)}\lb\sqrt{(1-\rho_k^2)\sigma_{k}^2}+\sqrt\frac{\sigma_{k}^2n_{k,t}}{n_{k,t}+N_k}\rb \right\} \\
        \overset{(5)}{\le} & \E\sum_{t=4K}^T\I_{ \mathcal{A}_{k,t} }\I\left\{\Delta_k (1-\epsilon) \leq  \sqrt{(1+\epsilon)(t^{\frac{2}{n_{k,t}-3}}-1)}\lb\sqrt{(1-\rho_k^2)\sigma_{k}^2}+\sqrt\frac{\sigma_{k}^2n_{k,t}}{n_{k,t}+N_k}\rb \right\}.
    \end{aligned}
    \end{equation}
    Here, the inequality (1) is from the definition of $\mathcal{B}_{k,t}$ and the fact that $\I_{ \mathcal{D}_{k,t} }\leq 1$; the equality (2) is from the definition of $\hat B_{k,n}(t)$; the inequality (3) is from the definition of $\mathcal{E}_{k,t}$ and the fact that $(1+\epsilon/3)^2\leq1+\epsilon, \forall\epsilon\in (0,1)$; the inequality (4) is from Proposition \ref{prop: quantile bound}; and the inequality (5) uses the fact that $\I_{ \mathcal{E}_{k,t} }\leq 1$.

    Define the event $$
    \mathcal{G}_{k,t} = \left\{\Delta_k (1-\epsilon) \leq  \sqrt{(1+\epsilon)(t^{\frac{2}{n_{k,t}-3}}-1)}\lb\sqrt{(1-\rho_k^2)\sigma_{k}^2}+\sqrt\frac{\sigma_{k}^2n_{k,t}}{n_{k,t}+N_k}\rb \right\}.
    $$
    Then under the sample size condition 
    $$N_k\geq \frac{1}{\delta_k}\lb\frac{2\ln T}{\ln\lb 1+\frac{\Delta_k^2}{4\sigma_{k}^2} \frac{(1-\epsilon)^2}{1+\epsilon}\rb}+3\rb, $$
    we can control the ratio of online and offline samples on the event $\mathcal{G}_{k,t}$ by
    \begin{equation}
    \begin{aligned}
    \label{eqn: control sample ratio}
        \mathcal{G}_{k,t} \subseteq& \left\{\Delta_k (1-\epsilon) \leq  \sqrt{(1+\epsilon)(t^{\frac{2}{n_{k,t}-3}}-1)}2\sigma_{k} \right\} 
        \subseteq \left\{\ln\lb 1+\frac{\Delta_k^2}{4\sigma_{k}^2} \frac{(1-\epsilon)^2}{1+\epsilon}\rb \leq  \frac{2\ln t}{n_{k,t}-3} \right\}\\
        \subseteq& \left\{n_{k,t}\leq   \frac{2\ln T}{\ln\lb 1+\frac{\Delta_k^2}{4\sigma_{k}^2} \frac{(1-\epsilon)^2}{1+\epsilon}\rb}+3 \right\}\subseteq \left\{\frac{n_{k,t}}{n_{k,t}+N_k}\leq \delta_k\right\}.
    \end{aligned}
    \end{equation}
   Combine it with \eqref{eqn: unknown var term 1}, we have 
    \begin{equation}
        \begin{aligned}
            &\E\sum_{t=4K}^T\I_{\mathcal{A}_{k,t}\cap \mathcal{B}_{k,t}\cap \mathcal{D}_{k,t}\cap \mathcal{E}_{k,t} } \overset{(1)}{\le}  \E\sum_{t=4K}^T\I_{ \mathcal{A}_{k,t} }\I_{ \mathcal{G}_{k,t} } \overset{(2)}{=}  \E\sum_{t=4K}^T\I_{ \mathcal{A}_{k,t} }\I_{ \mathcal{G}_{k,t} }\I\left\{\frac{n_{k,t}}{n_{k,t}+N_k}\leq \delta_k\right\}\\
            \overset{(3)}{\le} &\E\sum_{t=4K}^T\I_{ \mathcal{A}_{k,t} }\I\left\{\frac{n_{k,t}}{n_{k,t}+N_k}\leq \delta_k\right\}\I\left\{\Delta_k (1-\epsilon) \leq  \sqrt{(1+\epsilon)(t^{\frac{2}{n_{k,t}-3}}-1)}\lb\sqrt{(1-\rho_k^2)\sigma_{k}^2}+\sqrt{\sigma_{k}^2\delta_k}\rb \right\}\\
            \overset{(4)}{\le}  &\E\sum_{t=4K}^T\I_{ \mathcal{A}_{k,t} } \I\left\{n_{k,t}\leq   \frac{2\ln T}{\ln\lb 1+\frac{\Delta_k^2}{\sigma_{k}^2} \frac{(1-\epsilon)^2}{1+\epsilon}\frac{1}{(\sqrt{1-\rho_k^2}+\sqrt{\delta_k})^2}\rb}+3 \right\} \\ \overset{(5)}{\le}&\E\sum_{n=4}^\infty \I\left\{n\leq   \frac{2\ln T}{\ln\lb 1+\frac{\Delta_k^2}{\sigma_{k}^2} \frac{(1-\epsilon)^2}{1+\epsilon}\frac{1}{(\sqrt{1-\rho_k^2}+\sqrt{\delta_k})^2}\rb}+3 \right\}  = \frac{2\ln T}{\ln\lb 1+\frac{\Delta_k^2}{\sigma_{k}^2} \frac{(1-\epsilon)^2}{1+\epsilon}\frac{1}{(\sqrt{1-\rho_k^2}+\sqrt{\delta_k})^2}\rb}+3
        \end{aligned}
    \end{equation}
    Here, the inequality (1) is by \eqref{eqn: unknown var term 1} and the definition of $\mathcal{G}_{k,t}$; equality (2) is by \eqref{eqn: control sample ratio}; inequality (3) and (4) are straightforward algebra; inequality (5) uses the fact that $\mathcal{A}_{k,t} = \{n_{k,t+1} = n_{k,t}+1\}$ and $ \I\left\{n_{k,t}\leq   \frac{2\ln T}{\ln\lb 1+\frac{\Delta_k^2}{\sigma_{k}^2} \frac{(1-\epsilon)^2}{1+\epsilon}\frac{1}{(\sqrt{1-\rho_k^2}+\sqrt{\delta_k})^2}\rb}+3 \right\}$ depends on $t$ only through $n_{k,t}$, thus we can transform the sum over $t$ into the sum over $n$.

    For the second term of \eqref{eqn: regret decomposition unknown var}, 
    \begin{equation}
    \begin{aligned}
        &\E\sum_{t=4K}^T\I_{\mathcal{A}_{k,t}\cap \mathcal{B}_{k,t}^C\cap \mathcal{D}_{k,t}\cap \mathcal{E}_{k,t} }\leq \E\sum_{t=4K}^T\I_{\mathcal{A}_{k,t}\cap \mathcal{B}_{k,t}^C\cap \mathcal{D}_{k,t} } \\ 
        \leq& \sum_{t=4K}^T\P\lb \mukntk +\hat B_{k,n_{k,t}}(t)\geq  \hat \mu_{k^\star,n_{k^\star, t}}^\mla +\hat B_{k^\star,n_{k^\star, t}}(t), \quad\mu_{k} +\hat B_{k,n_{k,t}}(t) + \tilde \epsilon_k<\mu^\star, \quad \mukntk \leq \mu_k+\tilde \epsilon_k/2\rb \\
        \leq& \sum_{t=4K}^T\P\left(\mu^\star > \hat \mu_{k^\star,n_{k^\star,t}}^\mla + \hat B_{k^\star,n_{k^\star,t}}(t) + \tilde \epsilon_k/2\right)\leq \sum_{t=4K}^T\sum_{n=1}^T\P\left(\mu^\star > \hat U_{k^\star,n}(t) + \tilde \epsilon_{k}/2\right) \\
        \overset{(1)}{\le}& \sum_{t=4K}^T\frac{1}{2t\sqrt{\ln t}}\sum_{n=1}^{T}\lb\exp\!\left(-\frac{(n+N_{k^\star})\tilde\epsilon_k^2}{32\sigma_{k^\star}^2}\right)
+\exp\!\left(-\frac{n \tilde\epsilon_k^2}{64\sigma_{k^\star}^2}\right)
+\frac{25}{n^2}\rb \\
\leq&C\sum_{t=4K}^T\frac{1}{t\sqrt{\ln t}}\lb
1+\frac{\sigma_{k^\star}^2}{\epsilon^2\Delta_k^2}\rb
\leq C\lb
1+\frac{\sigma_{k^\star}^2}{\epsilon^2\Delta_k^2}\rb\int_{t = 3}^T\frac{dt}{t\sqrt{\ln t}} \leq C \lb
1+\frac{\sigma_{k^\star}^2}{\epsilon^2\Delta_k^2}\rb\sqrt{\ln T}.
    \end{aligned}
    \end{equation}
    Here, $C>0$ is a universal constant, and the inequality (1) is from the definition of confidence band $\hat B_{k,n_{k,t}}(t)$ and Lemma \ref{lem:gaussian-shifted-kn}.

    For the third term of \eqref{eqn: regret decomposition unknown var}, 
    \begin{equation}
    \begin{aligned}
        &\E\sum_{t=4K}^T\I_{\mathcal{A}_{k,t}\cap \mathcal{D}_{k,t}^C\cap \mathcal{E}_{k,t} } \leq \E\sum_{t=4K}^T\I_{ \mathcal{A}_{k,t} }\I_{ \mathcal{E}_{k,n_{k,t}}^3 }\I\{ \mukntk> \mu_k+\tilde \epsilon_k/2 \} \\
        \overset{(1)}{\leq}& \E\sum_{n=4}^\infty\I_{ \mathcal{E}_{k,n}^3 }\I\{ \mukt> \mu_k+\tilde \epsilon_k/2 \} \\
        \overset{(2)}{\leq} & \sum_{n=4}^\infty\E\left[\I_{ \mathcal{E}_{k,n}^3 } \P\lb\mukt> \mu_k+\tilde \epsilon_k/2|\mathcal{S}_{k,n} \rb \right]\\
        \overset{(3)}{\leq} &\sum_{n=4}^\infty\E\left[\I_{ \mathcal{E}_{k,n}^3 } \exp\lb-\frac{\Delta_k^2\epsilon^2}{8(\frac{1}{n+N_k}\rho_k^2\sigma_{k}^2 + \frac{1}{n}Z_{k,n} (1-\rho_k^2)\sigma_{k}^2)}\rb\right] \\
        \overset{(4)}{\leq}& \sum_{n=4}^\infty\E\left[\I_{ \mathcal{E}_{k,n}^3 } \exp\lb-\frac{n\Delta_k^2\epsilon^2}{8(1+\epsilon)\sigma_{k}^2}\rb\right] \leq \sum_{n=4}^\infty \exp\lb-\frac{n\Delta_k^2\epsilon^2}{8(1+\epsilon)\sigma_{k}^2}\rb \\
        \leq & \frac{1}{\exp\lb\frac{\Delta_k^2\epsilon^2}{8(1+\epsilon)\sigma_{k}^2}\rb - 1} \leq \frac{C\sigma_{k}^2}{\epsilon^2\Delta_k^2}< \infty.
    \end{aligned}
    \end{equation}
    Here, $C>0$ is a universal constant, inequality (1) is from the fact that $\mathcal{A}_{k,t} = \{n_{k,t+1} = n_{k,t}+1\}$; inequality (2) is from the fact that $Z_{k,n}\in \mathcal{S}_{k,n}$ in Corollary \ref{cor: Skt}; inequality (3) is from Corollary \ref{cor: Skt} and Hoeffding's inequality of Gaussian random variable; inequality (4) is from the definition of $\mathcal{E}_{k,t}^3$.
    
    For the fourth term of \eqref{eqn: regret decomposition unknown var}, 
    \begin{equation}
    \begin{aligned}
        &\E\sum_{t=4K}^T\I_{ \mathcal{A}_{k,t} }\I_{ \mathcal{E}_{k,t}^C } \overset{(1)}{\leq} \E\sum_{n=4}^\infty\I_{ (\mathcal{E}_{k,n}^{1})^C }+ \I_{ (\mathcal{E}_{k,n}^{2})^C }+ \I_{ (\mathcal{E}_{k,n}^{3})^C } \\
        \overset{(2)}{\leq} & \sum_{n=4}^\infty\lb\P(\hat\sigma_{R, k,n}^2 > (1+\epsilon)\sigma_k^2)+\P(\hat\sigma_{\epsilon, k, n}^2 > (1+\epsilon/3)(1-\rho_k^2)\sigma_k^2)+\P(Z_{k,n} > (1+\epsilon/3))\rb\\
        \overset{(3)}{\leq} &\sum_{n=4}^\infty(e^{-\epsilon/3}(1+\epsilon/3))^{\frac{n-2}{2}}+(e^{-\epsilon}(1+\epsilon))^{\frac{n-1}{2}}+ 2+\sum_{n=6}^\infty\frac{27}{\epsilon^2}\frac{1}{(n-3)(n-5)} \\
        \overset{(4)}{\leq}& \frac{1}{\sqrt{\frac{e^{\epsilon/3}}{1+\epsilon/3}}-1} + \frac{1}{\sqrt{\frac{e^{\epsilon}}{1+\epsilon}}-1} +2+ \frac{27}{\epsilon^2}\frac{\pi^2}{6} \overset{(5)}{\leq} \frac{C}{\epsilon^2}.
    \end{aligned}
    \end{equation}
    Here, $C>0$ is a universal constant, inequality (1) is from the fact that $\mathcal{A}_{k,t} = \{n_{k,t+1} = n_{k,t}+1\}$ and $\mathcal{E}_{k,t} = \mathcal{E}_{k,t'}$ for any round $t, t'$ such that  $n_{k,t} = n_{k,t'}$;  inequality (2) is by the definition of $\mathcal{E}_{k,n}^1, \mathcal{E}_{k,n}^2, \mathcal{E}_{k,n}^3$; inequality (3) is from Lemma \ref{lem: tail sigma} and \ref{lem: tail Z}; inequality (4) is straightforward algebra; inequality (5) uses the fact that $\frac{e^x}{1+x}\geq (1+\frac{x^2}{8})^2$.

    Combining all four terms together, we obtain that 
    \begin{equation}
    \label{eqn: regret bound appendix}
        \E[\reg] \leq \sum_{k\neq k^\star}\lb\frac{2\ln T}{\ln\lb 1+\frac{\Delta_k^2}{\sigma_{k}^2}\frac{(1-\epsilon)^2}{(1+\epsilon)}\frac{1}{(\sqrt{1-\rho_k^2}+\sqrt{\delta_k})^2}\rb}+C\lb
1+\frac{\sigma_{k^\star}^2}{\epsilon^2\Delta_k^2}\rb\sqrt{\ln T}+\frac{C}{\epsilon^2}\lb1+\frac{\sigma_{k}^2}{\Delta_k^2}\rb\rb\Delta_k,
    \end{equation}
    which finishes the proof.
\end{proof}
Furthermore, we can remove the dependency on $\epsilon$ in Theorem \ref{thm: MLA-UCB} using  the following technical lemma: 
\begin{lemma}[Proposition 10 from \cite{cowan2018normal}]
\label{lem: log expansion}
    For any $G>0, \epsilon\in [0,\frac{1}{2}]$, the following holds: 
    \begin{equation}
        \frac{1}{\ln\lb1+G\frac{(1-\epsilon)^2}{1+\epsilon}\rb} \leq \frac{1}{\ln(1+G)} + \frac{10G}{(1+G)(\ln (1+G))^2}\epsilon
    \end{equation}
\end{lemma}
\begin{proof}[Proof of Theorem \ref{thm: regret bound}]
    Using Lemma \ref{lem: log expansion}, denote $G_k = \frac{\Delta_k^2}{\sigma_{k}^2}\frac{1}{(\sqrt{1-\rho_k^2}+\sqrt{\delta_k})^2}$ we can derive from \eqref{eqn: regret bound appendix} that 
    \begin{equation}
    \begin{aligned}
        &\E[\reg] \leq \sum_{k\neq k^\star}\lb\frac{2\ln T}{\ln\lb 1+\frac{\Delta_k^2}{\sigma_{k}^2}\frac{(1-\epsilon)^2}{(1+\epsilon)}\frac{1}{(\sqrt{1-\rho_k^2}+\sqrt{\delta_k})^2}\rb}+C\lb
1+\frac{\sigma_{k^\star}^2}{\epsilon^2\Delta_k^2}\rb\sqrt{\ln T}+\frac{C}{\epsilon^2}\lb1+\frac{\sigma_{k}^2}{\Delta_k^2}\rb\rb\Delta_k \\
        \leq&\sum_{k\neq k^\star}\lb\frac{2\ln T}{\ln\lb 1+G_k\rb}+\frac{10G_k\ln T}{(1+G_k)(\ln(1+G_k))^2}\epsilon+C\lb
1+\frac{\sigma_{k^\star}^2}{\epsilon^2\Delta_k^2}\rb\sqrt{\ln T}+\frac{C}{\epsilon^2}\lb1+\frac{\sigma_{k}^2}{\Delta_k^2}\rb\rb\Delta_k.
    \end{aligned}
    \end{equation}
    Take $\epsilon =\frac{1}{2(\ln T)^{1/6}}$, since $T\geq 4K \geq 8$, we know that $\epsilon =\frac{1}{2(\ln T)^{1/6}} <\frac{1}{2}$, therefore the sample size condition \eqref{eqn: sample size condition} is satisfied, and we have
    \begin{equation}
    \begin{aligned}
        \E[\reg] \leq&\sum_{k\neq k^\star}\bigg(\frac{2\ln T}{\ln\lb 1+G_k\rb}+\frac{10G_k(\ln T)^{5/6}}{(1+G_k)(\ln(1+G_k))^2}\\
        &+C\lb
1+\frac{\sigma_{k^\star}^2}{\Delta_k^2}\rb(\ln T)^{5/6}+{C}\lb1+\frac{\sigma_{k}^2}{\Delta_k^2}\rb(\ln T)^{1/3}\bigg)\Delta_k\\
        \leq& \sum_{k\neq k^\star}\lb\frac{2\ln T}{\ln\lb 1+\frac{\Delta_k^2}{\sigma_{k}^2}\frac{1}{(\sqrt{1-\rho_k^2}+\sqrt{\delta_k})^2}\rb}+O\lb(\ln T)^{5/6}\rb\rb\Delta_k
    \end{aligned}
    \end{equation}
    which concludes the proof. 
\end{proof}
\begin{proof}[Proof of Corollary \ref{cor: regret bound inf N}]
    Take $a_k = 1+\frac{\Delta_k^2}{\sigma_{k}^2}\frac{1}{(\sqrt{1-\rho_k^2}+\sqrt{\delta_k})^2}$, and $b_k = 1+\frac{\Delta_k^2}{\sigma_{k}^2(1-\rho_k^2)}$, then $a_k\leq b_k$. Using the concavity of the function $\ln (x)$, we have $\ln(a_k)\geq \ln(b_k) - \frac{b_k-a_k}{a_k}$, and hence 
    $$
    \frac{1}{\ln(a_k)} \leq \frac{1}{\ln (b_k)} + \frac{b_k-a_k}{a_k\ln(a_k)\ln(b_k)}\leq\frac{1}{\ln (b_k)} + \frac{b_k-a_k}{a_k\ln(a_k)^2} .
    $$
    Notice that as long as $\delta_k\leq 1$ and $\delta_k = O((\ln T)^{-1/3})$, we have 
    $$a_k = 1+\frac{\Delta_k^2}{\sigma_{k}^2}\frac{1}{(\sqrt{1-\rho_k^2}+\sqrt{\delta_k})^2}\geq 1+\frac{\Delta_k^2}{4\sigma_{k}^2}$$
    $$
    b_k-a_k = \frac{\Delta_k^2}{\sigma_{k}^2}\frac{\delta_k+2\sqrt{\delta_k}\sqrt{1-\rho_k^2}}{(\sqrt{1-\rho_k^2}+\sqrt{\delta_k})^2(1-\rho_k^2)}\leq \frac{\Delta_k^2}{\sigma_{k}^2}\frac{\delta_k+2\sqrt{\delta_k}\sqrt{1-\rho_k^2}}{(1-\rho_k^2)^2} = O((\ln T)^{-1/6}).
    $$
    Therefore, combining it with Theorem \ref{thm: regret bound}, we know that under the condition $N_k = \Omega((\ln T)^{4/3})$, we have $\delta_k = O((\ln T)^{-1/3})$ and
    \begin{equation}
    \begin{aligned}
        \E[\reg] \leq& \sum_{k\neq k^\star}\lb\frac{2\ln T}{\ln\lb 1+\frac{\Delta_k^2}{\sigma_{k}^2}\frac{1}{(\sqrt{1-\rho_k^2}+\sqrt{\delta_k})^2}\rb}+O\lb(\ln T)^{5/6}\rb\rb\Delta_k \\ 
        = &\sum_{k\neq k^\star}\lb \frac{2\ln T}{\ln (a_k)} +O\lb(\ln T)^{5/6}\rb\rb\Delta_k \\
        \leq &\sum_{k\neq k^\star}\lb \frac{2\ln T}{\ln (b_k)} + \frac{(b_k-a_k)\ln T}{a_k\ln(a_k)^2} +O\lb(\ln T)^{5/6}\rb\rb\Delta_k\\
        = &\sum_{k\neq k^\star}\lb \frac{2\ln T}{\ln (1+ \frac{\Delta_k^2}{\sigma_{k}^2(1-\rho_k^2)})} +O\lb(\ln T)^{5/6}\rb\rb\Delta_k,
    \end{aligned}
    \end{equation}
    which finishes the proof. 
\end{proof}
\subsection{Proof of regret lower bound}
Before starting the proof, we need the following technical lemma.
\begin{lemma}
\label{lem:gaussian_KL}
Let $m,u\in\mathbb R^d$ be two fixed vectors, and let $\Sigma\succ0$ be a fixed matrix in $\mathbb R^{d\times d}$. Define $\Sigma' := \Sigma + uu^\top$. Then
\begin{equation}
\label{eq:gauss_KL_rankone}
\mathrm{KL}\!\Bigl(\mathcal N(m,\Sigma)\,\Big\|\,\mathcal N(m+u,\Sigma')\Bigr)
=
\frac12 \ln\bigl(1+u^\top \Sigma^{-1} u\bigr).
\end{equation}
\end{lemma}

\begin{proof}
Recall the KL divergence between two Gaussians:
\[
\mathrm{KL}\bigl(\mathcal N(m,\Sigma)\,\big\|\,\mathcal N(m+u,\Sigma')\bigr)
=
\frac12\left(
\ln\frac{\det\Sigma'}{\det\Sigma}
- d
+ \mathrm{tr}(\Sigma'^{-1}\Sigma)
+ u^\top \Sigma'^{-1} u
\right).
\]
Let $\Sigma'=\Sigma+uu^\top$. By the matrix determinant lemma,
\[
\det(\Sigma+uu^\top)=\det(\Sigma)\bigl(1+u^\top\Sigma^{-1}u\bigr),
\]
and by the Sherman--Morrison formula,
\[
(\Sigma+uu^\top)^{-1}
=\Sigma^{-1}-\frac{\Sigma^{-1}uu^\top\Sigma^{-1}}{1+u^\top\Sigma^{-1}u}.
\]
Denote $\alpha:=u^\top\Sigma^{-1}u$. Then
\[
\mathrm{tr}(\Sigma'^{-1}\Sigma)
=\mathrm{tr}\!\left(I-\frac{\Sigma^{-1}uu^\top}{1+\alpha}\right)
=d-\frac{\alpha}{1+\alpha},
\quad
u^\top\Sigma'^{-1}u
=\alpha-\frac{\alpha^2}{1+\alpha}
=\frac{\alpha}{1+\alpha}.
\]
Plugging into the KL formula shows that the trace terms cancel, hence
\[
\mathrm{KL}\bigl(\mathcal N(m,\Sigma)\,\big\|\,\mathcal N(m+u,\Sigma')\bigr)
=\frac12\ln(1+\alpha)
=\frac12\ln\bigl(1+u^\top\Sigma^{-1}u\bigr),
\]
which concludes the proof.
\end{proof}

\begin{remark}
The reason why we choose this specific $\Sigma'$ is that, using the first-order condition, we can easily show that 
\[\Sigma + uu^\top = \arg\min_{\Sigma'} \mathrm{KL}\bigl(\mathcal N(m,\Sigma)\,\big\|\,\mathcal N(m+u,\Sigma')\bigr).\]
\end{remark}

Now we present the proof of Theorem \ref{thm: lower bound}. The main idea of the proof comes from \cite{lai1985asymptotically} and \cite{burnetas1996optimal}. Again, we will use $\E_{\theta}^\pi$ and $\P_{\theta}^\pi$ to denote the expectation and probability under a specific parameter and policy.
\begin{proof}[Proof of Theorem \ref{thm: lower bound}]
    Consider a fixed suboptimal arm $k\neq k^\star$, define $$
    M_k = \frac{1}{2}\ln\lb 1+\dfrac{\Delta_k^2}{\sigma_k^2(1-\rho_k^2)}\rb.
    $$
    Then it suffices to show that for any fixed $\delta>0$
\begin{equation}
\label{eq:nk_lower}
\liminf_{T\to\infty}\frac{\mathbb E_\theta^\pi[n_{k,T}]M_k}{\ln T}
\ge
 1-2\delta.
\end{equation}
Using the Markov inequality, we have 
$$
\P^\pi_\theta\lb \frac{n_{k,T}}{\ln T}\geq \frac{1-2\delta}{M_k}\rb\leq \frac{\E^\pi_\theta [n_{k,T}]M_k}{\ln T(1-2\delta)}.
$$
Then it suffices to show 
$$
\lim_{T\to\infty}\P^\pi_\theta\lb \frac{n_{k,T}}{\ln T}\geq \frac{1-2\delta}{M_k}\rb=1
$$

Consider a fixed alternative $\theta':=(\{\mu_s',\tilde\mu_s',\Sigma_s'\}_{s=1}^K)$ such that $$\mu_k'=\mu_k+\Delta_k+\epsilon, \tilde\mu_k'=\tilde\mu_k, \Sigma_k' = \Sigma_k+uu^T, \text{ where } u=(\Delta_k+\epsilon, 0)^T$$,
and all parameters for the other arm remains the same. Use Lemma \ref{lem:gaussian_KL}, we know that $$\mathrm{KL}\!\Bigl(\mathcal N((\mu_k,\tilde\mu_k),\Sigma_k)\Big\|\mathcal N((\mu_k',\tilde\mu_k'),\Sigma'_k)\Bigr) = \frac{1}{2}\ln \lb1+u^\top\Sigma_k^{-1}u\rb = \frac{1}{2}\ln\lb1+\frac{(\Delta_k+\epsilon)^2}{\sigma_k^2(1-\rho_k^2)}\rb.$$
Therefore, we can choose $\epsilon>0$ such that 
$$
I_{\theta,\theta'}:=\mathrm{KL}\!\Bigl(\mathcal N((\mu_k,\tilde\mu_k),\Sigma_k)\Big\|\mathcal N((\mu_k',\tilde\mu_k'),\Sigma_k')\Bigr) = (1+\delta) M_k
$$

Then arm $k$ becomes the optimal arm under this new instance, thus by Definition \ref{def: uniformly good policy}, we have 
\begin{equation}
\label{eqn: bound on alternative}
    \mathbb E_{\theta'}^\pi[T-n_{k,T}] = o(T^a).
\end{equation}
Define the likelihood ratio as 
$$
L_t = \Pi_{s=1}^t\frac{f_{\theta}(R_{{k,s}},\hat R_{{k,s}})}{f_{\theta'}(R_{{k,s}},\hat R_{{k,s}})}, 
$$
where $f_{\theta}(R_{{k,s}},\hat R_{{k,s}})$ denote the probability density under instance $\theta$. Define two events $A_{k,T}$ and $B_{k,T}$ as
$$
A_{k,T} = \left\{\frac{n_{k,T}}{\ln T}<\frac{1-2\delta}{M_k}\right\}, B_{k,T} = \left\{\ln L_{n_{k,T}}\leq (1-\delta)\ln T\right\}.
$$
Then it suffices to show that $\P^\pi_\theta(A_{k,T}\cap B_{k,T})=o(1)$ and $\P^\pi_\theta(A_{k,T}\cap B_{k,T}^C)=o(1)$.

For the first event, change the measure from $\P_\theta$ to $\P_{\theta'}$, we know that
\begin{equation}
    \begin{aligned}
        &\P^\pi_\theta(A_{k,T}\cap B_{k,T}) = \E^\pi_\theta[1_{A_{k,T}\cap B_{k,T}}] = \E^\pi_{\theta'}[1_{A_{k,T}\cap B_{k,T}}L_{n_{k,T}}]
        \leq T^{1-\delta}\P^\pi_{\theta'}({A_{k,T}})
    \end{aligned}
\end{equation}
where the last inequality uses the fact that $\ln L_{n_{k,T}}\leq (1-\delta)\ln T$ on event $B_{k,T}$. Then, use \eqref{eqn: bound on alternative} we know that 
$$
\P^\pi_{\theta'}({A_{k,T}}) = \P^\pi_{\theta'}\lb T-n_{k,T}>T-\frac{1-2\delta}{M_k}\ln T\rb\leq \frac{\E^\pi_{\theta'}[T-n_{k,T}]}{T-\frac{1-2\delta}{M_k}\ln T} \leq \frac{o(T^a)}{T-\frac{1-2\delta}{M_k}\ln T}=o(T^{a-1}),
$$
thus $\P^\pi_\theta(A_{k,T}\cap B_{k,T})=o(T^{a-\delta})=o(1)$ by taking $a<\delta$.

For the second event, define $b_{k,T} = \frac{(1-2\delta)\ln T}{M_k}$, then $A_{k,T} = \{n_{k,T}\leq b_{k,T}\}$, therefore we have 
\begin{equation}
    \begin{aligned}
        &\P^\pi_\theta(A_{k,T}\cap B_{k,T}^C) \leq \P^\pi_\theta\lb \max_{t\leq \lfloor b_{k,T}\rfloor}\{\ln L_t\}\geq (1-\delta)\ln T\rb \\ 
        =&\P^\pi_\theta\lb \max_{t\leq \lfloor b_{k,T}\rfloor}\{\ln L_t\} / b_{k,T}\geq \frac{1-\delta}{1-2\delta}M_{k}\rb,\\
        =&\P^\pi_\theta\lb \max_{t\leq \lfloor b_{k,T}\rfloor}\{\ln L_t\} / b_{k,T}\geq \frac{1-\delta}{(1-2\delta)(1+\delta)}I_{\theta,\theta'}\rb,\\
        =&\P^\pi_\theta\lb \max_{t\leq \lfloor b_{k,T}\rfloor}\{\ln L_t\} / b_{k,T}\geq \frac{1-\delta}{1-\delta-2\delta^2}I_{\theta,\theta'}\rb,\\
    \end{aligned}
\end{equation}
By the law of large numbers, $\ln L_{t}/t\rightarrow I_{\theta,\theta'}$ almost surely under $\P_{\theta}$, thus $\max_{t\leq n}\ln L_{t}/n\rightarrow I_{\theta,\theta'}$
almost surely under $\P_{\theta}$, which concludes that 
\begin{equation}
\P^\pi_\theta(A_{k,T}\cap B_{k,T}^C) = o(1).    
\end{equation}
Combining together, it implies that $\P_{\theta}^\pi (A_{k,T})=o(1)$, which finishes the proof. 
\end{proof}
\subsection{Proof of non-Gaussian rewards }
\begin{lemma}
\label{lem: student to normalize}
Let $S_n = \sum_{i=1}^n X_i$, $V_n^2 = \sum_{i=1}^n X_i^2$, and let $t_n = S_n / (\sqrt{n}\hat{\sigma}_n)$ denote the Student's $t$-statistic, where $\hat{\sigma}_n^2 = n^{-1} \sum_{i=1}^n (X_i - \bar{X})^2$. For any $\sqrt{n}> x > 0$ and $n > 1$, the following two events are identical:
\begin{equation}
    \left\{\frac{S_n}{V_n} \ge x\right\} = \left\{t_n \ge x \sqrt{\frac{n}{n - x^2}}\right\}.
\end{equation}
\end{lemma}

\begin{proof}
Recall
\[
\hat\sigma_n^2=\frac{1}{n}\Bigl(V_n^2-\frac{S_n^2}{n}\Bigr),
\quad
t_n=\frac{S_n}{\sqrt n\,\hat\sigma_n}
=\frac{S_n\sqrt{n}}{\sqrt{nV_n^2-S_n^2}}.
\]
Define $A:=S_n/V_n$. Then $
t_n^2=\frac{nA^2}{n-A^2}.$
Since $x> 0$, we have
\begin{align*}
&\frac{S_n}{V_n}\ge x
\Longleftrightarrow\ A\geq x\Longleftrightarrow\ A^2(n-x^2)\ge x^2(n-A^2) \text{ and } S_n>0\\
&\Longleftrightarrow\ \frac{nA^2}{n-A^2}\ge \frac{nx^2}{n-x^2} \text{ and } S_n>0\\
&\Longleftrightarrow\ t_n^2\ge x^2\frac{n}{\,n-x^2\,} \text{ and } S_n>0 \\
&\Longleftrightarrow\ t_n\ge x\sqrt{\frac{n}{\,n-x^2\,}}.
\end{align*}
Therefore, the two events are identical, which concludes the desired equation.
\end{proof}

\begin{lemma}
\label{lem: reg residual}
Consider the simple linear regression model $Y_i = \alpha + \beta X_i + \epsilon_i$ for $i=1, \dots, n$. Let $\hat{\alpha}$ and $\hat{\beta}$ denote the Ordinary Least Squares (OLS) estimators obtained by regressing $Y$ on $X$. The sum of squared residuals is given by:
\begin{equation}
    \sum_{i=1}^n (Y_i - \hat{\alpha} - \hat{\beta}X_i)^2 = \left(1 - r_{X,\epsilon}^2\right) \sum_{i=1}^n (\epsilon_i - \bar{\epsilon})^2,
\end{equation}
where $r_{X,\epsilon}$ is the sample correlation coefficient between the regressor $X$ and the error term $\epsilon$, and $\bar{\epsilon} = \frac{1}{n}\sum_{i=1}^n \epsilon_i$ is the sample mean of the errors.
\end{lemma}
\begin{proof}
Since $Y_i=\alpha+\beta X_i+\epsilon_i$, the OLS residuals from regressing $Y$ on $(1,X)$ equal the OLS residuals from regressing $\epsilon$ on $(1,X)$. Hence
\[
\sum_{i=1}^n (Y_i-\hat\alpha-\hat\beta X_i)^2
=\min_{a,b}\sum_{i=1}^n (\epsilon_i-a-bX_i)^2.
\]
Centering removes the intercept, then the right-hand side equals
\begin{equation*}
    \begin{aligned}
        &\min_{b}\sum_{i=1}^n\bigl((\epsilon_i-\bar\epsilon)-b(X_i-\bar X)\bigr)^2\\
        =& \min_{b}\left\{\sum_{i=1}^n(\epsilon_i-\bar\epsilon)^2-2b\sum_{i=1}^n(\epsilon_i-\bar\epsilon)(X_i-\bar X)+b^2\sum_{i=1}^n(X_i-\bar X)^2\right\}\\
        =&\sum_{i=1}^n(\epsilon_i-\bar\epsilon)^2-\frac{\Bigl(\sum_{i=1}^n (X_i-\bar X)(\epsilon_i-\bar\epsilon)\Bigr)^2}{\sum_{i=1}^n (X_i-\bar X)^2}.
    \end{aligned}
\end{equation*}
Now write the sample correlation
\[
r_{X,\epsilon}
=\frac{\sum_{i=1}^n (X_i-\bar X)(\epsilon_i-\bar\epsilon)}
{\sqrt{\sum_{i=1}^n (X_i-\bar X)^2}\sqrt{\sum_{i=1}^n(\epsilon_i-\bar\epsilon)^2}},
\]
which gives
\[
\frac{\Bigl(\sum_{i=1}^n (X_i-\bar X)(\epsilon_i-\bar\epsilon)\Bigr)^2}{\sum_{i=1}^n (X_i-\bar X)^2}
=r_{X,\epsilon}^2\sum_{i=1}^n(\epsilon_i-\bar\epsilon)^2.
\]
Substituting yields the claim.
\end{proof}
\begin{lemma}
\label{lem: variance concentration}
Let $X_1,\dots,X_n$ be i.i.d.\ sub-Gaussian random variables with mean $\mu$, variance $\sigma^2>0$, and sub-Gaussian norm $\|X_i-\mu\|_{\psi_2} \le K$. Define the sample variance
\[
\hat\sigma^2 := \frac{1}{n}\sum_{i=1}^n (X_i-\bar X)^2,
\quad \bar X := \frac{1}{n}\sum_{i=1}^n X_i.
\]
Then there exists an absolute constant $c>0$ such that for any $\epsilon \in (0,1)$ and $n \ge 2$,
\[
\P\!\left(\frac{|\hat\sigma^2-\sigma^2|}{\sigma^2}\ge \epsilon\right)
\le 4 \exp\left(-c \, n \epsilon^2 \frac{\sigma^4}{K^4}\right).
\]
\end{lemma}

\begin{proof}
Without loss of generality, we may assume $\mu=0$. Expanding the square in the sample variance, we can write it exactly as:
\[
\hat\sigma^2 = \frac{1}{n}\sum_{i=1}^n X_i^2 - \bar X^2.
\]
Subtracting $\sigma^2$ and applying the triangle inequality, we bound the deviation by:
\[
|\hat\sigma^2-\sigma^2| \le \left| \frac{1}{n}\sum_{i=1}^n (X_i^2-\sigma^2) \right| + \bar X^2.
\]
For any $\epsilon \in (0,1)$, we can apply a union bound by splitting the error evenly:
\[
\P\!\left(|\hat\sigma^2-\sigma^2|\ge \epsilon\sigma^2\right) \le \P\!\left(\left|\frac{1}{n}\sum_{i=1}^n (X_i^2-\sigma^2)\right| \ge \frac{\epsilon\sigma^2}{2}\right) + \P\!\left(\bar X^2 \ge \frac{\epsilon\sigma^2}{2}\right).
\]

First, since $X_i$ is sub-Gaussian with $\|X_i\|_{\psi_2} \le K$, the centered variables $X_i^2 - \sigma^2$ are sub-exponential with norm bounded by $C_1 K^2$ for some absolute constant $C_1$. Applying Bernstein's inequality for sub-exponential random variables gives:
\[
\P\!\left(\left|\frac{1}{n}\sum_{i=1}^n (X_i^2-\sigma^2)\right| \ge \frac{\epsilon\sigma^2}{2}\right) \le 2\exp\left(-c_1 n \min\left\{\frac{\epsilon^2\sigma^4}{K^4}, \frac{\epsilon\sigma^2}{K^2}\right\}\right).
\]
Since we established that $\sigma \le K$ universally, and we are evaluating the bound for $\epsilon \in (0, 1)$, it follows that $\frac{\epsilon\sigma^2}{K^2} < 1$. Therefore, the quadratic term dominates the minimum, yielding:
\[
\P\!\left(\left|\frac{1}{n}\sum_{i=1}^n (X_i^2-\sigma^2)\right| \ge \frac{\epsilon\sigma^2}{2}\right) \le 2\exp\left(-c_1 n \frac{\epsilon^2\sigma^4}{K^4}\right).
\]

Second, for the empirical mean term, recall that $\bar X$ is a sub-Gaussian random variable with norm $\|\bar X\|_{\psi_2} \le C_2 K / \sqrt{n}$. Its tail probability is therefore bounded by:
\[
\P\!\left(\bar X^2 \ge \frac{\epsilon\sigma^2}{2}\right) = \P\!\left(|\bar X| \ge \sigma\sqrt{\frac{\epsilon}{2}}\right) \le 2\exp\left(-c_2 n \frac{\epsilon\sigma^2}{K^2}\right).
\]
Because $\epsilon < 1$ and $\sigma \le K$, we can lower bound the exponent to match the previous term:
\[
\frac{\epsilon\sigma^2}{K^2} \ge \left(\frac{\epsilon\sigma^2}{K^2}\right)^2 = \frac{\epsilon^2\sigma^4}{K^4}.
\]
Thus, the second term is bounded by $2\exp(-c_2 n \frac{\epsilon^2\sigma^4}{K^4})$.

Letting $c = \min(c_1, c_2)$, we combine the two bounds to conclude:
\[
\P\!\left(\frac{|\hat\sigma^2-\sigma^2|}{\sigma^2}\ge \epsilon\right) \le 4 \exp\left(-c \, n \epsilon^2 \frac{\sigma^4}{K^4}\right),
\]
which holds for all $n \ge 2$.
\end{proof}
\begin{lemma}
\label{lem:corr_concentration}
Let $(X_i,Y_i)_{i=1}^n$ be i.i.d.\ copies of $(X,Y)$ with
\[
\sigma_X^2:=\var(X)\in(0,\infty),\quad \sigma_Y^2:=\var(Y)\in(0,\infty),\quad \cov(X,Y)=0.
\]
Assume $X$ and $Y$ are sub-Gaussian, and let $K := \max\{\|X-\E X\|_{\psi_2}, \|Y-\E Y\|_{\psi_2}\}$. Let $\widehat\rho$ be the sample correlation coefficient
\[
\widehat\rho
:=\frac{\sum_{i=1}^n (X_i-\bar X)(Y_i-\bar Y)}
{\sqrt{\sum_{i=1}^n (X_i-\bar X)^2}\,\sqrt{\sum_{i=1}^n (Y_i-\bar Y)^2}},
\quad
\bar X:=\frac1n\sum_{i=1}^n X_i,\ \bar Y:=\frac1n\sum_{i=1}^n Y_i.
\]
Then there exist absolute constants $C, c > 0$ such that for all $n\ge 2$ and $t\in(0,1]$,
\[
\P\bigl(|\widehat\rho|\ge t\bigr)
\le C \exp\left(-c \, n t^2 \frac{\sigma_X^2 \sigma_Y^2}{K^4}\right).
\]
\end{lemma}

\begin{proof}
Since correlation is invariant to shifts, we may assume without loss of generality that $\E X=\E Y=0$. Define empirical moments and variances
\[
S_{xy}:=\frac1n\sum_{i=1}^n X_iY_i, \quad
\hat\sigma_{XY}:=\frac1n\sum_{i=1}^n (X_i-\bar X)(Y_i-\bar Y),
\]
\[
\hat\sigma_X^2:=\frac1n\sum_{i=1}^n (X_i-\bar X)^2,\quad
\hat\sigma_Y^2:=\frac1n\sum_{i=1}^n (Y_i-\bar Y)^2.
\]
Notice that 
\[
\{|\widehat\rho|\ge t\} \subseteq \{\hat\sigma_X^2\leq \tfrac12\sigma_X^2\}\cup\{\hat\sigma_Y^2\leq \tfrac12\sigma_Y^2\} \cup \{|\hat\sigma_{XY}|\ge (t/2)\sigma_X\sigma_Y\},
\]
therefore
\begin{equation}
\P(|\widehat\rho|\ge t)\le \P(\hat\sigma_X^2\leq \tfrac12\sigma_X^2) + \P(\hat\sigma_Y^2\leq \tfrac12\sigma_Y^2)+\P\bigl(|\hat\sigma_{XY}|\ge (t/2)\sigma_X\sigma_Y\bigr).
\label{eq:rho split event}
\end{equation}

First, applying Lemma \ref{lem: variance concentration} with $\epsilon = 1/2$, the empirical variance bounds decay exponentially:
\begin{equation}
\label{eqn: sigma X tail}
\P(\hat\sigma_X^2\leq \tfrac12\sigma_X^2) \le 4\exp\left(-c_1 n \frac{\sigma_X^4}{K^4}\right),
\end{equation}
and similarly for $\hat\sigma_Y^2$.

Second, for the cross-covariance term in \eqref{eq:rho split event}, use the fact that $\hat \sigma_{XY} = S_{xy}-\bar X\bar Y$ to write:
\[
\P\bigl(|\hat\sigma_{XY}|\ge (t/2)\sigma_X\sigma_Y\bigr)
\le
\P\bigl(|S_{xy}|\ge (t/4)\sigma_X\sigma_Y\bigr)
+\P\bigl(|\bar X\bar Y|\ge (t/4)\sigma_X\sigma_Y\bigr).
\]

For $S_{xy}$, since $X_i$ and $Y_i$ are sub-Gaussian and uncorrelated, their product $X_i Y_i$ is sub-exponential with mean zero. Bernstein's inequality yields:
\[
\P\bigl(|S_{xy}|\ge (t/4)\sigma_X\sigma_Y\bigr) \le 2\exp\left(-c_2 n \min\left\{ \frac{t^2 \sigma_X^2 \sigma_Y^2}{K^4}, \frac{t \sigma_X \sigma_Y}{K^2} \right\} \right) \le 2\exp\left(-c_2 n \frac{t^2 \sigma_X^2 \sigma_Y^2}{K^4}\right),
\]
using the assumptions that $t \le 1$ and $\sigma \le K$.

For the sample mean product $\bar X\bar Y$, apply the AM-GM inequality $2|\bar X\bar Y| \le \bar X^2 + \bar Y^2$:
\[
\P\bigl(|\bar X\bar Y|\ge (t/4)\sigma_X\sigma_Y\bigr) \le \P\bigl(\bar X^2 \ge (t/4)\sigma_X\sigma_Y\bigr) + \P\bigl(\bar Y^2 \ge (t/4)\sigma_X\sigma_Y\bigr).
\]
Because $\bar X$ and $\bar Y$ are sub-Gaussian with norms scaling as $K/\sqrt{n}$, their squares possess exponential tails. This provides:
\begin{equation}
\P\bigl(|\hat\sigma_{XY}|\ge (t/2)\sigma_X\sigma_Y\bigr) \le 6\exp\left(-c_3 n \frac{t^2 \sigma_X^2 \sigma_Y^2}{K^4}\right).
\label{eqn: sigma XY tail}
\end{equation}

Finally, we conclude the proof by combining \eqref{eqn: sigma X tail}, the analogous bound for $Y$, and \eqref{eqn: sigma XY tail} into \eqref{eq:rho split event}.
\end{proof}

Now we turn to the proof of the regret analysis for batched rewards. First, as an analog of Lemma \ref{lem:gaussian-shifted-kn}, we prove the following shifted confidence bound result for $\mukt$ under batched non-Gaussian rewards. We first introduce some notations. 
Let 
\[
\hat\sigma_{R,\hat R,k,n}
:=
\frac1{mn}\sum_{s=1}^{mn}
\bigl(R_{k,s}-\E_{k,n}[R]\bigr)\bigl(\hat R_{k,s}-\E_{k,n}[\hat R]\bigr),
\]
\[
\hat\sigma_{\hat R,k,n}^2
:=
\frac1{mn}\sum_{s=1}^{mn}\bigl(\hat R_{k,s}-\E_{k,n}[\hat R]\bigr)^2,
\quad
\hat\beta_{k,n}
:=
\frac{\hat\sigma_{R,\hat R,k,n}}{\hat\sigma_{\hat R,k,n}^2},
\]
\[
\hat\sigma_{R,k,n}^2
:=
\frac1{mn}\sum_{s=1}^{mn}\bigl(R_{k,s}-\E_{k,n}[R]\bigr)^2,
\]
\[
\hat\sigma_{\hat R,k,n,\all}^2
:=
\frac1{m(n+N_k) }
\left(
\sum_{s=1}^{mn}\bigl(\hat R_{k,s}-\E_{k,n}^{\all}[\hat R]\bigr)^2
+
\sum_{s=1}^{mN_k}\bigl(\hat R^{\off}_{k,s}-\E_{k,n}^{\all}[\hat R]\bigr)^2
\right),
\]
and
\[
\hat\sigma_{\epsilon,k,n}^2
:=
\frac1{mn}
\sum_{s=1}^{mn}
\left(
R_{k,s}
-
\hat\mu^{\mathrm{MLA}}_{k,n}
-
\hat\beta_{k,n}\bigl(\hat R_{k,s}-\E_{k,n}^{\all}[\hat R]\bigr)
\right)^2.
\]
Also define the sample correlation
\[
\hat c_{k,n}
:=
\frac{\sum_{s=1}^{mn}(\epsilon_{k,s}-\E_{k,n}[\epsilon])(\hat R_{k,s}-\E_{k,n}[\hat R])}
{\sqrt{\sum_{s=1}^{mn}(\epsilon_{k,s}-\E_{k,n}[\epsilon])^2}
 \sqrt{\sum_{s=1}^{mn}(\hat R_{k,s}-\E_{k,n}[\hat R])^2}}.
\]

For each $t\in\{K,\dots,T\}$, define
\[
x_t:=\sqrt{2\ln t},
\]
\[
\widetilde B^{(1)}_{k,n}(t)
:=
\frac{\hat\sigma_{\hat R,k,n,\all}}{\hat\sigma_{\hat R,k,n}}
\sqrt{\frac{2\hat\sigma_{R,k,n}^2\ln t}{m(n+N_k)-2\ln t}},
\]
\[
\widetilde B^{(2)}_{k,n}(t)
:=
\frac{1+|\hat c_{k,n}|}{\sqrt{1-\hat c_{k,n}^2}}
\sqrt{\frac{2\hat\sigma_{\epsilon,k,n}^2\ln t}{mn-2\ln t}},
\]
and
\[
\widetilde B_{k,n}(t):=\widetilde B^{(1)}_{k,n}(t)+\widetilde B^{(2)}_{k,n}(t).
\]

Define the good events
\begin{equation}
    \begin{aligned}
    \label{eqn: def Ekn batch}
        &\mathcal E^1_{k,n}:=\{\hat\sigma_{R,k,n}^2\le (1+\varepsilon)\sigma_k^2\},\\
&\mathcal E^2_{k,n}:=\{\hat\sigma_{\epsilon,k,n}^2\le (1+\varepsilon)(1-\rho_k^2)\sigma_k^2\},\\
&\mathcal E^3_{k,n}
:=
\left\{
\hat\sigma_{\hat R,k,n,\all}^2
\le
\left(1+\frac1{3m^{1/3}}\right)\tilde\sigma_k^2
\right\},\\
&\mathcal E^4_{k,n}
:=
\left\{
\hat\sigma_{\hat R,k,n}^2
\ge
\left(1-\frac1{3m^{1/3}}\right)\tilde\sigma_k^2
\right\},\\
&\mathcal E^5_{k,n}
:=
\left\{
|\hat c_{k,n}|
\le
\frac1{3m^{1/3}}
\right\},\\
&\mathcal E_{k,n}
:=
\bigcap_{i=1}^5 \mathcal E^i_{k,n}.
    \end{aligned}
\end{equation}
\begin{lemma}
\label{lem:batched-shifted-simple-kn}
Under Assumptions \ref{asm: batch reward} and \ref{asm: m and T}, there exist universal constants $C_1,C_2>0$ such that for every
arm $k\in[K]$, every deterministic batch count $n\ge 1$, every $t\in\{K,\dots,T\}$,
and every $h>0$,
\begin{align}
&\mathbb P\Bigl(
\hat\mu^{\mla}_{k,n}+\tilde B_{k,n}(t)<\mu_k-h,\ \mathcal E_{k,n}
\Bigr)\notag\\
& \le
\frac{C_1}{t\sqrt{\ln t}}
\left[
\exp\!\left(-\frac{C_2h\sqrt{mn\ln t}}{\sigma_k}\right)
+
\exp(-C_2(n^{1/6}-1)\sqrt{\ln t})
\right].
\label{eq:batched-shifted-simple-main}
\end{align}
\end{lemma}

\begin{proof}
From Proposition \ref{prop: intercept}, we know that 
    \begin{equation}
    \begin{aligned}
    \label{eqn: MLA linear form}
        \mukt = & \E_{k,n}[R] - \hat\beta_{k,n}(\E_{k,n}[\hat R] - \E_{k,n}^\all[\hat R]).
    \end{aligned}
    \end{equation}
    Define $S_{1,k,n}$ and $S_{2,k,n}$ as follows: 
    \begin{equation}
    \label{eqn: S1 S2 non Gaussian}
        S_{1,k,n} = \hat \beta_{k,n}(\E_{k,n}^\all[\hat R] - \tilde\mu_k),\quad 
        S_{2,k,n} = \E_{k,n}[R]-\mu_k - \hat \beta_{k,n}(\E_{k,n}[\hat R] - \tilde\mu_k), 
    \end{equation}
    then we have $\mukt-\mu_k = S_{1,k,n}+ S_{2,k,n}$. Now we will handle these two terms separately.
Hence
\begin{align}
\Bigl\{
\hat\mu^{\mathrm{MLA}}_{k,n}+\widetilde B_{k,n}(t)<\mu_k-h
\Bigr\}
&\subseteq
\Bigl\{
S_{1,k,n}<-\widetilde B^{(1)}_{k,n}(t)-h/2
\Bigr\}\cup
\Bigl\{
S_{2,k,n}<-\widetilde B^{(2)}_{k,n}(t)-h/2
\Bigr\}.
\label{eq:batched-shifted-simple-union}
\end{align}
We bound the two probabilities on the right-hand side separately.

\medskip
\noindent
\textbf{Step 1: the $S_{1,k,n}$ term.}
Since
\[
|\hat\beta_{k,n}|
=
\left|
\frac{\hat\sigma_{R,\hat R,k,n}}{\hat\sigma_{\hat R,k,n}^2}
\right|
\le
\frac{\hat\sigma_{R,k,n}}{\hat\sigma_{\hat R,k,n}},
\]
the event
$
\{
S_{1,k,n}<-\widetilde B^{(1)}_{k,n}(t)-h/2
\}
$
implies
\[
\frac{\hat\sigma_{R,k,n}}{\hat\sigma_{\hat R,k,n}}
\bigl|\E_{k,n}^{\all}[\hat R]-\tilde\mu_k\bigr|
\ge
\frac{\hat\sigma_{\hat R,k,n,\all}}{\hat\sigma_{\hat R,k,n}}
\sqrt{\frac{2\hat\sigma_{R,k,n}^2\ln t}{m(n+N_k)-2\ln t}}
+\frac h2.
\]
Recall $x_t = \sqrt{2\ln t}$. Multiplying both sides by $\frac{\sqrt{m(n+N_k)}}{\hat\sigma_{\hat R,k,n,\all}}\frac{\hat\sigma_{\hat R,k,n}}{\hat\sigma_{ R,k,n}}$ yields
\begin{align}
\frac{\sqrt{m(n+N_k)}\,
\bigl|\E_{k,n}^{\all}[\hat R]-\tilde\mu_k\bigr|}
{\hat\sigma_{\hat R,k,n,\all}}
&\ge
x_t\sqrt{\frac{m(n+N_k)}{m(n+N_k)-x_t^2}}
+
\frac{h\sqrt{m(n+N_k)}}{2\hat\sigma_{R,k,n}}
\frac{\hat\sigma_{\hat R,k,n}}{\hat\sigma_{\hat R,k,n,\all}}.
\label{eq:batched-simple-S1-raw}
\end{align}
On $\mathcal E_{k,n}$ we have
\[
\hat\sigma_{R,k,n}\le \sqrt{1+\epsilon}\,\sigma_k\le \sqrt{\frac32}\,\sigma_k,
\]
and
\[
\frac{\hat\sigma_{\hat R,k,n}}{\hat\sigma_{\hat R,k,n,\all}}
\ge
\sqrt{\frac{1-\frac1{3m^{1/3}}}{1+\frac1{3m^{1/3}}}}
\ge \frac12,
\]
because $\frac1{3m^{1/3}}\le 1/3$.
Therefore \eqref{eq:batched-simple-S1-raw} implies that, on $\mathcal E_{k,n}$,
\begin{align}
\frac{\sqrt{m(n+N_k)}\,
\bigl|\E_{k,n}^{\all}[\hat R]-\tilde\mu_k\bigr|}
{\hat\sigma_{\hat R,k,n,\all}}
&\ge
x_t\sqrt{\frac{m(n+N_k)}{m(n+N_k)-x_t^2}}
+
\frac{h\sqrt{m(n+N_k)}}{2\sqrt{6}\sigma_k}.
\label{eq:batched-simple-S1-good}
\end{align}

Now define
\begin{equation}
\label{eqn: define psi}
    \psi_M(x):=x\sqrt{\frac{M }{M-x^2}},
\quad 0< x<\sqrt M,
\end{equation}
and
\[
U^{(1)}_{k,n}
:=
\frac{
m(n+N_k)\,\bigl|\E_{k,n}^{\all}[\hat R]-\tilde\mu_k\bigr|
}{
\sqrt{
\sum_{s=1}^{mn}(\hat R_{k,s}-\tilde\mu_k)^2
+
\sum_{s=1}^{mN_k}(\hat R_{k,s}^{\off}-\tilde\mu_k)^2
}
}.
\]
By Lemma \ref{lem: student to normalize},
\[
U^{(1)}_{k,n}\ge x \Leftrightarrow
\frac{\sqrt{m(n+N_k)}\,
\bigl|\E_{k,n}^{\all}[\hat R]-\tilde\mu_k\bigr|}
{\hat\sigma_{\hat R,k,n,\all}}
\ge
\psi_{m(n+N_k)}(x).
\]
Since
\[
x_t\sqrt{\frac{m(n+N_k)}{m(n+N_k)-x_t^2}}
=
\psi_{m(n+N_k)}(x_t),
\]
it follows from \eqref{eq:batched-simple-S1-good} that
\begin{align}
\Bigl\{
S_{1,k,n}<-\widetilde B^{(1)}_{k,n}(t)-h/2
\Bigr\}\cap \mathcal E_{k,n}
\subseteq
\left\{
U^{(1)}_{k,n}\ge \psi_{m(n+N_k)}^{-1}
\left(
\psi_{m(n+N_k)}(x_t)+ \frac{h\sqrt{m(n+N_k)}}{2\sqrt{6}\sigma_k}
\right)
\right\}.
\label{eq:batched-simple-S1-reduce}
\end{align}

Next, define the moderate-deviation boundary
\[
d^{(1)}_{k,n}:=\frac{(m(n+N_k))^{1/6}}{\ell_{\hat R_k}^{1/3}},
\]
where $\ell_{\hat R_k}$ is defined in \eqref{eq:ell_R}. 
Under Assumption \ref{asm: m and T}, which requires $m \geq 8 (\ln T)^3\ell_{\hat R_k}$, for sufficiently large $T$,
\[
x_t=\sqrt{2\ln t}\le \sqrt{2\ln T}\le d^{(1)}_{k,n}, \quad  d^{(1)}_{k,n}\le \sqrt{m(n+N_k)/2}.
\]
Moreover, $
\psi_M'(x)=\frac{M\sqrt{M }}{(M-x^2)^{3/2}},
$
hence
$
\sup_{0\le x\le \sqrt{M/2}}\psi_M'(x)\le 3.
$
Therefore, for every
$
0\le u\le d^{(1)}_{k,n}-x_t,
$
the mean-value theorem gives
\[
\psi_{m(n+N_k)}(x_t+u)
\le
\psi_{m(n+N_k)}(x_t)+3u.
\]
Set
\[
u^{(1)}_{k,n,t}(h)
:=
\min\left\{
\frac{h\sqrt{m(n+N_k)}}{6\sqrt{6}\sigma_k},
\ d^{(1)}_{k,n}-x_t
\right\}.
\]
Then
\[
\psi_{m(n+N_k)}(x_t+u^{(1)}_{k,n,t}(h))
\le
\psi_{m(n+N_k)}(x_t)+\frac{h\sqrt{m(n+N_k)}}{2\sqrt{6}\sigma_k}.
\]
Combining this with \eqref{eq:batched-simple-S1-reduce}, we obtain
\[
\Bigl\{
S_{1,k,n}<-\widetilde B^{(1)}_{k,n}(t)-h/2
\Bigr\}\cap \mathcal E_{k,n}
\subseteq
\Bigl\{
U^{(1)}_{k,n}\ge x_t+u^{(1)}_{k,n,t}(h)
\Bigr\}.
\]
Now Lemma \ref{lem: self normalize} implies
\begin{align}
\mathbb P\Bigl(
S_{1,k,n}<-\widetilde B^{(1)}_{k,n}(t)-h/2,\ \mathcal E_{k,n}
\Bigr)
&\le
C\bigl(1-\Phi(x_t+u^{(1)}_{k,n,t}(h))\bigr)
\nonumber\\
&\le
C\Phi(-x_t)e^{-x_tu^{(1)}_{k,n,t}(h)}
\nonumber\\
&\le
\frac{C}{t\sqrt{\ln t}}e^{-x_tu^{(1)}_{k,n,t}(h)},
\label{eq:batched-simple-S1-tail}
\end{align}
where the second inequality comes from the fact that for any $x,u\ge 0$,
\begin{equation}
\Phi(-(x+u))
=\int_{x+u}^{\infty}\frac{e^{-z^2/2}}{\sqrt{2\pi}}\,dz
\le\int_x^{\infty}\frac{e^{-uz}e^{-z^2/2}}{\sqrt{2\pi}}\,dz
\le e^{-ux}\Phi(-x),
\label{eq:gaussian-tail-shift-second}
\end{equation}
and the third inequality uses the Mills' inequality $\P(X>t)\leq \frac{1}{\sqrt{2\pi}}\frac{1}{t}\exp(-\frac{t^2}{2})$.

We now lower-bound $u^{(1)}_{k,n,t}(h)$.
First, we have
$
\frac{h\sqrt{m(n+N_k)}}{6\sqrt{6}\sigma_k}
\ge
\frac{h\sqrt{mn}}{6\sqrt{6}\sigma_k}.
$
Second, from Assumption \ref{asm: m and T},
$
m\ge 8(\ln T)^3\ell_{\hat R_k}^2,
$
so
\[
d^{(1)}_{k,n}
=
\frac{(m(n+N_k))^{1/6}}{\ell_{\hat R_k}^{1/3}}
\ge
\frac{(mn)^{1/6}}{\ell_{\hat R_k}^{1/3}}
\ge
n^{1/6}\sqrt{2\ln T}
\ge
n^{1/6}x_t.
\]
Hence
$
d^{(1)}_{k,n}-x_t
\ge
(n^{1/6}-1)x_t,
$
and 
\[
u^{(1)}_{k,n,t}(h)
\ge
\min\left\{
\frac{h\sqrt{mn}}{6\sqrt{6}\sigma_k},
\ (n^{1/6}-1)x_t.
\right\}
\]

Plugging this lower bound into \eqref{eq:batched-simple-S1-tail} and using
\[
e^{-x_t\min\{a,b\}}\le e^{-x_ta}+e^{-x_tb},
\]
we conclude that
\begin{align}
\mathbb P\Bigl(
S_{1,k,n}<-\widetilde B^{(1)}_{k,n}(t)-h/2,\ \mathcal E_{k,n}
\Bigr)
&\le
\frac{C_1}{t\sqrt{\ln t}}
\left[
\exp\!\left(-\frac{C_2h\sqrt{mn\ln t}}{\sigma_k}\right)
+
\exp\lb-C_2(n^{1/6}-1)\sqrt{\ln t}\rb
\right]
\label{eq:batched-simple-S1-final}
\end{align}
for some universal constant $C_1, C_2>0$.

\textbf{Step 2: the $S_{2,k,n}$ term.}
Define
\[
V_{\epsilon,k,n}^2
:=
\frac1{mn  }\sum_{s=1}^{mn}(\epsilon_{k,s}-\E_{k,n}[\epsilon])^2,
\]

The definition of residuals in Assumption \ref{asm: batch reward} implies that 
    $$
    R_{k,s} = \mu_k+ \beta_k\lb \hat R_{k,s}-\tilde \mu_k\rb + \epsilon_{k,s},
    $$
    plug this into the definition of $S_{2,k,n}$ \eqref{eqn: S1 S2 non Gaussian}, we have 
    \begin{equation}
        \begin{aligned}
            S_{2,k,n} =& \E_{k,n} [\epsilon] + ( \beta_{k} - \hat\beta_{k,n})(\E_{k,n}[\hat R] - \tilde\mu_k) \\ 
            =&\E_{k,n} [\epsilon] - \frac{\sum_{s=1}^{mn} \epsilon_{k,s}(\hat R_{k,s} - \E_{k,n} [\hat R] )}{\sum_{s=1}^{mn} (\hat R_{k,s} - \E_{k,n} [\hat R])^2}(\E_{k,n}[\hat R] - \tilde\mu_k)\\
            = &\E_{k,n} [\epsilon] - \frac{\sum_{s=1}^{mn} (\epsilon_{k,s} - \E_{k,n} [\epsilon])(\hat R_{k,s} - \E_{k,n} [\hat R] )}{\sum_{s=1}^{mn} (\hat R_{k,s} - \E_{k,n} [\hat R])^2}(\E_{k,n}[\hat R] - \tilde\mu_k).\\
        \end{aligned}
    \end{equation}
    
    Then we can write the following decomposition
\[
\sqrt{mn}\frac{S_{2,k,n}}{V_{\epsilon,k,n}}
=
A_{k,n}-\hat c_{k,n}B_{k,n},
\]
where
\[
A_{k,n}:=\frac{\sqrt{mn}\E_{k,n}[\epsilon]}{V_{\epsilon,k,n}},
\quad
B_{k,n}:=\frac{\sqrt{mn}(\E_{k,n}[\hat R]-\tilde\mu_k)}{\hat \sigma_{\hat R,k,n}}.
\]

By Lemma \ref{lem: reg residual} and the definition of $\hat\sigma_{\epsilon,k,n}^2$,
\[
\hat\sigma_{\epsilon,k,n}^2
=
(1-\hat c_{k,n}^2)V_{\epsilon,k,n}^2.
\]
Hence, on $\mathcal E_{k,n}$,
\[
V_{\epsilon,k,n}^2
=
\frac{\hat\sigma_{\epsilon,k,n}^2}{1-\hat c_{k,n}^2}
\le
\frac{(1+\epsilon)(1-\rho_k^2)\sigma_k^2}{1-\left(\frac1{3m^{1/3}}\right)^2}
\le 2\sigma_k^2.
\]

Since
\[
\widetilde B^{(2)}_{k,n}(t)
=
\frac{1+|\hat c_{k,n}|}{\sqrt{1-\hat c_{k,n}^2}}
\sqrt{\frac{2\hat\sigma_{\epsilon,k,n}^2\ln t}{mn-2\ln t}}
=
(1+|\hat c_{k,n}|)V_{\epsilon,k,n}\sqrt{\frac{2\ln t}{mn-2\ln t}},
\]
the event
\[
\Bigl\{
S_{2,k,n}<-\widetilde B^{(2)}_{k,n}(t)-h/2
\Bigr\}\cap \mathcal E_{k,n}
\]
implies
\[
A_{k,n}-\hat c_{k,n}B_{k,n}
\le
-(1+|\hat c_{k,n}|)\sqrt{\frac{2mn\ln t}{mn-2\ln t}}-\frac{\sqrt{mn}h}{2V_{\epsilon,k,n}}.
\]
Now suppose simultaneously that
\[
|A_{k,n}|<\sqrt{\frac{2mn\ln t}{mn-2\ln t}}+\frac{\sqrt{mn}h}{4\sqrt2\,\sigma_k}
\quad\text{and}\quad
|B_{k,n}|<\sqrt{\frac{2mn\ln t}{mn-2\ln t}}+\frac{\sqrt{mn}h}{4\sqrt2\,\sigma_k}.
\]
Then
\[
A_{k,n}-\hat c_{k,n}B_{k,n}
>
-(1+|\hat c_{k,n}|)\bigl(\sqrt{\frac{2mn\ln t}{mn-2\ln t}}+\frac{\sqrt{mn}h}{4\sqrt2\,\sigma_k}\bigr).
\]
On $\mathcal E_{k,n}$ we have $|\hat c_{k,n}|\le \frac1{3m^{1/3}}\le 1/3$, so
\[
(1+|\hat c_{k,n}|)\frac{h}{4\sqrt2\,\sigma_k}
\le
\frac43\cdot \frac{h}{4\sqrt2\,\sigma_k}
=
\frac{h}{3\sqrt2\,\sigma_k}
\le
\frac{h}{2V_{\epsilon,k,n}},
\]
because $V_{\epsilon,k,n}\le \sqrt2\,\sigma_k$ on $\mathcal E_{k,n}$.
Hence
\[
A_{k,n}-\hat c_{k,n}B_{k,n}
>
-(1+|\hat c_{k,n}|)\sqrt{\frac{2mn\ln t}{mn-2\ln t}}-\frac{\sqrt{mn}h}{2V_{\epsilon,k,n}},
\]
which is impossible.
Therefore,
\begin{align}
&\Bigl\{
S_{2,k,n}<-\widetilde B^{(2)}_{k,n}(t)-h/2
\Bigr\}\cap \mathcal E_{k,n}
\subseteq \nonumber\\
&\Bigl\{
|A_{k,n}|\ge \sqrt{\frac{2mn\ln t}{mn-2\ln t}}+\frac{\sqrt{mn}h}{4\sqrt2\,\sigma_k}
\Bigr\}
\cup
\Bigl\{
|B_{k,n}|\ge \sqrt{\frac{2mn\ln t}{mn-2\ln t}}+\frac{\sqrt{mn}h}{4\sqrt2\,\sigma_k}
\Bigr\}.
\label{eq:batched-simple-S2-reduce}
\end{align}

For the $A_{k,n}$ term, define
\[
U^{(2,\epsilon)}_{k,n}
:=
\frac{mn\,|\E_{k,n}[\epsilon]|}{\sqrt{\sum_{s=1}^{mn}\epsilon_{k,s}^2}}, \quad 
U^{(2,\hat R)}_{k,n}
:=
\frac{mn\,|\E_{k,n}[\hat R]-\tilde\mu_k|}
{\sqrt{\sum_{s=1}^{mn}(\hat R_{k,s}-\tilde\mu_k)^2}}.
\]
By Lemma \ref{lem: student to normalize} and the definition of $\psi$ in \eqref{eqn: define psi},
\[
|A_{k,n}|\ge \psi_{mn}(x)
\Leftrightarrow
U^{(2,\epsilon)}_{k,n}\ge x,
\quad 
|B_{k,n}|\ge \psi_{mn}(x)
\Leftrightarrow
U^{(2,\hat R)}_{k,n}\ge x.
\]
Define
\[
d^{(2)}_{k,n}
:=
\min\left\{
\frac{(mn)^{1/6}}{\ell_{\epsilon_k}^{1/3}},
\frac{(mn)^{1/6}}{\ell_{\hat R_k}^{1/3}}
\right\}.
\]
Under Assumption \ref{asm: m and T}, which which requires $m \geq 8 (\ln T)^3\max\{\ell_{\hat R_k},\ell_{\epsilon_k}\}$, for sufficiently large $T$,
\[
x_t = \sqrt{2\ln t} \le \sqrt{2\ln T}\le d^{(2)}_{k,n},
\quad
d^{(2)}_{k,n}\le \sqrt{mn/2}.
\]
Moreover, $
\psi_M'(x)=\frac{M\sqrt{M }}{(M-x^2)^{3/2}},
$
hence
$
\sup_{0\le x\le \sqrt{M/2}}\psi_M'(x)\le 3.
$
Therefore, for every
$
0\le u\le d^{(2)}_{k,n}-x_t,
$
the mean-value theorem gives
\[
\psi_{mn}(x_t+u)
\le
\psi_{mn}(x_t)+3u.
\]
Set
\[
u^{(2)}_{k,n,t}(h)
:=
\min\left\{
\frac{h\sqrt{mn}}{12\sqrt2\,\sigma_k},
\ d^{(2)}_{k,n}-x_t
\right\}.
\]
Then
\[
\psi_{mn}(x_t+u^{(2)}_{k,n,t}(h))
\le
\psi_{mn}(x_t)+\frac{\sqrt{mn}h}{4\sqrt2\,\sigma_k}
=
\sqrt{\frac{2mn\ln t}{mn-2\ln t}}+\frac{\sqrt{mn}h}{4\sqrt2\,\sigma_k}.
\]
Therefore, \eqref{eq:batched-simple-S2-reduce} implies
\[
\Bigl\{
S_{2,k,n}<-\widetilde B^{(2)}_{k,n}(t)-h/2
\Bigr\}\cap \mathcal E_{k,n}
\subseteq
\Bigl\{
U^{(2,\epsilon)}_{k,n}\ge x_t+u^{(2)}_{k,n,t}(h)
\Bigr\}
\cup
\Bigl\{
U^{(2,\hat R)}_{k,n}\ge x_t+u^{(2)}_{k,n,t}(h)
\Bigr\}.
\]
Applying Lemma \ref{lem: self normalize} to both self-normalized means gives
\begin{align}
\mathbb P\Bigl(
S_{2,k,n}<-\widetilde B^{(2)}_{k,n}(t)-h/2,\ \mathcal E_{k,n}
\Bigr)
&\le
C\bigl(1-\Phi(x_t+u^{(2)}_{k,n,t}(h))\bigr)
\nonumber\\
&\le
\frac{C}{t\sqrt{\ln t}}e^{-x_tu^{(2)}_{k,n,t}(h)}.
\label{eq:batched-simple-S2-tail}
\end{align}

By Assumption \ref{asm: m and T},
\[
m\ge 8(\ln T)^3\max\{\ell_{\epsilon_k}^2,\ell_{\hat R_k}^2\},
\]
so
\[
d^{(2)}_{k,n}
\ge
n^{1/6}\sqrt{2\ln T}
\ge
n^{1/6}x_t,
\]
and thus
\[
d^{(2)}_{k,n}-x_t\ge (n^{1/6}-1)x_t.
\]
Therefore,
\[
u^{(2)}_{k,n,t}(h)
\ge
\min\left\{
\frac{h\sqrt{mn}}{12\sqrt{2}\sigma_k},
\ (n^{1/6}-1)x_t
\right\}.
\]
Plugging this into \eqref{eq:batched-simple-S2-tail}, we obtain
\begin{align}
\mathbb P\Bigl(
S_{2,k,n}<-\widetilde B^{(2)}_{k,n}(t)-h/2,\ \mathcal E_{k,n}
\Bigr)
&\le
\frac{C_1}{t\sqrt{\ln t}}
\left[
\exp\!\left(-\frac{C_2h\sqrt{mn\ln t}}{\sigma_k}\right)
+
\exp\lb-C_2(n^{1/6}-1)\sqrt{\ln t}\rb
\right]
\label{eq:batched-simple-S2-final}
\end{align}
for some universal constant $C_1, C_2>0$.

Finally, combine \eqref{eq:batched-shifted-simple-union},
\eqref{eq:batched-simple-S1-final}, and \eqref{eq:batched-simple-S2-final}.
This proves \eqref{eq:batched-shifted-simple-main}.
\end{proof}

    Now we are ready to prove Theorem \ref{thm: regret bound non Gaussian}. The proof strategy is similar to Theorem \ref{thm: MLA-UCB} with some adaptation to the new settings. Again, we start with a more general theorem below, and we will then derive Theorem \ref{thm: regret bound non Gaussian} as a corollary.
    \begin{theorem}
    \label{thm: MLA-UCB non Gaussian}
        Under Assumption \ref{asm: batch reward} and \ref{asm: m and T}, for any $\epsilon\in(0,1/2)$ and $T\geq \max\{K+1, 5\}$, if the sample size of offline predictions satisfies 
\begin{equation}
    \label{eqn: sample size condition batch}
    mN_k\geq \frac{1}{\delta_k}\lb\frac{8\sigma_k^2\ln T}{\Delta_k^2}\frac{(1+\frac{1}{m^{1/3}})^2(1+\epsilon)}{(1-\epsilon)^2} + \frac{2\ln T}{m}\rb
\end{equation}
for some $\delta_k\leq1$, then the expected regret of Algorithm \ref{alg:PPUCB batch} can be bounded by:
\begin{equation}
\label{eqn: regret batch eps}
        \begin{aligned}
            &\frac{1}{m}\E[\reg]\\
            \leq& \sum_{k\neq k^\star}\lb\lb\sqrt{\delta_k}+\sqrt{1-\rho_k^2}\rb^2 \frac{2\sigma_k^2\ln T}{m\Delta_k^2} + \frac{C\sigma_k^2}{m\Delta_k^2\epsilon^2} + \frac{C\epsilon\sigma_k^2\ln T}{m\Delta_k^2} 
        +\frac{C}{(\ln T)^4\epsilon^2}+ C\sqrt{\ln T} \rb\Delta_k\\
        &+C\Delta_{\max}T\lb\frac{\exp(-Cm^{1/3})}{1-\exp(-Cm^{1/3})} + \frac{\exp(-Cm\epsilon^2)}{1-\exp(-Cm\epsilon^2)}\rb,
        \end{aligned}
    \end{equation}
    where $\Delta_{\max} = \max_{k\neq k^\star}\Delta_k$
    \end{theorem}
    \begin{proof}[Proof of Theorem \ref{thm: MLA-UCB non Gaussian}]
    For any $\epsilon\in(0,1/2)$, define $\tilde \epsilon_k = {\Delta_k\epsilon}$. Denote the confidence bound in Algorithm \ref{alg:PPUCB batch} as 
    \begin{equation}
        \label{eqn: def Bkt non Gaussian}
        \hat B_{k,n} (t) =  \lb1+\frac{1}{m^{1/3}}\rb\lb\sqrt{\frac{ 2\hat\sigma_{R,k,n}^2\ln t}{m(n+N_k)}}+\sqrt{\frac{2\hat\sigma_{\epsilon,k,n}^2\ln t}{mn}}\rb,
    \end{equation}
    and define the following events (recall the definition of $\mathcal{E}_{k,n}^i, i=1,2,3,4,5$ from \eqref{eqn: def Ekn batch})
    \begin{equation}
    \begin{aligned}
        \mathcal{A}_{k,t} =& \{A_{t} = k\} \subset \{\mukntk +\hat B_{k,n_{k,t}}(t)\geq  \hat \mu_{k^\star,n_{k^\star,t}}^\mla +\hat B_{k^\star,n_{k^\star,t}}(t) \}, \\
        \mathcal{B}_{k,t} =& \{\mu_{k} +\hat B_{k,n_{k,t}}(t) + \tilde \epsilon_k\geq\mu^\star\}, \\
        \mathcal{D}_{k,t} =& \{\mukntk \leq \mu_k+\tilde \epsilon_k/2 \},\\ 
        \mathcal{H}_{k,t} = & \mathcal{E}_{k,n_{k,t}}\cap \mathcal{E}_{k^\star,n_{k^\star,t}}. 
    \end{aligned}
    \end{equation}
    Then we can express the expected regret as 
    \begin{equation}
        \E[\operatorname{Reg_T}] = \E\sum_{k\neq k^\star}mn_{k,T}\Delta_k = m\sum_{k\neq k^\star}\Delta_k \E\sum_{t=1}^T \I_{\mathcal{A}_{k,t}}. 
    \end{equation}
    Furthermore, we can decompose the expectation as follows
    \begin{equation}
    \label{eqn: regret decomposition batch}
    \begin{aligned}
        \E\sum_{t=1}^T\I_{\mathcal{A}_{k,t}} =& \E\sum_{t=K}^T\I_{\mathcal{A}_{k,t}\cap \mathcal{B}_{k,t}\cap \mathcal{D}_{k,t}\cap \mathcal{H}_{k,t} } +\E\sum_{t=K}^T\I_{\mathcal{A}_{k,t}\cap \mathcal{B}_{k,t}^C\cap \mathcal{D}_{k,t}\cap \mathcal{H}_{k,t} } \\
        &+ \E\sum_{t=K}^T\I_{\mathcal{A}_{k,t}\cap \mathcal{D}_{k,t}^C\cap \mathcal{H}_{k,t} }+ \E\sum_{t=K}^T\I_{\mathcal{A}_{k,t}\cap \mathcal{H}_{k,t}^C } + 1,
    \end{aligned}
    \end{equation}
    and we are going to bound these terms separately in the following analysis. 
    
    \textbf{Bound the first term of \eqref{eqn: regret decomposition batch}} For the first term of \eqref{eqn: regret decomposition batch},
    \begin{equation}
    \begin{aligned}
    \label{eqn: unknown var term 1 batch}
        &\E\sum_{t=K}^T\I_{\mathcal{A}_{k,t}\cap \mathcal{B}_{k,t}\cap \mathcal{D}_{k,t}\cap \mathcal{H}_{k,t} } \overset{(1)}{\le} \E\sum_{t=K}^T\I_{ \mathcal{A}_{k,t} }\I_{ \mathcal{H}_{k,t} } \I\{\Delta_k (1-\epsilon) \leq \hat B_{k,n_{k,t}}\} \\
        \overset{(2)}{=} & \E\sum_{t=K}^T\I_{ \mathcal{A}_{k,t} }\I_{ \mathcal{H}_{k,t} } \I\left\{\Delta_k (1-\epsilon) \leq  \lb1+\frac{1}{m^{1/3}}\rb\lb\sqrt{\frac{ 2\hat\sigma_{R, k,n_{k,t}}^2\ln t}{m(n_{k,t}+N_k)}}+\sqrt{\frac{2\hat\sigma_{\epsilon,k,n_{k,t}}^2\ln t}{mn_{k,t}}}\rb\right\} \\ 
        \overset{(3)}{\le} & \E\sum_{t=K}^T\I_{ \mathcal{A}_{k,t} }\I_{ \mathcal{H}_{k,t} } \I\left\{\Delta_k \leq  \frac{1+\frac{1}{m^{1/3}}}{1-\epsilon}\lb\sqrt{\frac{ 2(1+\epsilon)\sigma_{k}^2\ln t}{m(n_{k,t}+N_k)}}+\sqrt{\frac{2(1+\epsilon)(1-\rho_k^2)\sigma_{k}^2\ln t}{mn_{k,t}}}\rb \right\} \\
        \overset{(4)}{\le} & \E\sum_{t=K}^T\I_{ \mathcal{A}_{k,t} }\I\left\{\Delta_k \leq  \frac{1+\frac{1}{m^{1/3}}}{1-\epsilon}\lb\sqrt{\frac{ 2(1+\epsilon)\sigma_{k}^2\ln t}{m(n_{k,t}+N_k)}}+\sqrt{\frac{2(1+\epsilon)(1-\rho_k^2)\sigma_{k}^2\ln t}{mn_{k,t}}}\rb \right\}.
    \end{aligned}
    \end{equation}
    Here, the inequality (1) is from the definition of $\mathcal{B}_{k,t}$ and the fact that $\I_{ \mathcal{D}_{k,t} }\leq 1$; the equality (2) is from the definition of $\hat B_{k,n}$ in \eqref{eqn: def Bkt non Gaussian}; the inequality (3) is from the definition of $\mathcal{H}_{k,t}$; and the inequality (4) uses the fact that $\I_{ \mathcal{H}_{k,t} }\leq 1$.

    Define the event $$
    \mathcal{G}_{k,t} = \left\{\Delta_k \leq  \frac{1+\frac{1}{m^{1/3}}}{1-\epsilon}\lb\sqrt{\frac{ 2(1+\epsilon)\sigma_{k}^2\ln t}{m(n_{k,t}+N_k)}}+\sqrt{\frac{2(1+\epsilon)(1-\rho_k^2)\sigma_{k}^2\ln t}{mn_{k,t}}}\rb \right\}.
    $$
    Then under the sample size condition \eqref{eqn: sample size condition batch}
    $$N_k\geq \frac{1}{\delta_k}\lb\frac{8\sigma_k^2\ln T}{m\Delta_k^2}\frac{(1+\frac{1}{m^{1/3}})^2(1+\epsilon)}{(1-\epsilon)^2} + \frac{2\ln T}{m}\rb, $$
    we can control the ratio of online and offline samples on the event $\mathcal{G}_{k,t}$ by
    \begin{equation}
    \begin{aligned}
    \label{eqn: control sample ratio batch}
        &\mathcal{G}_{k,t} \subseteq  \left\{\Delta_k \leq  2\frac{1+\frac{1}{m^{1/3}}}{1-\epsilon}\sqrt{\frac{ 2(1+\epsilon)\sigma_{k}^2\ln t}{mn_{k,t}}} \right\}\\ 
        \subseteq& \left\{n_{k,t}\leq \frac{8\sigma_k^2\ln T}{m\Delta_k^2}\frac{(1+\frac{1}{m^{1/3}})^2(1+\epsilon)}{(1-\epsilon)^2} +\frac{2\ln T}{m}\right\}\subseteq \left\{\frac{n_{k,t}}{n_{k,t}+N_k}\leq \delta_k\right\}.
    \end{aligned}
    \end{equation}
    Combine it with \eqref{eqn: unknown var term 1 batch}, we have 
    \begin{equation}
        \begin{aligned}
            &\E\sum_{t=K}^T\I_{\mathcal{A}_{k,t}\cap \mathcal{B}_{k,t}\cap \mathcal{D}_{k,t}\cap \mathcal{H}_{k,t} } \overset{(1)}{\le}  \E\sum_{t=K}^T\I_{ \mathcal{A}_{k,t} }\I_{ \mathcal{G}_{k,t} } \overset{(2)}{=}  \E\sum_{t=K}^T\I_{ \mathcal{A}_{k,t} }\I_{ \mathcal{G}_{k,t} }\I\left\{\frac{n_{k,t}}{n_{k,t}+N_k}\leq \delta_k\right\}\\
            \overset{(3)}{\le} &\E\sum_{t=K}^T\I_{ \mathcal{A}_{k,t} }\I\left\{\Delta_k \leq  \frac{1+\frac{1}{m^{1/3}}}{1-\epsilon}\lb\sqrt{\frac{ 2(1+\epsilon)\sigma_{k}^2\ln t}{mn_{k,t}/\delta_k}}+\sqrt{\frac{2(1+\epsilon)(1-\rho_k^2)\sigma_{k}^2\ln t}{mn_{k,t}}}\rb \right\}.\\
            \overset{(4)}{\le}  &\E\sum_{t=K}^T\I_{ \mathcal{A}_{k,t} } \I\left\{mn_{k,t}\leq \frac{(1+\frac{1}{m^{1/3}})^2(1+\epsilon)}{(1-\epsilon)^2}\lb\sqrt{\delta_k}+\sqrt{1-\rho_k^2}\rb^2 \frac{2\sigma_k^2\ln T}{\Delta_k^2}  \right\} \\
            \overset{(5)}{\le}&\E\sum_{n=1}^\infty \I\left\{mn\leq \frac{(1+\frac{1}{m^{1/3}})^2(1+\epsilon)}{(1-\epsilon)^2}\lb\sqrt{\delta_k}+\sqrt{1-\rho_k^2}\rb^2 \frac{2\sigma_k^2\ln T}{\Delta_k^2}  \right\}\\
            =& \frac{1}{m}\lb\frac{(1+\frac{1}{m^{1/3}})^2(1+\epsilon)}{(1-\epsilon)^2}\lb\sqrt{\delta_k}+\sqrt{1-\rho_k^2}\rb^2 \frac{2\sigma_k^2\ln T}{\Delta_k^2} \rb
        \end{aligned}
    \end{equation}
    Here, the inequality (1) is by \eqref{eqn: unknown var term 1 batch} and the definition of $\mathcal{G}_{k,t}$; equality (2) is by \eqref{eqn: control sample ratio batch}; inequality (3) and (4) are straightforward algebra; inequality (5) uses the fact that $\mathcal{A}_{k,t} = \{n_{k,t+1} = n_{k,t}+1\}$ and $ \I\left\{mn_{k,t}\leq \frac{(1+\frac{1}{m^{1/3}})^2(1+\epsilon)}{(1-\epsilon)^2}\lb\sqrt{\delta_k}+\sqrt{1-\rho_k^2}\rb^2 \frac{2\sigma_k^2\ln T}{\Delta_k^2} \right\}$ depends on $t$ only through $n_{k,t}$, thus we can transform the sum over $t$ into the sum over $n$.

    \textbf{Bound the second term of \eqref{eqn: regret decomposition batch}} For the second term of \eqref{eqn: regret decomposition batch}, notice that 
    \begin{equation}
        \begin{aligned}
        \label{eqn: tilde hat Bkt}
            &\mathcal{H}_{k,t}\subset\mathcal{E}_{k^\star,n_{k^\star,t}}^3\cap \mathcal{E}_{k^\star,n_{k^\star,t}}^4\cap \mathcal{E}_{k^\star,n_{k^\star,t}}^5\\
            =&\left\{|\hat c_{k^\star,n_{k^\star,t}}|\leq \frac{1}{3m^{1/3}}, \sigma_{\hat R, k^\star,n_{k^\star,t}, \all}^2\leq \lb1+\frac{1}{3m^{1/3}}\rb\tilde\sigma_{k^\star}^2, \hat \sigma_{\hat R, k^\star,n_{k^\star,t}}^2\geq \lb1-\frac{1}{3m^{1/3}}\rb\tilde\sigma_{k^\star}^2\right\}\\
            \overset{(1)}{\subset}&\left\{\frac{1+|\hat c_{k^\star,n_{k^\star,t}}|}{\sqrt{1-\hat c_{k^\star,n_{k^\star,t}}^2}}\leq 1+\frac{1}{2m^{1/3}}, \frac{\hat \sigma_{\hat R, k^\star,n_{k^\star,t} \all}}{\hat \sigma_{\hat R, k^\star,n_{k^\star,t}}}\leq 1+\frac{1}{2m^{1/3}}\right\}\\
            \overset{(2)}{\subset}&\left\{\tilde B_{k^\star,n_{k^\star,t}}(t)\leq \hat B_{k^\star,n_{k^\star,t}}(t)\right\},
        \end{aligned}
    \end{equation}
    where (1) uses the algebra fact that $\frac{1+x}{\sqrt{1-x^2}}=\sqrt\frac{1+x}{1-x}\leq 1+1.5x, \forall x\in[0,1/3]$, and (2) uses the fact that for $m\geq 8(\ln T)^4\ge 8(\ln T)^3$ when $T\ge 3$,  $\sqrt\frac{mn_{k,t}}{mn_{k,t}-2\ln t}\leq \sqrt\frac{m}{m-2\ln T}\leq \sqrt\frac{1}{1-m^{-2/3}}\leq 1+\frac{1}{3}m^{-1/3}$, and $(1+\frac{1}{2m^{1/3}})(1+\frac{1}{3m^{1/3}})\leq 1+\frac{1}{m^{1/3}}, \forall m\ge 1$, which implies
    \begin{equation*}
        \begin{aligned}
\widetilde B_{k,n}(t)&:= \frac{\hat\sigma_{\hat R,k,n,\all}}{\hat\sigma_{\hat R,k,n}}
\sqrt{\frac{2\hat\sigma_{R,k,n}^2\ln t}{m(n+N_k)-2\ln t}}+ \frac{1+|\hat c_{k,n}|}{\sqrt{1-\hat c_{k,n}^2}}
\sqrt{\frac{2\hat\sigma_{\epsilon,k,n}^2\ln t}{mn-2\ln t}} \\
\leq &\lb1+\frac{1}{2m^{1/3}}\rb 
\lb \sqrt{\frac{2\hat\sigma_{R,k,n}^2\ln t}{m(n+N_k)-2\ln t}}+ 
\sqrt{\frac{2\hat\sigma_{\epsilon,k,n}^2\ln t}{mn-2\ln t}}\rb\\
\leq &\lb1+\frac{1}{2m^{1/3}}\rb \lb1+\frac{1}{3m^{1/3}}\rb
\lb \sqrt{\frac{2\hat\sigma_{R,k,n}^2\ln t}{m(n+N_k)}}+ 
\sqrt{\frac{2\hat\sigma_{\epsilon,k,n}^2\ln t}{mn}}\rb\\
\leq &\lb1+\frac{1}{m^{1/3}}\rb
\lb \sqrt{\frac{2\hat\sigma_{R,k,n}^2\ln t}{m(n+N_k)}}+ 
\sqrt{\frac{2\hat\sigma_{\epsilon,k,n}^2\ln t}{mn}}\rb\\
=&\hat B_{k,n}(t).
        \end{aligned}
    \end{equation*}
    Therefore, the second term of \eqref{eqn: regret decomposition batch} can be bounded by
    \begin{equation}
    \begin{aligned}
        &\E\sum_{t=K}^T\I_{\mathcal{A}_{k,t}\cap \mathcal{B}_{k,t}^C\cap \mathcal{D}_{k,t}\cap \mathcal{H}_{k,t} } \\ 
        =& \E\sum_{t=K}^T\I\left\{ \mukntk +\hat B_{k,n_{k,t}}(t)\geq  \hat \mu_{k^\star,n_{k^\star,t}}^\mla +\hat B_{k^\star,n_{k^\star,t}}(t), \mu_{k} +\hat B_{k,n_{k,t}}(t) + \tilde \epsilon_k<\mu^\star,  \mukt \leq \mu_k+\tilde \epsilon_k/2\right\}\I_{\mathcal{H}_{k,t}} \\
        \leq& \E\sum_{t=K}^T\I\left\{\mu^\star > \hat \mu_{k^\star,n_{k^\star,t}}^\mla + \hat B_{k^\star,n_{k^\star,t}}(t)+\tilde \epsilon_k/2\right\}\I_{\mathcal{H}_{k,t}} \\
        \overset{(1)}{\le} &\E\sum_{t=K}^T\I\left\{\mu^\star > \hat \mu_{k^\star,n_{k^\star,t}}^\mla + \tilde B_{k^\star,n_{k^\star,t}}(t)+\tilde \epsilon_k/2\right\}\I_{\mathcal{E}_{k^\star,n_{k^\star,t}}}\\
        = &\sum_{t=K}^T\P(\mu^\star > \hat \mu_{k^\star,n_{k^\star,t}}^\mla +\tilde B_{k^\star,n_{k^\star,t}}(t)+\tilde \epsilon_k/2, \mathcal{E}_{k^\star,n_{k^\star,t}})
        \\
        \leq &\sum_{t=K}^T\sum_{n=1}^T\P(\mu^\star > \hat \mu_{k^\star,n}^\mla +\tilde B_{k^\star,n}(t)+\tilde \epsilon_k/2, \mathcal{E}_{k^\star,n})\\
        \overset{(2)}{\le}&
        \sum_{t=K}^T\sum_{n=1}^T\frac{C_1}{t\sqrt{\ln t}}
\left[
\exp\!\left(-\frac{C_2\Delta_k\epsilon\sqrt{mn\ln t}}{\sigma_{k^\star}}\right)
+
\exp(-C_2(n^{1/6}-1)\sqrt{\ln t})\right]\\
    \end{aligned}
    \end{equation}
    Here, the inequality (1) is from \eqref{eqn: tilde hat Bkt}, and the inequality (2) is from Lemma \ref{lem:batched-shifted-simple-kn}.
    For the first series, let
\[
a_t:=C_2\frac{\Delta_k\epsilon\sqrt{m\ln t}}{\sigma_{k^\star}}.
\]
Then
\[
\sum_{n=1}^T e^{-a_t\sqrt n}
\le
1+\int_0^\infty e^{-a_t\sqrt x}\,dx
=
1+\frac{2}{a_t^2}
\le
C\left(
1+\frac{\sigma_{k^\star}^2}{m \Delta_k^2\epsilon^2\ln t}
\right).
\] 
For the second series, since $t\geq K\geq 2$,
\[
\sum_{n=1}^T \exp(-C(n^{1/6}-1)\sqrt{\ln t})
\le
\sum_{n=1}^\infty \exp(-C(n^{1/6}-1))
<\infty.
\]
Therefore we conclude that 
\begin{equation}
    \begin{aligned}
        &\E\sum_{t=K}^T\I_{\mathcal{A}_{k,t}\cap \mathcal{B}_{k,t}^C\cap \mathcal{D}_{k,t}\cap \mathcal{H}_{k,t} } \leq \sum_{t=K}^T\sum_{n=1}^T\frac{C_1}{t\sqrt{\ln t}}
\left[
\exp\!\left(-\frac{C_2\Delta_k\epsilon\sqrt{mn\ln t}}{\sigma_{k^\star}}\right)
+
\exp(-C_2(n^{1/6}-1)\sqrt{\ln t})\right]\\
\leq &\sum_{t=K}^T \frac{C}{t\sqrt{\ln t}}\lb 1+ \frac{\sigma_{k^\star}^2}{m \Delta_k^2\epsilon^2\ln t}\rb \leq C\sqrt{\ln T} + \frac{C\sigma_{k^\star}^2}{m \Delta_k^2\epsilon^2}\leq C\sqrt{\ln T} + \frac{C\sigma_{k}^2}{m \Delta_k^2\epsilon^2}
    \end{aligned}
\end{equation}

    \textbf{Bound the third term of \eqref{eqn: regret decomposition batch}} For the third term of \eqref{eqn: regret decomposition batch}, 
    \begin{equation}
    \begin{aligned}
    \label{eqn: batch term 3 first step}
        &\E\sum_{t=K}^T\I_{\mathcal{A}_{k,t}\cap \mathcal{D}_{k,t}^C\cap \mathcal{H}_{k,t} } \leq \E\sum_{t=K}^T\I_{ \mathcal{A}_{k,t} }\I_{\cap_{i=1}^5\mathcal{E}_{k,n_{k,t}}^i}\I_{ \mathcal{D}_{k,t}^C } \overset{(1)}{\leq}  \sum_{n=1}^\infty\E\left[\I_{\cap_{i=1}^5\mathcal{E}_{k,n}^i} \I\left\{\mukt> \mu_k+\tilde \epsilon_k/2\right\} \right]\\
    \end{aligned}
    \end{equation}
    Here, inequality (1) is from the fact that $\mathcal{A}_{k,t} = \{n_{k,t+1} = n_{k,t}+1\}$ and $\mathcal{D}_{k,t} = \mathcal{D}_{k,t'}$ for any round $t, t'$ such that  $n_{k,t} = n_{k,t'}$.

    Use Proposition \ref{prop: intercept}, we know that 
    \begin{equation}
    \begin{aligned}
        \mukt-\mu_k = & \E_{k,n}[R]-\mu_k - \hat \beta_{k,n}(\E_{k,n}[\hat R] - \E_{k,n}^\all[\hat R]),
    \end{aligned}
    \end{equation}
    which implies 
    \begin{equation}
    \begin{aligned}
        |\mukt-\mu_k| \leq & |\E_{k,n}[R]-\mu_k| + |\hat \beta_{k,n}||\E_{k,n}[\hat R] -\tilde\mu_k|+ |\hat \beta_{k,n}||\E_{k,n}^\all[\hat R]-\tilde\mu_k|.
    \end{aligned}
    \end{equation}
    Since $|\hat \beta_{k,n}| = |\frac{\hat \sigma_{R,\hat R, k,n}}{\hat \sigma_{R, k,n}\hat \sigma_{\hat R, k,n}}|\frac{\hat \sigma_{R, k,n}}{\hat \sigma_{\hat R, k,n}} \leq \frac{\hat \sigma_{R, k,n}}{\hat \sigma_{\hat R, k,n}}$, we know that $\cap_{i=1}^5\mathcal{E}_{k,n}^i \subset\{|\hat \beta_{k,n}|\leq \sqrt{\frac{1+\epsilon}{1-\frac{1}{3m^{1/3}}}}\frac{\sigma_k}{\tilde\sigma_k}\}\subset\{|\hat \beta_{k,n}|\leq 2\frac{\sigma_k}{\tilde\sigma_k}\}$. Therefore, 
    \begin{equation}
        \begin{aligned}
            &\I_{ \cap_{i=1}^5\mathcal{E}_{k,n}^i} \I\left\{\mukt> \mu_k+\tilde \epsilon_k/2\right\}\\
            \leq& \I_{ \cap_{i=1}^5\mathcal{E}_{k,n}^i } \I\left\{|\E_{k,n}[R]-\mu_k|\geq \frac{\tilde\epsilon_k}{6} \text{ or } |\hat \beta_{k,n}||\E_{k,n}[\hat R] -\tilde\mu_k|\geq \frac{\tilde\epsilon_k}{6}\text{ or }|\hat \beta_{k,n}||\E_{k,n}^\all[\hat R]-\tilde\mu_k|\geq \frac{\tilde\epsilon_k}{6}\right\}\\
            \leq&  \I\left\{|\E_{k,n}[R]-\mu_k|\geq \frac{\tilde\epsilon_k}{6}\right\}+\I\left\{2\frac{\sigma_k}{\tilde\sigma_k}|\E_{k,n}[\hat R] -\tilde\mu_k|\geq \frac{\tilde\epsilon_k}{6}\right\}+\I\left\{2\frac{\sigma_k}{\tilde\sigma_k}|\E_{k,n}^\all[\hat R]-\tilde\mu_k|\geq \frac{\tilde\epsilon_k}{6}\right\}
        \end{aligned}
    \end{equation}
    Plug back to \eqref{eqn: batch term 3 first step}, we have 
    \begin{equation}
    \begin{aligned}
    \label{eqn: batch term 3}
        &\E\sum_{t=K}^\infty \I_{\mathcal{A}_{k,t}\cap \mathcal{D}_{k,t}^C\cap \mathcal{H}_{k,t} } \leq  \sum_{n=1}^\infty \E\left[\I_{\cap_{i=1}^5\mathcal{E}_{k,n}^i } \I\left\{\mukt> \mu_k+\tilde \epsilon_k/2\right\} \right]\\
        \leq& \sum_{n=1}^\infty  \lb\P\lb|\E_{k,n}[R]-\mu_k|\geq \frac{\tilde\epsilon_k}{6}\rb+\P\lb2\frac{\sigma_k}{\tilde\sigma_k}|\E_{k,n}[\hat R] -\tilde\mu_k|\geq \frac{\tilde\epsilon_k}{6}\rb+\P\lb2\frac{\sigma_k}{\tilde\sigma_k}|\E_{k,n}^\all[\hat R]-\tilde\mu_k|\geq \frac{\tilde\epsilon_k}{6}\rb\rb \\
        \overset{(1)}{\leq}& C\sum_{n=1}^\infty \lb \exp\lb-\frac{Cmn\tilde\epsilon_k^2}{\sigma_k^2}\rb + \exp\lb-\frac{Cm(n+N_k)\tilde\epsilon_k^2}{\sigma_k^2}\rb\rb\leq C\frac{1}{\exp\lb\frac{Cm\epsilon^2\Delta_k^2}{\sigma_k^2}\rb - 1} \leq C\frac{\sigma_k^2}{m\epsilon^2 \Delta_k^2}\\
    \end{aligned}
    \end{equation}
    where the inequality (1) uses Hoeffding's inequality.

    \textbf{Bound the fourth term of \eqref{eqn: regret decomposition batch}} For the fourth term of \eqref{eqn: regret decomposition batch}, notice that 
    \begin{equation}
        \label{eqn: batch term 4 first step}
        \I_{ \mathcal{A}_{k,t} }\I_{ \mathcal{H}_{k,t}^C}\leq \I_{ \mathcal{A}_{k,t} }\I_{(\cap_{i=1}^5\mathcal{E}_{k,n_{k,t}}^i)^C}+\I_{ \mathcal{A}_{k,t} }\I_{(\cap_{i=1}^5\mathcal{E}_{k^\star,n_{k^\star,t}}^i)^C},
    \end{equation}
 and we will handle two terms separately. For event $\mathcal{E}_{k,n}^1, \mathcal{E}_{k,n}^3, \mathcal{E}_{k,n}^4$, apply Lemma \ref{lem: variance concentration}, we know that 
    \begin{equation}
        \begin{aligned}
        \label{eqn: event E 1}
            \P((\mathcal{E}_{k,n}^1)^C)\leq 4\exp(-Cmn\epsilon^2), \P((\mathcal{E}_{k,n}^3)^C)\leq 4\exp(-Cm^{1/3}n), \P((\mathcal{E}_{k,n}^4)^C)\leq 4\exp(-Cm^{1/3}n).
        \end{aligned}
    \end{equation}
    For event $\mathcal{E}_{k,n}^2$, notice that Lemma \ref{lem: reg residual} implies 
    \begin{equation}
        \hat{\sigma}^2_{\epsilon,k,n} 
            =  \frac{1-\correps^2}{mn  }\sum_{s=1}^{mn} (\epsilon_{k,s} - \E_{k,n} [\epsilon])^2\leq \frac{1}{mn  }\sum_{s=1}^{mn} (\epsilon_{k,s} - \E_{k,n} [\epsilon])^2,
    \end{equation}
    therefore apply Lemma \ref{lem: variance concentration} we have
    \begin{equation}
        \begin{aligned}
        \label{eqn: event E 2}
            &\P((\mathcal{E}_{k,n}^2)^C) \leq \P\lb\frac{|\frac{1}{mn  }\sum_{s=1}^{mn} (\epsilon_{k,s} - \E_{k,n} [\epsilon])^2-(1-\rho_k^2)\sigma_k^2|}{(1-\rho_k^2)\sigma_k^2}\geq \epsilon\rb\leq  4\exp(-Cmn\epsilon^2)
        \end{aligned}
    \end{equation}
    For event $\mathcal{E}_{k,n}^5$, apply Lemma \ref{lem:corr_concentration},
    \begin{equation}
    \label{eqn: event E 5}
        \P((\mathcal{E}_{k,n}^5)^C)=\P\lb|\correps|\geq \frac{1}{3m^{1/3}}\rb \leq 6\exp(-Cm^{1/3}n)
    \end{equation}
    Combining \eqref{eqn: event E 1}, \eqref{eqn: event E 2}, and \eqref{eqn: event E 5}, we have 
    \begin{equation}
    \begin{aligned}
    \label{eqn: batch term 4 second step}
        &\E[\I_{(\cap_{i=1}^5\mathcal{E}_{k,n}^i)^C}]\leq \P((\mathcal{E}_{k,n}^1)^C)+\P((\mathcal{E}_{k,n}^2)^C)+\P((\mathcal{E}_{k,n}^3)^C) + \P((\mathcal{E}_{k,n}^4)^C) + \P((\mathcal{E}_{k,n}^5)^C)\\
        \leq &C(\exp(-Cm^{1/3}n)+\exp(-Cmn\epsilon^2)).
    \end{aligned}
    \end{equation}
    Hence the first term of \eqref{eqn: batch term 4 first step} can be bounded by
    \begin{equation}
    \begin{aligned}
    \label{eqn: batch term 4 term 1}
        &\E\sum_{t=K}^T\I_{ \mathcal{A}_{k,t} } \I_{(\cap_{i=1}^5\mathcal{E}_{k,n_{k,t}}^i)^C} 
        \overset{(1)}{\leq} \E\sum_{n=1}^T \I_{(\cap_{i=1}^5\mathcal{E}_{k,n}^i)^C}
        \overset{(2)}{\leq}  \sum_{n=1}^TC(\exp(-Cm^{1/3}n)+\exp(-Cmn\epsilon^2)) \\
        \leq  &C\lb \frac{1}{m\epsilon^2}+\frac{1}{m^{1/3}}\rb
    \end{aligned}
    \end{equation}
    Here, inequality (1) is from the fact that $\mathcal{A}_{k,t} = \{n_{k,t+1} = n_{k,t}+1\}$ and $\mathcal{E}_{k,t} = \mathcal{E}_{k,t'}$ for any round $t, t'$ such that  $n_{k,t} = n_{k,t'}$;  inequality (2) is from \eqref{eqn: batch term 4 second step}.
    For the second term of \eqref{eqn: batch term 4 first step}, we consider the aggregation over all $k\neq k^\star$
    \begin{equation}
        \begin{aligned}
        \label{eqn: bound variance remainder}
            &\sum_{k\neq k^\star}\Delta_k\E\sum_{t=K}^T\I_{ \mathcal{A}_{k,t} } \I_{(\cap_{i=1}^5\mathcal{E}_{k^\star,n_{k^\star,t}}^i)^C} \leq \Delta_{\max}\sum_{t=K}^T\E[\I_{(\cap_{i=1}^5\mathcal{E}_{k^\star,n_{k^\star,t}}^i)^C}] \ \leq \Delta_{\max}\sum_{t=K}^T\sum_{n=1}^T\E[\I_{(\cap_{i=1}^5\mathcal{E}_{k^\star,n}^i)^C}]\\
            \leq & \Delta_{\max}\sum_{t=K}^T\sum_{n=1}^TC(\exp(-Cm^{1/3}n)+\exp(-Cmn\epsilon^2))\\
            \leq & C\Delta_{\max}T\lb\frac{\exp(-Cm^{1/3})}{1-\exp(-Cm^{1/3})} + \frac{\exp(-Cm\epsilon^2)}{1-\exp(-Cm\epsilon^2)}\rb
        \end{aligned}
    \end{equation}
    
    \textbf{Combine bounds} Combining bounds on all four terms together, we obtain that 
    \begin{equation}
    \label{eqn: num pulls bound batch}
    \begin{aligned}
        &\E[n_{k,T}] =\E\sum_{t=1}^T\I_{\mathcal{A}_{k,t}}\\ \leq & \frac{(1+\frac{1}{m^{1/3}})^2(1+\epsilon)}{(1-\epsilon)^2}\lb\sqrt{\delta_k}+\sqrt{1-\rho_k^2}\rb^2 \frac{2\sigma_k^2\ln T}{m\Delta_k^2} + \frac{C\sigma_{k}^2}{m \Delta_k^2\epsilon^2}\\
        &+ C\lb\frac{1}{m\epsilon^2}+\frac{1}{m^{1/3}}\rb+C\sqrt{\ln T}+ 1 + \E\sum_{t=K}^T\I_{ \mathcal{A}_{k,t} } \I_{(\cap_{i=1}^5\mathcal{E}_{k^\star,n_{k^\star,t}}^i)^C}\\
        \overset{(1)}{\leq } &\frac{1+\epsilon}{(1-\epsilon)^2}\lb\sqrt{\delta_k}+\sqrt{1-\rho_k^2}\rb^2 \frac{2\sigma_k^2\ln T}{m\Delta_k^2} + \frac{C\sigma_k^2}{m\Delta_k^2\epsilon^2} \\
        &+\frac{C}{(\ln T)^4\epsilon^2}+ C\sqrt{\ln T} + \E\sum_{t=K}^T\I_{ \mathcal{A}_{k,t} } \I_{(\cap_{i=1}^5\mathcal{E}_{k^\star,n_{k^\star,t}}^i)^C} \\
        \overset{(2)}{\leq } &\lb\sqrt{\delta_k}+\sqrt{1-\rho_k^2}\rb^2 \frac{2\sigma_k^2\ln T}{m\Delta_k^2} + \frac{C\epsilon\sigma_k^2\ln T}{m\Delta_k^2}+ \frac{C\sigma_k^2}{m\Delta_k^2\epsilon^2} \\
        &+\frac{C}{(\ln T)^4\epsilon^2}+ C\sqrt{\ln T} + \E\sum_{t=K}^T\I_{ \mathcal{A}_{k,t} } \I_{(\cap_{i=1}^5\mathcal{E}_{k^\star,n_{k^\star,t}}^i)^C}
    \end{aligned}
    \end{equation}
    here the inequality (1) uses the fact that $m=\Omega((\ln T)^4),\lb\sqrt{\delta_k}+\sqrt{1-\rho_k^2}\rb^2\leq 4 $, thus expansion $\frac{1}{m^{1/3}}\frac{1+\epsilon}{(1-\epsilon)^2} \lb\sqrt{\delta_k}+\sqrt{1-\rho_k^2}\rb^2 \frac{2\sigma_k^2\ln T}{m\Delta_k^2}\leq C\frac{\sigma_k^2}{m\Delta_k^2}$ is negligible, and we use the fact that $\frac{1}{m\epsilon^2}+\frac{1}{m^{1/3}}\leq C\lb\frac{1}{(\ln T)^4\epsilon^2}+\frac{1}{(\ln T)^{4/3}}\rb$;  the inequality (2) uses the algebra fact that $\frac{1+x}{(1-x)^2}\leq 1+10x, \forall x\in (0, 1/2)$. 
    Now, we obtain the desired regret bound using \eqref{eqn: bound variance remainder}
    \begin{equation}
        \begin{aligned}
            &\frac{1}{m}\E[\reg]=\sum_{k\neq k^\star}\E[n_{k,T}]\Delta_k\\
            \leq& \sum_{k\neq k^\star}\lb\lb\sqrt{\delta_k}+\sqrt{1-\rho_k^2}\rb^2 \frac{2\sigma_k^2\ln T}{m\Delta_k^2} + \frac{C\sigma_k^2}{m\Delta_k^2\epsilon^2} + \frac{C\epsilon\sigma_k^2\ln T}{m\Delta_k^2} 
    +\frac{C}{(\ln T)^4\epsilon^2}+ C\sqrt{\ln T} \rb\Delta_k\\
        &+ C\Delta_{\max}T\lb\frac{\exp(-Cm^{1/3})}{1-\exp(-Cm^{1/3})} + \frac{\exp(-Cm\epsilon^2)}{1-\exp(-Cm\epsilon^2)}\rb
        \end{aligned}
    \end{equation}
\end{proof}
Now we can derive Theorem \ref{thm: regret bound non Gaussian} as a direct corollary of Theorem \ref{thm: MLA-UCB non Gaussian}.
\begin{proof}[Proof of Theorem \ref{thm: regret bound non Gaussian}]
    Take $\epsilon = \frac{1}{2}(\ln T)^{-1/3}$, then \eqref{eqn: regret batch eps} implies 
    \begin{equation}
    \begin{aligned}
        \frac{1}{m}\E[\reg]\leq& \sum_{k\neq k^\star}\lb\lb\sqrt{\delta_k}+\sqrt{1-\rho_k^2}\rb^2 \frac{2\sigma_k^2\ln T}{m\Delta_k^2} +\frac{C\sigma_k^2(\ln T)^{2/3}}{m\Delta_k^2}+ C\sqrt{\ln T}\rb\Delta_k\\
        +& C\Delta_{\max}T\frac{\exp(-Cm^{1/3})}{1-\exp(-Cm^{1/3})} 
    \end{aligned}
    \end{equation}
    Under Assumption \ref{asm: m and T}, the term $m^{1/3}\geq (\ln T)^{4/3}$, thus $T\exp(-Cm^{1/3}) \leq  \exp(\ln T-C(\ln T)^{4/3}) = O(1)$, and $1-\exp(-Cm^{1/3}) = \Omega(1)$, thus $C\Delta_{\max}T\frac{\exp(-Cm^{1/3})}{1-\exp(-Cm^{1/3})} \leq C \Delta_{\max}\leq C\sum_{k\neq k^\star} \Delta_k$. 
    Moreover, $\lb\sqrt{\delta_k}+\sqrt{1-\rho_k^2}\rb^2  \leq (1-\rho_k^2)+3\sqrt{\delta_k}$ for any $\delta_k\leq 1$. As a result, if we take $\delta_k = (\ln T)^{-2/3}$, then the sample size condition \eqref{eqn: sample size condition batch} is equivalent to $N_k = \Omega(\frac{\sigma_k^2}{m\Delta_k^2}(\ln T )^{5/3})$, and we obtain the desired regret bound 
    \begin{equation}
        \frac{1}{m}\E[\reg]\leq \sum_{k\neq k^\star}\lb \frac{2(1-\rho_k^2)\sigma_k^2\ln T}{m\Delta_k^2} +\frac{C\sigma_k^2(\ln T)^{2/3}}{m\Delta_k^2}\rb\Delta_k.
    \end{equation}
\end{proof}
\end{document}